\documentclass[12pt,a4paper,reqno]{amsart}
\usepackage[utf8]{inputenc}
\usepackage[T1]{fontenc}
\usepackage[english]{babel}
\usepackage{amsfonts,amsmath, amssymb,latexsym,times,mathrsfs,bm}
\usepackage{comment}
\usepackage{ae,aecompl}
\usepackage{tikz-cd}

\usepackage[colorlinks=true]{hyperref}
\hypersetup{ pdfstartview=FitH }

\numberwithin{equation}{section}
\def\TT{{\mathbb{T}}}

 \def\HH{{\mathbb{H}}}
\def\ZZ{{\mathbb{Z}}}  \def\SS{{\mathbb{S}}}   \def\CC{{\mathbb{C}}}
\def\RR{{\mathbb{R}}}  \def\TT{{\mathbb{T}}}

\def\cG{{\mathcal{G}}}
\def\cI{{\mathcal{I}}}
\def\cB{{\mathcal{B}}}

\def\cJ{{\mathcal{J}}}

\def\scrE{{\mathscr{E}}}
\def\scrT{{\mathscr{T}}}
\def\scrW{{\mathscr{W}}}

\def\scrF{{\mathscr{F}}}
\def\scrK{{\mathscr{K}}}
\def\scrR{{\mathscr{R}}}
\def\scrD{{\mathscr{D}}}
\def\scrS{{\mathscr{S}}}

\def\cL{{\mathcal{L}}}

\def\scrM{{\mathscr{M}}}

\def\scrU{{\mathscr{U}}}
\def\scrV{{\mathscr{V}}}
\def\scrL{{\mathscr{L}}}

\def\cK{{\mathcal{K}}}
\def\cV{{\mathcal{V}}}

\def\cU{{\mathcal{U}}}

\def\vecV{\overrightarrow V}

\def\Map{\mathrm{Map}}

\renewcommand{\epsilon}{\varepsilon}
\newcommand{\ip}[1]{\langle #1 \rangle}
\newcommand{\ipp}[1]{\langle \!  \langle #1 \rangle\! \rangle}

\newcommand{\LieTT}{\mathfrak{t}}

\newcommand{\vol}{\mathrm{vol}}

\newcommand{\id}{\mathrm{id}}
\newcommand{\Diff}{\mathrm{Diff}}
\newcommand{\Symp}{\mathrm{Symp}}

\newcommand{\End}{\mathrm{End}}

\newcommand{\del}{\partial}

\renewcommand{\Im}{\mathrm{Im}\;}

\newcommand{\mub}{{\bm{\mu}}}

\newcommand{\resp}{\emph{resp.}{} }

\newtheorem{theointro}{Theorem}

\newtheorem{conj}[theointro]{Conjecture}

\newtheorem{lemma}[subsubsection]{Lemma}
\newtheorem{prop}[subsubsection]{Proposition}
\newtheorem{cor}[subsubsection]{Corollary}
\newtheorem{theo}[subsubsection]{Theorem}

\newtheorem{dfn}[subsubsection]{Definition}
\newtheorem{questions}[subsubsection]{Questions}

\theoremstyle{definition}

\newtheorem{rmk}[subsubsection]{Remark}

\newtheorem*{rmk*}{Remark}
\newtheorem*{rmks*}{Remarks}
\newtheorem{rmks}[subsubsection]{Remarks}

\thanks{Supported by  IRL CRM-CNRS 3457, UQAM, CIRGET, UMR6629,
ANR-21-CE40-0017, CHL 11-LABX-0020-01}

\title[Symplectic maps and HK geometry]
{Symplectic maps and hyperKähler moment map geometry}

\author{Yann Rollin}
\address{Yann Rollin, Laboratoire Jean Leray, Universit\'e de Nantes}
\email{yann.rollin@univ-nantes.fr}

\begin{document}
\begin{abstract}
We obtain a correspondence between the group of 
	symplectic diffeomorphisms of a 
	$4$-dimensional real torus and the vanishing locus
	of a certain hyperKähler moment map. This observation gives rise to
	a new flow, called the modified moment map flow. The construction
	can be   adapted to the polyhedral setting, for which we  prove a
	Duistermaat type theorem. 
	This paper lays out the ground work for some effective 
	polyhedral symplectic
	geometry and  for a potential Morse-Bott theory, with applications
	to the topology of the space of symplectic maps of the $4$-torus.
\end{abstract}

{\Huge \sc \bf\maketitle}

\section{Introduction}
\label{sec:intro}
\subsection{Motivations}
Symplectic geometry is the natural mathematical framework for
  \emph{Hamiltonian mechanics}. The group of symplectic
  diffeomorphisms is a cornerstone of the theory, but its topology remains
  \emph{unknow}, to a large extent~\cite[\S$10.4$]{MDS}. 
  The interest for the topology of these
  groups
  has stirred dazzling developments in the
  field of symplectic topology.
  In particular, the
\emph{Arnold  
conjecture},  about 
the number of \emph{fixed points} of a  symplectic diffeomorphism, can
be understood as a natural extension of the
\emph{Poincaré-Birkhoff fixed point theorem}. 
Beautiful techniques were introduced   in the field of
symplectic topology for solving the Arnold conjecture, 
intimately related to the existence  of 
\emph{pseudoholomorphic curves} and  \emph{Floer homology}.
More generally, Gromov showed that certain symplectic properties
are \emph{flexible},
thanks some \emph{h-principle} and \emph{convex integration}
methods~\cite{Gro85}, whereas others are \emph{rigid}, due to
the existence of pseudoholomorphic curves~\cite{Gro86}. The \emph{Hofer
geometry} of the symplectic diffeomorphism group  was subsequently developped
and gave new insights about such rigidity properties~\cite{LMD}.
 Although the local structure of the group of symplectic diffeomorphism is completely
 described as an infinite dimensional Lie group~\cite{Ban},  
 its  large scale topology  remains deeply mysterious,
 in spite of all the efforts of the community, 
   except for some special cases
(cf.~\cite{MDA}, for instance). In particular, nothing 
more than the local structure is known about the topology of
the symplectic diffeomorphism group of a
\emph{$4$-dimensional  torus}.

One of the famous result about the  symplectic diffeomorphism groups 
is due to Eliasberg and Gromov~\cite{Gro86,Eli87}, who proved a striking theorem: 
although it is possible to approximate any smooth diffeomorphism by
 \emph{volume preserving diffeomorphisms}, in the $C^0$-sense,
it is not possible to do so with  symplectic diffeomorphisms.
More precisely, the Eliashberg-Gromov theorem states that every 
smooth diffeomorphism, which is a $C^0$-limit of a sequence of symplectic
diffeomorphism, must be symplectic  as well. 
This rigidity result is a strong incentive to relax the regularity
of \emph{symplectic maps} and gave rise to the field of \emph{$C^0$-symplectic 
geometry} which is, in some sense, an exploration of the 
closure of the group of symplectic diffeomorphisms. 

Our paper  deals with the case of a
$4$-dimensional real torus $M$, endowed with a
\emph{canonical  symplectic form} $\omega_M$ and a conjugate
\emph{hyperKähler structure}.
We introduce a new flow, called the \emph{modified moment map flow},
associated to a \emph{hyperKähler moment map geometry} on
a \emph{moduli space} of differential forms on $M$ and a triply \emph{Hamiltonian gauge group
action} on the moduli space.
The \emph{fixed point locus} of the modified moment map flow is closely
related to the group $\Symp(M,\omega_M)$ of symplectic 
diffeomorphisms of $M$. The flow seems to be an interesting
tool to investigate the topology of  $\Symp(M,\omega_M)$, 
in the spirit of  \emph{Morse-Bott theory}.
Some essential properties of the modified moment map  flow are proved in
this paper. They are, in
some sense,  the first
steps toward  an infinite dimensional Morse-Bott theory, that would lead 
to interesting topological invariants of  the  group  of symplectic
diffeomorphisms.

We show that the hyperKähler moment map construction  can be
readily adapted to
 \emph{polyhedral geometry}. We consider the space of
polyhedral maps of the torus~$M$, with respect to a prescribed
triangulation~$\scrT$.
The space of polyhedral symplectic maps $\Symp(M,\omega_M,\scrT)$
can be regarded as a finite
dimensional approximation of the space $\Symp^{PL}(M,\omega_M)$ 
of piecewise linear symplectic maps
of the torus (cf. \S\ref{rmk:dl} for more details).
We show that all our constructions, in the smooth setting,
have analogues in the polyhedral
setting and we define a \emph{polyhedral modified moment map
flow}. This flow is an \emph{ordinary differential equation} and much stronger
results can be proved. In the spirit of the Duistermaat theorem, we
construct a continuous retraction thanks
to the polyhedral modified moment map flow at Theorem~\ref{theo:duistermaat} . 
This
result  is
a strong incentive to carry out  numerical computer
experiments and produce \emph{effective
examples} of polyhedral symplectic maps of the $4$-torus~\cite{JR2}.

Very little is known about piecewise linear symplectic geometry,
although some seminal works have been carried out by Gratza \cite{Gra},
Bertelson-Distexe \cite{Dis19,BD},
Jauberteau-Rollin-Tapie \cite{JRT,R} and Etourneau~\cite{Eto}.
In each case, the
authors try to extend some 
elementary properties  of smooth symplectic geometry to 
the piecewise linear sympletic category. 
The task turns out to be unexepectedly
hard, which could be a manifestation of symplectic rigidity.
For
example~\cite{R} shows that every smoothly immersed torus in $\RR^4$ can
be approximated, 
in the $C^0$ sense, 
by  piecewise linear immersed Lagrangian tori. 
This
result can be understood as
an extension of the  
Gromov-Lees Lagrangian flexibility theorem~\cite{Gro86,Lee76}.

A tentative definition for piecewise linear symplectic manifold can be
found in \cite{BD}, but most of the basic
properies of smooth symplectic geometry are lost
in the piecewise linear case. Indeed, the classical \emph{local Darboux theorem} 
is a conjecture for piecewise linear symplectic manifolds.
The fact that the group smooth symplectic diffeomorphism 
is \emph{locally arc connected} is classical~\cite{Ban,MDS}. 
However the same statement  is a conjecture for
the space of piecewise linear symplectic maps.
 There is even worse: the only 
 constructions of polyhedral symplectic maps of the
 $4$-dimensional quotient torus  $M$ are the trivial examples of affine maps of $M$. 
 It is also possible
 to consider 
 some product constructions induced by area preserving polyhedral maps of the
 $2$-torus, but no general construction or deformation theory is available, like
 in the smooth setting.

A natural continuation of this work   
is to develop a Morse-Bott
 theory following \cite{Kir},
 associated to the polyhedral modified moment map flow
and
designed to obtain topological 
invariants for the space of polyhedral
symplectic maps of the $4$-torus 
$\Symp(M,\omega_M,\scrT)$ and, ultimately, 
for the space $\Symp^{PL}(M,\omega_M)$ of piecewise linear
symplectic maps.

\subsection{Notations}
A hyperKähler torus $M$ of real dimension $4$ is a quotient $V/\Gamma$
where
\begin{itemize}
	\item $V$ is an
affine space of real dimension~$4$.
		\item $\Gamma$ is a lattice of the underlying vector space~$\vecV$.
		\item The
vector space $\vecV$ is identified via a given linear isomorphism with the space of quaternions $\HH$.
\end{itemize}
The \emph{quaternionic multiplication} by complex numbers \emph{on the
right} endows $\HH$, hence $\vecV$, with a structure of complex vector space.
In addition, the vector space $\vecV$ carries a Euclidean inner product
$g_V$, deduced from the canonical inner product on $\HH$ and the linear isomorphism.
Hence the affine space
$V$ is  endowed with a flat Riemannian metric $g_V$ and a
compatible integrable almost complex
structure denoted $i$.
In conclusion, we have a flat Kähler structure on the affine
space
$(V,g_V,i,\omega_V)$, where $\omega_V$ is the Kähler form. The  Kähler structure
on $V$ descends to the quotient $M$ as a flat Kähler structure
$$
(M,g_M,i,\omega_M)
$$
called the \emph{canonical Kähler structure}.
We are now interested in the group of symplectic diffeomorphisms of the
torus denoted
$$
\Symp(M,\omega_M).
$$
There are many other $g_M$-compatible complex structures
on~$M$: the quaternionic multiplication by $i, j$ and $k$, \emph{on
the left} induce three $g_M$-compatible integrable almost complex structures $I$, $J$ and
$K$ on $M$, which define a \emph{conjugate hyperKähler structure} 
with Kähler forms
$\hat\omega_I$, $\hat\omega_J$ and $\hat\omega_K$. 
By definition the hyperKähler forms $\hat\omega_\bullet$  are compatible with the reverse
orientation of $\omega_M$.

\subsection{Statement of results in the smooth setting}
\label{sec:statsmooth}
The gauge group 
$$\TT=C^\infty(M,S^1)$$
acts by complex multiplication
on the moduli space 
$$\scrF=\Omega^1(M,\vecV)$$
of $\vecV$-valued differential $1$-forms
on~$M$. We show that the moduli space $\scrF$~carries a natural formal 
 hyperKähler structure, 
together with a triply Hamiltonian action 
of the gauge group~$\TT$.

\begin{theointro}
	\label{theo:A}
	We consider a  hyperKähler quotient torus $M=V/\Gamma$
	of real dimension $4$, endowed with its canonical symplectic form $\omega_M$ 
	and conjugate hyperKähler structure, given by the almost complex structures $I,
	J$ and $K$, with
	 associated Kähler forms $\hat\omega_\bullet$ for $\bullet
	=I,J,K$.
	The moduli space $\scrF=\Omega^1(M,\vecV)$  of $\vecV$-valued differential $1$-forms 
	on $M$, carries a canonical  hyperKähler structure
 $$
 (\scrF,\cG,\cI,\cJ,\cK),
 $$
	where $\cG$ is a Euclidean inner product compatible with the almost complex
	structures $\cI$, $\cJ$ and $\cK$ on $\scrF$. 
	The associated
	Kähler forms are denoted repectively $\Omega_I$, $\Omega_J$ and
	$\Omega_K$.

The moduli space $\scrF$ is endowed with a canonical action of the gauge group $\TT=C^\infty(M,S^1)$, 
	 by complex
	multiplication. 
	The hyperKähler structure of $\scrF$ is invariant under the
	$\TT$-action. Furthermore, the gauge group action is Hamiltonian with respect to $\Omega_I, \Omega_J$
	and $\Omega_K$ and the  corresponding moment maps 
$$
	 \mu_I,\mu_J,\mu_K:\scrF\to \LieTT\simeq C^\infty(M,\RR),
$$
	  are given by the formulas
	 $$
	 \mu_\bullet(F)=-\frac{(F^*\omega_V)\wedge\hat\omega_\bullet}{\vol_M},
	 $$
	 for $\bullet = I, J, K$, where $\vol_M = \frac{\omega_M^2}2$ and
	 $$(F^*\omega_V)(\eta_1,\eta_2)=\omega_V(F(\eta_1),F(\eta_2)).$$
 \end{theointro}
It is convenient to gather the various moment maps $\mu_\bullet$  into a
single  map
\begin{equation}
	\label{eq:hkmm}
	\mub:\scrF \to \LieTT^3\\
\end{equation}
given by $ \mub=(\mu_I,\mu_J,\mu_K)$,
called the \emph{hyperKähler moment map}. 

A smooth map $f:M\to M$ defines a \emph{tangent map}  $f_*:TM\to TM$.
Using the triviality of the tangent bundle of a quotient torus, the tangent map $f_*$ can be understood as a $\vecV$-valued
differential form, denoted
$\scrD f\in \scrF$
and closely related to the differential operator $d$ in local coordinates
(cf.~\S\ref{sec:diffdiff} for a more detailed definition).
Thus $\scrD$ defines a linear differential operator  between {the moduli spaces}
$$
\scrD:\scrM\to\scrF,
$$
where 
$$
\scrM = \{f:M\to M, f\in C^\infty\}
$$
is the  space of smooth maps of
$M$. 
Our next theorem provides an interpretations of symplectic diffeomorphisms
of $M$ as zeroes of the moment map~$\mub$:
\begin{theointro}
	\label{theo:eqmu}
	Let $M$ be a hyperKähler quotient torus of real dimension $4$, with
	canonical symplectic form $\omega_M$. Let $f:M\to M$ be a smooth
	map of the torus satisfying the cohomological condition
	$f^*[\omega_M]=[\omega_M]$. Then the following properties are
	equivalent:
	\begin{enumerate}
		\item $f^*\omega_M=\omega_M$.
		\item $\mub(\scrD f) = 0$, where $\mub$ is the hyperKähler moment
			map.
	\end{enumerate}
\end{theointro}

Considering the
\emph{downward gradient flow} of the
functional
\begin{align*}
	\phi:\scrF &\to\RR\\
	F&\mapsto \frac 12 \|\mub(F)\|_{L^2}^2
\end{align*}
seems most natural when contemplating
Theorem~\ref{theo:A} and Theorem~\ref{theo:eqmu}. 
The functional $\phi$ is to be
treated  as some type of Morse-Bott function on the moduli space $\scrF$ 
and the corresponding downward gradient flow is
designed to obtain some topological information about the critical locus
of $\phi$.

The original groundbreaking idea, 
that the norm squared $\phi$ of a moment map should behave like a
Morse-Bott function, is due to Atiyah-Bott. 
They applied an  equivariant Morse theory 
to the Yang-Mills functional
and calculated the Betti numbers of the moduli space of semistable bundles
over a Riemann surface \cite{AB}.
Then, Donaldson relied on the
Yang-Mills flow 
to give a purely gauge theoretic proof 
of the Narasimhan-Seshadri theorem~\cite{D}. 

In the finite dimensional Kähler setting, Kirwan showed that
the function $\phi$ can be
treated as a Morse-Bott function~\cite{Kir} and induces 
corresponding Morse-Bott  inequalities~\cite[\S
6]{KFM}. The crucial property, that the moment map flow provides 
retractions, is attributed
to an observation of Duistermaat in~\cite[footnote, p.167]{KFM}, but a nice
self-contained proof can be found in \cite{L05}.

However, the standard  moment map flow on $\scrF$  is not
interesting in our case, as far as we are concerned with the
topology of $\Symp(M,\omega_M)$. Indeed, the action of the
gauge group $\TT$ does not preserve the subspace $\Im\scrD$
 and neither does the gradient flow of~$\phi$. Thus,
such a flow does not capture the topology of 
$\scrM$.
We get around
this issue by introducing a 
variant of the moment map flow: given a cohomology
class $\alpha\in H^1(M,\vecV)\setminus 0$, the subspace 
of \emph{closed forms}  with cohomology
class contained in  the line $\RR\alpha$ is denoted 
$$
 \scrF_{\alpha}=\{F\in\scrF, dF=0 \mbox{ and } [F] \in\RR \alpha\}.
$$
We consider the downward gradient flow of the restricted functional
$\phi:\scrF_\alpha\to\RR$,  called the \emph{modified moment map flow} and
we
prove the following basic properties:
\begin{theointro}
	\label{theo:C}
The fixed point locus of the modified moment map flow on $\scrF_\alpha$, in
	other words the critical locus of $\phi:\scrF_\alpha\to\RR$, agrees
	with the vanishing locus of $\phi:\scrF_\alpha\to\RR$.
	Furthermore, 
	\begin{enumerate}
		\item the modified moment map flow has the short time existence property and
		\item the $L^2$-norm is non increasing along the flow.
	\end{enumerate}
\end{theointro}
\begin{rmk}
	The  \emph{short time existence property} hides some
	technical aspects of the problem. The exact statement involves a
	completed version the
	moduli space  $\scrF$ with respect to some Hölder norm. The
	technical version of this result is stated at Theorem~\ref{theo:cauchy}.
\end{rmk}
\begin{rmk}
The restriction of the functional $\phi$ to the vector space $\scrF_\alpha$ may seem
	somewhat arbitrary. At first glance, it makes sense
	to restrict $\phi$ to other subspaces. For example, the
	restriction of $\phi$ to the cohomology class~$\alpha$, understood as an
	affine subspace of $\scrF$, could be considered as well. 
	It turns out that the $L^2$-decay property of Theorem~\ref{theo:C}
	does not hold in this case. 
 In fact, the decay property extends to the polyhedral setting introduced at
	\S\ref{sec:polmap}. This feature of the polyhedral moment map flow  is a crucial
	ingredient for the construction of a retraction
	proved at Theorem~\ref{theo:duistermaat}, which  motivates our point of view.
\end{rmk}

	We say that a cohomology class $\alpha \in H^1(M,\vecV)$ is
\emph{symplectic}, if there exists a smooth map $f:M\to M$ such that
	$f^*[\omega_M]= [\omega_M]$ and $\alpha =[\scrD f]$.
The component $\scrM_\alpha$
of $\scrM$ is defined as  the subspace of maps~$f$ with cohomology class $[\scrD f]=\alpha$. Similarly, we denote
$\Symp_\alpha(M,\omega_M)$ the subspace of $\Symp(M,\omega_M)$ 
contained in $\scrM_\alpha$. 
By Proposition~\ref{prop:inj}, the map 
$$
\scrD:\Symp_\alpha(M,\omega_M)\to \scrF_\alpha
$$
is injective up to the action of $\vecV$ by translations. 
Furthermore,  the cone on
its image agrees with the vanishing locus of $\phi$ in $\scrF_\alpha\setminus 0$ by
Corollary~\ref{cor:nonproper}.
In an ideal situation, we would expect some Duistermaat type theorem, where
the modified moment map flow produces a
continuous retraction by deformation of $\scrF_\alpha$ onto the vanishing
locus of $\phi$.
This would be an important step to investigate the topology of
$\Symp_\alpha(M,\omega_M)$.
 However, we are facing several
issues, raised in the next questions:
\begin{questions}
	\label{question:open}
	\begin{enumerate}
	\item Does the modified moment map flow enjoy a long time existence
		property?\label{q:i1}
	\item Are the flow lines convergent?\label{q:i2}
	\item Are there some non trivial  flow lines  with
		limit $0\in\scrF_\alpha$?\label{q:i3}
	\end{enumerate}
\end{questions}
 In fact we have little
prospect to answer questions  \eqref{q:i1} and
\eqref{q:i2}, since the modified moment map flow does not seem to have any nice regularizing
properties, like parabolic flows for instance. Nevertheless,  the answer to these questions 
is settled in the context
of polyhedral geometry, where we obtain the best possible result 
at Theorem~\ref{theo:duistermaat}. 

Question~\eqref{q:i3} is more conceptual. Indeed
$F=0$ is the differential of a constant map of $M$, which is definitely
not symplectic.  
This singular phenomenon,  if it ever occurs, is an obstruction for
pushing homotopies of $\scrM_\alpha$ into
$\Symp_\alpha(M,\omega_M)$  via the modified moment map,
as explained in
details at \S\ref{sec:motivret}.
This motivates the
introduction at \S\ref{sec:renflow} of some 
classical tools of \emph{blowup analysis}, devoted to 
study the flow lines converging toward $F=0$.  In particular
we define the \emph{real blowup} of
the moduli space, 
the \emph{soliton equation} and 
the \emph{renormalized flow} along the unit sphere
$\SS_\alpha\subset \scrF_\alpha$. 
We show that the renormalized flow, 
which is the downward gradient flow of the restricted
functional $\phi:\SS_\alpha\to\RR$, 
can be interpreted as a blow up of
the modified moment map flow at $F=0$.
The space of solitons of $\scrF_\alpha$ is a cone, which contains the vanishing set of
$\phi$, called the space of \emph{non proper solitons}. The remaining 
solitons are called  the \emph{proper solitons} and it turns out that each proper soliton 
defines a flowline of the modified moment map
flow converging toward $F=0$, as proved by Lemma~\ref{lemma:dicho}. 
The  space of solitons of $\scrF_\alpha$  admits a
partition
$$
\{0\}\sqcup\scrS_{\alpha,p}\sqcup\scrS_{\alpha,np}
$$
where 
$\scrS_{\alpha,p}$ is the space of proper solitons and 
 $\scrS_{\alpha,np}$ is the space of non proper solitons in
 $\scrF_\alpha\setminus 0$. 
By Corollary~\ref{cor:nonproper},  
$\scrS_{\alpha,np}$  is spanned by  the image of
$\Symp_\alpha(M,\omega_M)$ by $\scrD$ and the action of $\RR^*$. 
From these formal considerations, it seems most natural to regard
$\phi:\SS_\alpha\to\RR$ as a Morse-Bott function. An adapted Morse-Bott
theory would provide a dynamical interaction between the  critical
components of $\phi$, in particular between the space of proper and non proper
solitons  of $\SS_\alpha$. It seems sensible to expect that
topological invariants could be associated to the group of symplectic
diffeomorphisms 
via  a  Morse-Bott cohomological construction for $\phi$.
In this paper,  we show that the renormalized flow has  the following basic  properties:
\begin{theointro}
	\label{theo:D}
	The renormalized flow along $\SS_\alpha$ has the short time existence
	property. The fixed points of the renormalized flow are the
	solitons contained in $\SS_\alpha$ and the minimum of
	$\phi:\SS_\alpha\to\RR$ is the
	critical locus that consists of non proper solitons in the sphere.
\end{theointro}
However, many essential  questions are opened at this stage:
\begin{questions}
	\label{q:proper}
Is the space of proper solitons of $\scrF_\alpha$  empty? Are there some
	proper solitons with vanishing cohomology class? 
Do the components of the space of solitons have a structure of manifold, in
some sense? 
\end{questions}
The first question is directly related to the existence of solutions of
the modified moment map flow converging toward zero. We show at  Theorem~\ref{theo:polflow} that
the answer to this question is always positive in the case
of polyhedral
geometry.

We  also answer partially the last question in the smooth setting, 
by proving at Theorem~\ref{theo:connectedsol} that $\scrS_{\alpha,p}$
and $\scrS_{\alpha,np}$ are \emph{topologically separated}, with respect to a suitable Hölder
topology. 
In this introduction, we state a non technical version of
this result, as some type of \emph{rigidity theorem}:  trying to
 deform a non proper soliton into a proper soliton may seem like a natural
 idea,
  but
 we prove that it is not possible.
\begin{theointro}
	\label{theo:connected}
	Let $\alpha$ be a symplectic cohomology class. For every map
	 $\psi:X\to
	\scrF_\alpha$ from a topological
	connected space $X$,
	continuous with respect to the Whitney topology, such that
	$$
	\psi(X)\subset
			\scrS_{\alpha,p}\cup\scrS_{\alpha,np} \quad \mbox{
				and } \quad
			\psi(X)\cap \scrS_{\alpha,np}\neq\emptyset,
			$$
			we have 
			$$\psi(X)\subset \scrS_{\alpha,np}.$$
\end{theointro}

\subsection{Statement of results in the polyhedral setting}
\label{sec:respol}
All the constructions and results presented in the smooth setting at
\S\ref{sec:statsmooth}  were initially motivated by
the \emph{polyhedral symplectic framework}, with the \emph{piecewise linear symplectic
geometry} in mind. 
Working in the smooth setting helped us unveil the relevant objects 
in the \emph{discrete} and \emph{finite dimensional} world of polyhedral
geometry.
We  quickly introduce the polyhedral concepts necessary for stating our
results in the following
paragraphs, but the reader may refer to  \S\ref{sec:polmap} for more details.

A \emph{Euclidean triangulation}  of the quotient torus $M$ is a
triangulation
$\scrT=(M,\scrK,\ell)$, in the usual sense, with the additional condition
that the restriction of the homeomorphism $\ell:|\scrK|\to M$ is affine along each simplex of
the simplicial complex~$\scrK$.
The \emph{space of
polyhedral maps} with respect to $\scrT$, denoted $\scrM(\scrT)$, is the
space of continuous maps, affine along each simplex to the triangulation.

The restriction of a  polyhedral map $f|_\sigma$ to a simplex of the
triangulation is differentiable, although the map $f$ is generally not differentiable
everywhere on $M$.  In particular the pull back of the differential form $f^*\omega_M$ is well defined along each simplex
of the triangulation.
Accordingly, a polyhedral map such that the pullback of $f^*\omega_M$
agrees with the pullback of $\omega_M$ along each simplex is
called a \emph{polyhedral symplectic map} and the space of polyhedral
symplectic maps is denoted $\Symp(M,\omega_M,\scrT)$.

By analogy with the smooth setting, we introduce the space $\scrF(\scrT)$
as the space of families 
 $F=(F_\sigma)_{\sigma\in\scrK_4}$, where 
			 $\sigma$
belongs to the  set of $4$-simplices $\scrK_4$ and
				$F_\sigma$ is a constant
$\vecV$-valued differential $1$-form along~$\sigma$.
Similarly to the smooth case, there is a well defined linear operator associated
to the differentiation of polyhedral maps
$$
\scrD:\scrM(\scrT)\to\scrF(\scrT).
$$
In the case where all the pullbacks of the family
$F=(F_\sigma)\in\scrF(\scrT)$ agree along common faces,
we say that $F$ is a \emph{Whitney form}. The subspace of Whitney forms in
$\scrF(\scrT)$ is denoted $\scrF_c(\scrT)$.  In particular, the operator
$\scrD$ takes values in the space of Whitney forms.
A Whitney forms in
$\scrF_c(\scrT)$
 defines a  \emph{Whitney cohomology class} and we can introduce the subspace
$\scrF_\alpha(\scrT)$ of Whitney forms with cohomology class in the line
$\RR\alpha$ for any $\alpha\in H^1(M,\vecV)\setminus 0$.
The gauge group $\TT(\scrT)$ is the space of functions $\lambda
:\scrK_4\to S^1$ and its Lie algebra $\LieTT(\scrT)$ is the space of
functions $\zeta:\scrK_4\to\RR$. The  group $\TT$ is a real torus with dimension the number
of facets of $\scrK$ and the group  $\TT(\scrT)$ acts on $\scrF(\scrT)$  by complex multiplication,
as in the smooth case.
Finally, all the constructions of hyperKähler structures
and moment maps in the smooth setting hold formally for $\scrF(\scrT)$ acted
on by $\TT(\scrT)$ and we obtain the following result:
\begin{theointro}
		\label{theo:AP}
	The statements of Theorem~\ref{theo:A} and Theorem~\ref{theo:eqmu}
	hold in the polyhedral setting.
\end{theointro}

Continuing the analogy with the smooth case, 
the energy functional 
$$\phi:\scrF(\scrT)\to\RR
$$ 
is given similarly 
by $\phi(F)=\|\mub(F)\|^2_{L^2}$, where
$\mub:\scrF(\scrT)\to\LieTT(\scrT)^3$ is the polyhedral hyperKähler moment
map. The
\emph{polyhedral modified moment map flow} is then defined as
the downward gradient flow of the restricted functional
$\phi:\scrF_\alpha(\scrT)\to \RR$, for $\alpha\in H^1(M,\vecV)\setminus
0$. 
We may also consider the downward gradient flow of $\phi$ along the unit sphere
$\SS_\alpha(\scrT)$ in $\scrF_\alpha(\scrT)$ and obtain a \emph{polyhedral
renormalized flow}.

The polyhedral modified moment map flow and the polyhedral renormalized
flow are both gradient flows on finite dimensional
manifolds and  strong properties follow from the ordinary differential
equation theory.
In the case of  polyhedral modified moment
map flow, we prove the followind Duistermaat type theorem:
\begin{theointro}
	\label{theo:duistermaat}
	Let $\alpha$ be a  cohomology class in $H^1(M,\vecV)\setminus 0$.
	For every $F\in\scrF_\alpha(\scrT)$, there exists a unique solution
	of the polyhedral modified moment map flow $F_t$ defined for
	$t\in[0,+\infty)$, such that $F_0=F$. 
Furthermore, the flow line is convergent and its limit $F_\infty$ belongs
	to the vanishing locus of $\phi$. The extended flow 
	$$\Theta:[0,+\infty]\times\scrF_\alpha(\scrT)\to\scrF_\alpha(\scrT)
	$$
	given by $\Theta(t,F)=F_t$ and $\Theta(+\infty,F)=F_\infty$, defines a continuous retraction of
	$\scrF_\alpha(\scrT)$ onto the vanishing locus of $\phi$. Finally,
	the flow $\Theta$ has exponential convergence rate in a
	neighborhood of every regular point (cf.
	Definition~\ref{dfn:regular}) of the vanishing locus of $\phi$.
\end{theointro}
Given a symplectic class $\alpha\in H^1(M,\vecV)$, we consider the subspaces
$$
\scrM_\alpha(\scrT)\subset \scrM(\scrT) \quad \mbox{ and } \quad
\Symp_\alpha(M,\omega_M,\scrT)\subset \Symp(M,\omega_M,\scrT)
$$ of maps $f$ with
cohomology class $[\scrD f]=\alpha$, as in the smooth setting.

We discuss at \S\ref{sec:motivret} a potential application of
the modified moment map flow to identify the homotopy type of the space of
symplectic maps. Theorem~\ref{theo:duistermaat} shows that 
 hypothesis (H1)
and (H2) of \S\ref{sec:motivret} always hold in the polyhedral setting.
Under the additional hypothesis (H3) that $\Theta$ does  not admit any
non trivial flow line converging toward $F=0$,  the evaluation map
$ev:\scrM_\alpha(\scrT)\to M$ at some point $x_0\in M$ would be a homotopy
equivalence. 
But hypothesis $(H3)$ is too strong, in fact 
it never holds
in the polyhedral case. Indeed, Theorem~\ref{theo:polflow} states that
non trivial flow lines  converging toward $F=0$ always exist for the flow 
$\Theta$.
This motivates the introduction of blowup analysis, \emph{polyhedral soliton
equation} and 
\emph{polyhedral renormalized flow} similar to the smooth setting.
The prospect to 
obtain 
 a Morse-Bott theory adapted
to $\phi:\SS_\alpha(\scrT)\to\RR$ seems much higher than in the smooth
situation, thanks to the ODE
theory and the finite dimensional context.
At this stage we can show that the polyhedral renormalized flow satisfies
the following elementary properties:
\begin{theointro}
	\label{theo:polflow}
	Let $\alpha\in H^1(M,\vecV)$ be a symplectic class.
The polyhedral renormalized flow is a complete  gradient flow on the sphere
	$\SS_\alpha(\scrT)$. The solitons of the sphere are the fixed
	points of the flow and the limiting orbits of the  renormalized flow are
	contained in the space of solitons. Non proper solitons of
	$\SS_\alpha(\scrT)$ are in correspondence with elements of
	$\Symp_\alpha(M,\omega_M,\scrT)$, up to the action of $\vecV$  and
	the the antipodal map of the sphere. The space of 
	proper solitons in $\SS_\alpha(\scrT)$ is not empty. Furthermore,
	each proper soliton defines
	non trivial solutions of the polyhedral modified moment map flow,
	converging toward $F=0$.
\end{theointro}
Some more effort has to be undertaken
to prove that $\phi:\SS_\alpha\to \RR$ is  
Morse-Bott. Our attempts to  prove this have been unsuccessful so far.
We also failed to prove an
analogue of Theorem~\ref{theo:connected} in the polyhedral setting and we conclude our
introduction with 
the following conjecture:
\begin{conj}
	\label{conj:intro}
Given a symplectic cohomology class $\alpha$ and a Euclidean triangulation
	$\scrT_1$ of a $4$-dimensional hyperKähler quotient torus $M$,
	there exists a refinement
	$\scrT_2$ of $\scrT_1$ such that the energy functional
	$\phi:\SS_\alpha(\scrT_2)\to \RR$ has a well defined
	Morse-Bott theory.
\end{conj}

\subsection{Acknowledgments}
Most of this research was carried out in 2023-2024 thanks to the support of
the CNRS International Research Lab
CRM at Montreal. I was hosted by the CIRGET at UQAM university and wish to
thank  all my colleagues  for creating such a stimulating mathematical
environment, in particular, Antonio Alfieri, Vestislav Apostolov,
Steven Boyer, Charles Cifarelli, Olivier Collin,  Alexandra
Haedrich, Abdellah Lahdili, Duncan McCoy, Frédéric Rochon and Carlo Scarpa.

\tableofcontents

\section{Kähler moment map}
\label{sec:smap}
\subsection{Smooth symplectic maps}
We recall some basic facts in  the general case of a \emph{connected  
symplectic manifold}  $(M,\omega_M)$ of real dimension~$2n$. A~\emph{symplectic
form}  $\omega_M$ is a non degenerate closed differential 
$2$-form, defined on the smooth manifold
$M$. The symplectic form determines 
 a \emph{volume
form} 
$$
\vol_M =\frac{\omega_M^n}{n!},
$$
hence a \emph{symplectic orientation} for $M$.
The  \emph{moduli space of smooth maps} of $M$, denoted
$$
\scrM =\{f:M\to M, f\in C^\infty\}.
$$
is considered from a \emph{purely formal
perspective} in most of this paper. However, \emph{Hölder versions} of
$\scrM$ shall be introduced, when the analysis on Banach spaces is needed.
The \emph{group of diffeomorphisms}  
$\Diff(M)$ is contained in $\scrM$ and the subgroup of~\emph{symplectic maps} (or
symplectomorphisms),
denoted $\Symp(M,\omega_M)$,  consists of smooth map
$f\in\scrM$ such that  
\begin{equation}
	\label{eq:sympmap}
f^*\omega_M =\omega_M.
\end{equation}
Formula~\ref{eq:sympmap} implies that $f$ is a diffeomorphism, under some
sensible assumptions and we recall the classical result:
\begin{prop}
	\label{prop:classical}
	Let $(M,\omega_M)$ be a closed connected symplectic manifold and let
	$f:M\to M$, be a symplectic map.
	 Then $f$ is an orientation preserving diffeomorphism of~$M$. Thus
	 $\Symp(M,\omega_M)$ is a subgroup of $\Diff(M)$ and we 
	  have the  inclusions
$$
\Symp(M,\omega_M)\subset \Diff(M)\subset \scrM.
$$
\end{prop}
\begin{proof}
The non-degeneracy of $\omega_M$ implies that the tangent map $f_*:TM\to TM$ is a
fiberwise isomorphism of the tangent bundle. By the \emph{implicit function
	theorem}, it follows that the map $f$ is a local diffeomorphism. 
A symplectic map satisfies the identity $f^*\vol_M
=\vol_M$ so that $f$ preserves the  symplectic
orientation and preserves the cohomology class $[\vol_M] \in
	H^{2n}(M,\RR)$; in particular, $f$ has degree~$1$.
 An orientation preserving local diffeomorphism of degree $1$ is injective. 
	The fact that $f$ is an injective local diffeomorphism of a closed
	manifold  implies
	that $f$ is a global diffeomorphism.
\end{proof}

\subsection{Quotient Hermitian torus}
\label{sec:herm}
From this point, we  assume that the manifold $M$ is a \emph{quotient
torus} of the form
$$
M= V/\Gamma,
$$
where $V$ is
an \emph{affine space} and
 $\Gamma$ is a \emph{lattice} of the underlying vector space~$\vecV$.
 We consider the case where $\vecV$ is a complex vector space, endowed with a \emph{Hermitian
metric} 
$h_V$.
Recall that a Hermitian metric   is  a definite positive  sesquilinear
form on $\vecV$,
anti-$\CC$-linear in the first
variable. 
Then, we can write
$$
h_V(\eta_1,\eta_2) = g_V (\eta_1,\eta_2) + i\omega_V(\eta_1,\eta_2),
$$
 for every $\eta_1,\eta_2\in\vecV$,
where $g_V$ and $\omega_V$ are the real and imaginary parts of
$h_V$. By definition, $g_V$ defines a \emph{Euclidean metric} on
$V$ and~$\omega_V$, a symplectic form. 

The structure of complex vector space on $\vecV$ can be regarded as an
\emph{almost complex structure},
deduced from the complex multiplication by $i$,  understood a linear
endomorphism:
	\begin{align*}
		i:	\vecV & \longrightarrow \vecV \\
		v &\mapsto iv.
\end{align*}

The  action by translation of $\eta \in \vecV$ on $x\in V$, denoted $x\mapsto
x+\eta$, descends as a canonical  action of $\vecV$ on the
quotient torus $M=V/\Gamma$.
The Hermitian and complex structures on $V$ also descend to the quotient
$M$ as
well. The induced almost complex structure is still denoted 
 $i:TM\to TM$. The induced Hermitian structure, symplectic structure and
 Riemannian metric on the quotient $M$  are denoted  respectively~$h_M$,
 $\omega_M$ and~$g_M$.
In conclusion, we have  a flat \emph{Kähler structure} 
$$(M,g_M,i,\omega_M)$$
on the
quotient torus $M$, preserved by the action of $\vecV$ on $M$.

\subsection{Differentials and differentiation}
\label{sec:diffdiff}
The tangent bundle 
$\pi: TV\to V$
of the affine space $V$ is canonically isomorphic to the trivial
bundle $\pi_1:V\times \vecV\to V$,
where~$\pi_1$ is the  first canonical projection.
	The action of $\vecV$ on  $V$  induces an action
on the tangent bundle $TV\to V$, which is trivial on the fibers.
The tangent bundle $\pi:TM\to M$ is obtained as a quotient of $TV\to V$ by
the lattice $\Gamma$. It follows that the tangent bundle  $TM\to M$ 
is also canonically isomorphic to the
trivial bundle $\pi_1:M\times \vecV\to M$. Hence we obtain a smooth $\vecV$-invariant bundle map 
trivialization 
\begin{equation}
	\label{eq:rho}
\rho:TM\to \vecV,
\end{equation}
given by $\pi_2:M\times \vecV\to\vecV$, where $\pi_2$ is the
canonical projection on the second factor.
In addition, the map $\rho$
respects the
Kähler structures, in the sense that 
$$\rho\circ i = i\circ\rho, \quad
\rho^*\omega_V = \omega_M,\quad \rho^* g_V = g_M \mbox{ and } \rho^* h_V =
h_M,
$$
where the pullbacks are understood in the \emph{linear sense}: 
$$(\rho^*\omega_V)(\eta_1,\eta_2) = \omega_V(\rho(\eta_1),
\rho(\eta_2)), \mbox{ etc\dots}
$$
The tangent map  of  $f\in\scrM$  sits in the commutative diagram
\begin{equation}
	\begin{tikzcd}
		TM\ar[d, "\pi"]\ar[r,"f_*"]  \ar[rr,
		bend left = 40,"\scrD f"] & TM\ar[d,"\pi"]\ar[r,"\rho"]& \vecV \\
		M \ar[r,"f"]& M
\end{tikzcd}
\end{equation}
where 
\begin{equation}
	\label{eq:D}
\scrD f:TM \to \vecV,
\end{equation}
is  the $\vecV$-valued differential $1$-form, called the \emph{differential
of $f$},
defined~by
\begin{equation}
	\label{eq:diff}
	\boxed{\scrD f =\rho \circ f_*.}
\end{equation}
The operator $\scrD$ can be regarded as a linear differential operator of
order~$1$ 
$$
\scrD :\scrM\to\scrF = \Omega^1(M,\vecV),
$$
where $\scrF$ denotes the space of $\vecV$-valued differential $1$-forms on $M$.
 The operator $\scrD$ is closely related to
the usual differential $d$ according to
Remark~\ref{rmk:lift}.
\begin{rmk}
	\label{rmk:lift}
An alternate construction of  $\scrD f$ proceeds  as follows: let
$\tilde f:V\to V$ be a lift of $f:M\to M$  to  the universal cover $p:V\to
	M$. In other words, $\tilde f$ is a smooth
map such that $p\circ \tilde f = f\circ p$. The differential $d\tilde f$
is a  $\vecV$-valued $1$-form on $V$, which is to say an element of
$\Omega^1(V,\vecV)$. By construction $\tilde f$ is $\Gamma$-invariant, so
that $d\tilde f$ descends to the quotient as an element of 
$\Omega^1(M,\vecV)$, which agrees with  $\scrD f$.
\end{rmk}

Since $\vecV$ acts on $M$ by translations, the space of maps $\scrM$ admits
a canonical action by the space of $\vecV$-valued functions, $\Omega^0(M,\vecV)$: given $f\in\scrM$ and $\eta \in\Omega^0(M,\vecV)$,
we define $f+\eta \in\scrM$ by
$$
(f+\eta)(x) = f(x)+\eta(x)
$$
for every $x\in M$.
In particular, the space of constant $\vecV$-valued functions, identified to $\vecV$, acts on $\scrM$.
By construction, we have the following proposition:
\begin{prop}
	\label{prop:inj}
	Given two maps $f$ and $h\in\scrM$, 
we have
	$$
\scrD f = \scrD h
$$
if, and only if, there exists a vector $\eta\in\vecV$, such that
	$f=h+\eta$.
	In particular, the map $\scrD$ descends as a bijection
$$
	\scrM/ \vecV\stackrel{\scrD}{\longrightarrow}
	\scrD(\scrM)\subset\scrF,
$$
where $\vecV$ acts on $\scrM$ by translation.
\end{prop}
\begin{proof}
	The construction of $\scrD f$ and $\scrD h$ by 
	Remark~\ref{rmk:lift} is given by $d\tilde f$ and $d\tilde h$, where
	the maps $\tilde f, \tilde
	h: V\to V$ are lifts of $f$ and $g$.
 The equation
	$\scrD f = \scrD h$ is equivalent to $d\tilde f= d\tilde h$,
	where $d$ is the usual differential. Therefore $\tilde f$ and
	$\tilde h$ agree up to  a translation by a vector
	$\eta\in \vecV$, which proves the proposition.
\end{proof}

\subsection{Cohomology}
\label{sec:cohom}
In this section, we recall some basic facts about the cohomology of a torus
and homotopy classes of maps of a torus.
The usual exterior differentials  for real valued differential forms
 are denoted $d_k:\Omega^k(M) \to \Omega^{k+1}(M)$, for $k\geq -1$, with the
 convention that $d_{-1}:0\to \Omega^0(M)$ is the $0$-map. 
 These operators admit a canonical  extension
to  $\vecV$-valued differential
 forms, 
denoted similarly
$d_k:\Omega^k(M, \vecV) \to \Omega^{k+1}(M, \vecV)$,
and  defined by the property that for every real valued differential $k$-form
$\beta$ and vector $\eta\in\vecV$, we have
\begin{equation}
	\label{eq:dfndiff}
	d_k (\beta\otimes \eta) = (d_k\beta)\otimes \eta,
\end{equation}
where $d_k\beta$ is the usual exterior differential.
Thus we have  a complex 
$$
0\stackrel{d_{-1}}\longrightarrow \Omega^0(M,\vecV)
\stackrel{d_{0}}\longrightarrow \Omega^1(M,\vecV)
\stackrel{d_{1}}\longrightarrow \cdots
\quad \cdots \stackrel{d_{k-1}}\longrightarrow \Omega^k(M,\vecV)
\stackrel{d_{k}}\longrightarrow \cdots
$$
 with associated \emph{De Rham cohomology spaces} 
$$
H^k(M,\vecV)= \frac{\ker d_k  }{\Im d_{k-1}} \mbox{ defined for every }
k\geq 0
$$
and canonical isomorphisms
$$
H^k(M,\RR)\otimes\vecV\simeq H^k(M, \vecV).
$$

\begin{lemma}
	\label{lemma:cohom}
The image of the map $\scrD:\scrM\to\scrF$ consists of closed forms. In
	other words
$d_1\circ \scrD= 0$,
which is to say
$$
\Im\scrD\subset \ker d_1.
$$
\end{lemma}
\begin{proof}
	The result follows immediately from Remark~\ref{rmk:lift}, Formula~\eqref{eq:dfndiff} and the fact
	that $d^2=0$.
\end{proof}

Every form $\scrD f\in\scrF$ is a closed $\vecV$-valued
$1$-form by Lemma~\ref{lemma:cohom} and defines a cohomology class denoted 
$[\scrD f]\in H^1(M,\vecV)$. Thus, we  have a canonical map 
\begin{equation}
\begin{tikzcd}[  baseline=(current  bounding  box.center), cramped,
  row sep = 0ex,
  column sep = 1.5em,
  /tikz/column 1/.append style={anchor=base east},
  /tikz/column 2/.append style={anchor=base west}
   ]
		 \alpha: \scrM \ar[r] & H^1(M,\vecV) \\
		 f\ar[r,mapsto] & {\alpha(f)=[\scrD f ]}.
\end{tikzcd}
\end{equation}
The quotient torus $M$ is equipped with a canonical isomorphism
$$
\psi:\vecV^*\to H^1(M,\RR).
$$
 Indeed, any element  $\beta \in \vecV^*$ can be understood as a
constant differential $1$-form on $V$. The constant form $\beta$ descends to the
quotient $M$ as a closed $1$-form denoted $\beta$ as well.
Finally, $\beta$ defines a cohomology class
$\psi(\beta)=[\beta]\in H^1(M,\RR)$. The fact that 
$\psi$ is an isomorphism  is a classical fact for tori.
The \emph{dual lattice} $\Gamma^*\subset \vecV^*$ is defined by
$$
\Gamma^* = \{\beta \in \vecV^*, \forall\gamma\in
\Gamma,\beta(\gamma)\in\ZZ \}.
$$
 Using the isomorphism~$\psi$, 
 the dual lattice $\Gamma^*\subset \vecV^*$ is
understood as  a
lattice in $H^1(M,\RR)$ identified to 
the group of integral classes  
$H^1(M,\ZZ) \subset H ^1(M,\RR) $ and we have a canonical isomorphism
$$
\psi:\Gamma^*\stackrel \sim \to H^1(M,\ZZ).
$$
Every map $f\in\scrM$, defines a canonical morphism 
$$f^*:H^1(M,\ZZ)\to
H^1(M,\ZZ)$$
 which depends only on the homotopy class of of $f$.
Using the isomorphism~$\psi$, the map $f^*\in\End(H^1(M,\ZZ))$ 
is identified to a lattice endomorphism
$f^*\in\End(\Gamma^*)$.
The canonical isomorphisms
 $\End(\Gamma^*)\simeq \End(\Gamma) \simeq
 \Gamma^*\otimes_\ZZ\Gamma$ show that
$$
\End(\Gamma^*)\simeq H^1(M,\ZZ)\otimes_\ZZ \Gamma
\subset H^1(M,\RR)\otimes_\RR   \vecV=  H^1(M,\vecV).
$$
If a class $\alpha\in H^1(M,\vecV)$ belongs to the image of
$\End(\Gamma^*)$, we say that it is an \emph{integral cohomology class}. 
We recall the classical result of topology,  well known in the case of a
circle:
\begin{prop}
	\label{prop:top}
	For $f\in\scrM$, the  morphism $f^*\in\End(\Gamma^*)$,  understood
	as an integral cohomology class in
	$H^1(M,\vecV)$, agrees 
with $\alpha(f)=[\scrD f]$.
	
	Two smooth maps $f$ and 
	$h: M\to M$ are homotopic if, and only
	if, $\alpha(f)=\alpha(h)$. 
	Furthermore, the map 
	\begin{align*}
		\scrM &\to
	\End(\Gamma^*)\\
		f& \mapsto f^*
	\end{align*}
	 is surjective and
	induces bijection between
	$\End(\Gamma^*)$ and the set of homotopy classes of smooth maps
	$f:M\to M$.
\end{prop}

According to the above proposition, homotopy classes of maps are
parametrized by integral cohomology classes $\alpha\in H^1(M,\vecV)$. Thus, it makes sense
to use the notation 
$$\scrM_\alpha=\{f\in\scrM, [\scrD f]= \alpha\}
$$
for each homotopy component of the space
of smooth maps of the torus. Similarly we use the notations
$$\Symp_\alpha(M,\omega_M) \quad\mbox{ and } \quad \Diff_\alpha(M),
$$for the subspaces of maps with
integral cohomology class $\alpha$.

\begin{rmk}
	\label{rmk:linearrep}
	The surjectivity property of
	Proposition~\ref{prop:top} can be realized using \emph{affine
	maps}:
	the canonical map
	$\End(\Gamma)\longrightarrow\End(\Gamma^*)$
	is an isomorphism, so that an element of $\alpha\in\End(\Gamma^*)$
	corresponds to a morphism
	$\vec f$ of the lattice $\Gamma$. The endomorphism $\vec f$ extends canonically as a linear map
	$\vec f \in \End(\vecV)$. Any affine map $f:V\to V$ with
	linear part $\vec f$ descends to the quotient
	as a map $f:M\to M$ such that  $\alpha(f)=\alpha$.
\end{rmk}

\begin{dfn}
	A cohomology class $\alpha \in H^1(M,\vecV)$ is called symplectic,
	if there exists $f\in\scrM$, such that
	$\alpha = [\scrD f ]$ and $f^*[\omega_M]=[\omega_M]$.
\end{dfn}

\begin{rmk}

	Let $\alpha\in H^1(M,\vecV)$ be a symplectic cohomology class. 
	According to Remark~\ref{rmk:linearrep}, 
	there exists an affine
	map $f:V\to V$ which descends to the quotient as $f:M\to M$ with
	the property that $\alpha=[\scrD f]$.
	We have $f^*[\omega_M]
	=[\omega_M]$ since $\alpha$ is symplectic. From the fact that $f$ is
	affine we deduce that $f^*\omega_M$ is a constant differential form, which
	forces $f^*\omega_M=\omega_M$. Thus $f$ is
	an affine symplectic map.
	We conclude that the set of symplectic classes is
	in $1:1$-correspondence with the group of symplectic isomorphisms of
	$\Gamma$
	\begin{equation}
		\Symp(\Gamma,\omega_V)=  \End(\Gamma)\cap
		\Symp(\vecV,\omega_V)
	\end{equation}
	where $\Symp(\vecV,\omega_V)$ is the group of linear symplectic
	endomorphisms of $\vecV$.
	In the above definition, every element $\vec f\in\End(\Gamma)$
	induces  a canonical linear endomorphism of $V$ and we may consider
	that $\End(\Gamma)\subset \End(\vecV)$.
\end{rmk}

\begin{rmk}
	\label{rmk:surj}
	The surjectivity property of Proposition~\ref{prop:top} can also be established
	directly, relying on integration:
	let $\alpha$ be an integral cohomology class in $H^1(M,\vecV)$ and
$F\in\scrF$, a representative of $\alpha$. Let
	Let $\tilde F= p ^* F\in\Omega^1(V,\vecV)$ be the lift of $F$ to the universal cover $p:V\to
M$. 
	Given $y_0,y_1$ and $ y\in V$, let $\tilde f:V\to V$ be the function
defined by
	\begin{equation}
		\label{eq:integral}
\tilde f(y)= y_1 + \int_{y_0}^y \tilde F,
	\end{equation}
where the integral is taken along any smooth path in $V$ from~$y_0$ to~$y$.
The integral is independent of the choice of path
between~$y_0$ and~$y$ since $\tilde F$ is closed and $V$ is simply connected.
By construction, 
$\tilde f(y+\gamma)-\tilde f(y)\in\Gamma$ for every $\gamma\in\Gamma$, since $\alpha=[F]$ is an integral
class. Hence $\tilde f$ descends as a smooth map $f:M\to
	M$. Using the notations $x_0=p(y_0)$, $x_1=p(y_1)$, we have
	$f(x_0)=x_1$, $\scrD f = F$ and $[\scrD f]=[F]=\alpha$, by
	definition.
	In
	conclusion we have  an integral
	\begin{align}
		\chi:\alpha & \to\scrM_\alpha \label{eq:primit}\\
		F &\mapsto f\nonumber
	\end{align}
	such that $f(x_0)=x_1$ and $f$ is defined by Formula~\eqref{eq:integral}.
\end{rmk}

\subsection{Bundle structures}
\label{sec:bunstruct}
In \S\ref{sec:herm}, we outlined how a Hermitian structure on $\vecV$
induces a flat Kähler structure~$(M,g_M,i,\omega_M)$ on the quotient torus~$M=V/\Gamma$. 
We consider a second almost complex
structure
$$J_V:\vecV\to\vecV$$
 on~$\vecV$, compatible with the Riemannian metric $g_V$, in the sense that
$$g_V(J_V \eta_1,J_V \eta_2)= g_V(\eta_1,\eta_2),$$
for every $\eta_1, \eta_2\in\vecV$. The
compatible almost complex structure $J_V$ provides an alternate 
symplectic form on $V$ defined by
$$
\hat\omega_{V} = g_V(J_V\cdot,\cdot).
$$
The almost complex structure $J_V$ and the symplectic form $\hat \omega_V$  descend to the
quotient $M$ and  provide an alternate Kähler structure
$(M,g_M,J_M,\hat \omega_M)$, where the Kähler form $\hat\omega_M$ is induced by
$\hat\omega_V$.
In conclusion, we have two competing Kähler structures on the torus
$$(M,g_M,i,\omega_M)\mbox{ and } (M,g_M,J_M,\hat\omega_M).
$$

The Riemannian metric induces a \emph{musical isomorphism} between the cotangent bundle 
$T^*M\to M$ and the tangent bundle $TM\to M$. 
Thus the cotangent bundle and all the bundles of $k$-forms
$\Lambda^k M\to M$ are endowed with an induced fiberwise Euclidean inner product
 deduced from $g_M$ and denoted $g$. 
It follows that the complex vector bundles 
$$\Lambda^k M\otimes_{\RR}
\vecV\to M$$
have several induced structures:
\begin{enumerate}
	\item  The vector bundles $\Lambda^k M \otimes_\RR \vecV \to M$  carry  fiberwise 
Euclidean inner products, denoted $g$ as well,
		induced by $g$ and $g_V$.
	\item \label{struct:i} The vector bundles $\Lambda^k M \otimes_\RR \vecV \to M$  carry a canonical structure
	of complex vector bundle deduced from the structure of complex
		vector space on~$\vecV$.
\item  The vector bundle $T^* M \otimes_\RR \vecV\to M$  carries a second fiberwise 
	almost complex structure $J$ deduced from $J_M$ and 
		defined by
$$
		J\cdot F =-F\circ J_{M} \mbox{ for every $F\in
		T^*M\otimes_\RR\vecV$}.
$$
The almost complex structure $J$ commutes with the
		multiplicative action of $\CC$ of complex vector bundle.
		Hence $J$ is complex linear in this sense. In particular $i$ and $J$
		commute:
		$$
		i  (J\cdot F)= -iF\circ J_M = J\cdot (i F).
		$$
		 \item The metrics are compatible with the almost complex
			 structures:
			 $$
		g(iF_1,iF_2)=g(F_1,F_2)\quad \mbox{ and } \quad g(J\cdot
		F_1,J\cdot F_2)=
		g(F_1,F_2)
		$$
		for every $F_1, F_2\in
		T_x^*M\otimes\vecV$.
\item The bundle $T^* M \otimes_\RR \vecV\to M$  carries a  fiberwise
	anti-symmetric $2$-form, denoted $\hat\omega$, defined by
$$
		\hat\omega (F_1,F_2)= g(J\cdot F_1,F_2).
$$
\end{enumerate}
\begin{rmk}
	Once the Euclidean inner product $g$ is understood, we will often use
the more compact notations
$$
	\ip{F_1,F_2}=g(F_1,F_2)\quad \mbox{and} \quad |F |^2=\sqrt{g(F,F)}.
$$
\end{rmk}

\subsection{An involution}
We define a bundle endomorphism 
$$\begin{tikzcd}
	T^*M\otimes_\RR\vecV \ar[rr,"R"]  \ar[dr]&  &\ar[dl] T^*M\otimes_\RR\vecV\\
	&M&
\end{tikzcd}
$$
 by
the formula
$$
R F =  iJ\cdot F
$$
 for every $F\in T^*M\otimes_\RR\vecV$.
We have $R^2 F= i^2J^2 \cdot F=F$ since $i$ and $J$ commute, which shows
that $R$ is an \emph{involution}.
This provides a splitting
$$
T^*M \otimes_\RR \vecV =\Lambda^{1,+} \oplus \Lambda^{1,-} 
$$
where the
$\Lambda^{1,\pm}$ are the eigenspaces of $R$ associated to the eigenvalues $\pm 1$.
By definition, the elements $F\in \Lambda^{1,+}$ (\resp $\Lambda^{1,-}$) 
are the  $J_M$-complex (\resp
anti-complex)  morphisms $F:TM\to\vecV$. 

The involution $R$ induces a canonical involution of the space of differential
forms 
$$
\scrR:\scrF\to\scrF
$$
defined by 
$(\scrR F)_x = R (F_x)$ for every $x\in M$.
Similarly, the fiberwise almost complex structure $J$ of $T^*M\otimes\vecV$
defines a
canonical almost complex structure 
$$\cJ:\scrF\to\scrF
$$
given by 
$(\cJ F)_x=
J(F_x)$
for  $F\in\scrF$ and every $x\in M$.
Then $\scrR F= i\cJ F$ and we have a similar splitting
$$
\scrF =\scrF^+ \oplus \scrF^-
$$
into complex  and anti-complex differential forms. It follows that every
differential form
$F\in\scrF$ admits a canonical decomposition
$$
F= F^+ + F^-,
$$
with components $F^\pm\in \scrF^\pm$.

\subsection{Gauge group action}
\label{sec:gac}
We recall that the complex vector bundle structure of $T^* M \otimes_\RR \vecV \to M$ is
deduced from the  complex vector space structure of $\vecV$
(cf. \ref{sec:bunstruct}, item
\eqref{struct:i}). 
The space of differential forms $\scrF$ inherits a module structure
over the ring
of complex
valued functions  $\lambda:M\to \CC$, with multiplication given by
$$(\lambda F)_x = \lambda(x)F_x \quad \mbox{ for every $x\in M$}.$$
We define an action of the \emph{infinite dimensional
complex torus} 
$$
\TT^\CC = C^\infty(M,\CC^*) 
$$
on $\scrF$ 
as follows: for every $\lambda\in\TT^\CC$ and $F\in\scrF$, we put
\begin{equation}
	\label{eq:action}
\lambda\cdot F= {\bar\lambda}^{-1}F^++\lambda F^-,
\end{equation}
 where multiplication in the
RHS comes from the module structure and $\bar\lambda$ denotes the
complex conjugate.
The gauge group $\TT^\CC$ contains  the \emph{real torus}
$$
\TT = C^\infty(M,S^1)\subset\TT^\CC, 
$$
where $S^1\subset \CC$ is identified to complex numbers of module $1$.
If  $\lambda\in \TT$, the action on $\scrF$ is given
by the standard complex multiplication $\lambda\cdot F =\lambda F$ since
$\lambda^{-1}=\bar\lambda$ in this case.
\begin{rmk}
	The action of $\TT^\CC$ defined by \eqref{eq:action} depends on the splitting
	$\scrF^\pm$, and ultimately on $J_M$. However, the action of $\TT$,
	which is merely the standard multiplication by a
	complex function,  does not
	depend on~$J_M$.
\end{rmk}

\subsection{Infinitesimal action}
Formally, $\TT^\CC$ is understood as a \emph{Lie group} acting on $\scrF$ with
Lie algebra $\LieTT^\CC$ identified to $C^\infty(M,\CC)$. The \emph{exponential map}
$$
\exp:\LieTT^\CC \to \TT^\CC
$$
is  defined by the formula
$$\exp(\zeta)= e^{i\zeta},\quad  \mbox{ for } \zeta\in\LieTT^\CC= C^\infty(M,\CC).$$
With these conventions, the Lie algebra
$\LieTT$ of the subgroup $\TT$ is identified to~$C^\infty(M,\RR)$ via the exponential
map. 
As usual, in the context of a Lie group action, the infinitesimal
action of an element of the Lie algebra $\zeta\in\LieTT^\CC$ is given the vector
field $X_\zeta$ on $\scrF$, defined by the formula
$$
X_\zeta(F) = \left . \frac{\del}{\del t}\exp(t\zeta)\cdot F \right|_{t=0}.
$$
\begin{lemma}
	\label{cor:infin}
The vector field $X_\zeta$ associated to the  infinitesimal action of
	$\zeta\in\LieTT^\CC$ on $\scrF$ satisfies the formula
$$
X_\zeta (F) = i \bar \zeta F^+ + i\zeta F^-,
$$
for every $F\in\scrF$.
	In the special case where
	$\zeta\in
	\LieTT$  we have
	\begin{equation}
		\label{eq:infintt}
		X_\zeta(F)= i\zeta F.
	\end{equation}
	 If $\zeta\in i\LieTT$, then
	 $$X_\zeta(F)=
-i\zeta \scrR F=\zeta \cJ\cdot F$$ 
and for every $\zeta \in\LieTT$, we have
\begin{equation}
	\label{eq:cxaction}
	\cJ \cdot X_\zeta(F) = i\zeta \cJ\cdot F = X_{i\zeta}(F).
\end{equation}

\end{lemma}
\begin{proof}
	This follows immediately from formula \eqref{eq:action} and the
	definition of the exponential map.
\end{proof}

\subsection{Kähler structure on the moduli space}
\label{sec:kahl}
The moduli space $\scrF$ of $\vecV$-valued $1$-forms  is endowed with a formal
canonical Kähler structure described below~:
\begin{enumerate}
	\item For $F_1$ and $F_2$ in $\scrF$, we define a Euclidean metric
		$\cG$ on $\scrF$ by
		\begin{equation}
			\label{eq:metr}
		\cG(F_1,F_2) = \int_M g(F_1,F_2) \vol_M,
		\end{equation}	
		where $\vol_M$ is the volume form associated to
		$\omega_M$.
	As $g$ is defined on every tensor bundle, it follows that the Euclidean metric $\cG$
		is defined for every tensor space as well, by
		the same formula. The short notation
		$$\ipp {\cdot,\cdot}=\cG(\cdot,\cdot)$$
		and the corresponding
		$L^2$-norm notation 
		$$\|F\|_{L^2}=\sqrt{\ipp{F,F}}
		$$
		will often be used instead of $\cG$.

\item The space $\scrF$ is endowed with  the almost complex structures
	$\cJ$ defined above: for $F\in\scrF$, 
		\begin{equation}
			\label{eq:ac}
			(\cJ\cdot F)|_x = J\cdot F_x 
		\end{equation}
		for every $x\in M$. Then $\cJ$ is compatible with
		$\cG$ in the sense that
		$$ \cG(\cJ\cdot,\cJ\cdot)=\cG.$$
	\item We define a symplectic form $\Omega$ on $\scrF$ by 
		\begin{equation}
			\label{eq:symp}
		\Omega(F_1,F_2) =\cG(\cJ\cdot F_1,F_2) =
		\int_M\hat\omega(F_1,F_2)\vol_M.
		\end{equation}
\end{enumerate}
In particular $(\scrF,\cG,\cJ,\Omega)$ is a formal Kähler structure.
We summarize our construction in the next proposition.
\begin{prop}
	\label{prop:strk}
Let $M=V/\Gamma$ be a quotient torus, where $V$ is an affine space with
	lattice $\Gamma$ and 
	$\vecV$,  a complex vector
	space endowed with a Hermitian structure $h_V$, as an underlying
	vector space.
	For every almost complex
	structure $J_V$  compatible with the Euclidean metric
	$g_V$ deduced from $h_V$, there exists a natural formal Kähler structure
	$(\cG,\cG,\cJ, \Omega)$
	on $\scrF$ 
	defined by Formulas~\eqref{eq:metr}, \eqref{eq:ac} 
	and \eqref{eq:symp},
	acted on by the gauge group~$\TT^\CC=C^\infty(M,\CC^*)$, with an
	action given by
	Formula~\eqref{eq:action}.
\end{prop}

\begin{prop}
	\label{prop:presj}
	The $\TT^\CC$-action on $\scrF$ preserves the almost complex structure $\cJ$. 
	Furthermore, the action of $\TT^\CC$  is the $\cJ$-complexified
	action of $\TT$. 
\end{prop}
\begin{proof}
It is sufficient to prove that the $\TT^\CC$ action commutes with $\cJ$ and
this proves the first statement. By definition $\scrR=i\cJ$, hence
	$-i\scrR=\cJ$. Furthemore, $\scrR$ is a module morphism over the
	ring of  complex valued functions. Then
	for every $\lambda\in \TT^\CC$, we have
	$\cJ\cdot\lambda\cdot F = -i\scrR \lambda \cdot F$ and, by definition of
	the torus action~\eqref{eq:action},
	\begin{align*}
		-i\scrR\lambda \cdot F &= -i\scrR( \bar\lambda^{-1}F^+ + \lambda
	F^-) \\
		&=-i(\bar\lambda^{-1}\scrR F^+ +\lambda\scrR F^-)\\
		&=  -\bar\lambda^{-1}iF^++ \lambda iF^-\\
		&= \lambda\cdot (-i\scrR
	F)\\
		&=\lambda\cdot \cJ \scrR F.
	\end{align*}
We conclude that $\lambda\cdot\cJ F= \cJ\lambda\cdot F$.
	
	The second part of the statement follows from
	Corollary~\ref{cor:infin}.
	By construction we have the splitting of the Lie algebra with
	trivial Lie bracket $\LieTT^\CC= \LieTT\oplus i\LieTT$, where
	$\LieTT$ is identified to real valued functions and $\LieTT^\CC$
	complex valued functions.
	By~Formula \eqref{eq:cxaction}, the infinitesimal
	action satisfies $\cJ X_\zeta = X_{i\zeta}$ for every $\zeta\in\LieTT$,
	which proves the second statement of the lemma.
\end{proof}

\begin{prop}
	\label{prop:presk}
The Kähler structure  $(\scrF,\cG,\cJ,\Omega)$ is invariant under the
	action of $\TT$.
\end{prop}
\begin{proof}
For  $\lambda\in\TT$, the action is just the usual complex multiplication
	$\lambda\cdot F =\lambda F $, so that $g(\lambda F,\lambda
	F)=|\lambda|^2g(F,F)=g(F,F)$. Therefore $\cG(\lambda\cdot F,\lambda\cdot
	F)=\cG(F,F)$ and the $\TT$-action preserves $\cG$. By
	Proposition~\ref{prop:presj}, the action
	preserves the almost complex structure $\cJ$, hence the Kähler form
	$\Omega$ is preserved as well.
\end{proof}

\subsection{Hamiltonian gauge group action}
\label{sec:ham}
The $\TT$-action on the moduli space~$\scrF$ turns
out to be Hamiltonian, as we are going to show in this section.
Let
$$
\mu : \scrF \to \LieTT= C^\infty(M,\RR)
$$
be the map given by the formula
\begin{equation}
	\label{eq:mmdef}
\mu(F) = -\frac 12 g(\scrR F,F).
\end{equation}
The map $\mu$ is $\TT$-invariant and simple computations give
\begin{align}
	D\mu|_F\cdot \dot F &= -\frac 12 \Big (g(\scrR F,\dot F)+g(\scrR \dot
	F,F)\Big )
	\\
	&=- g(\scrR
	F,\dot F)\label{eq:diffmu}\\
	&= - g(i\cJ F, \dot F)\\
	&= -\hat \omega(iF,\dot F).
\end{align}
Using the $L^2$-inner product on $\LieTT$,
we compute
\begin{align*}
	\ipp{D\mu|_F\cdot \dot F,\zeta}&=-\int_M \hat\omega( i F, \dot
F)\zeta\;\sigma_M\\
	&=-\int_M \hat\omega( i\zeta F, \dot
F)\;\vol_M \\
	&=-\Omega(i\zeta F,\dot F)\\
	&= -\iota _{X_\zeta(F)} \Omega(\dot F) ,
\end{align*}
where $X_\zeta$ is the infinitesimal action of $\zeta\in\LieTT$  on $\scrF$, given by
Formula~\eqref{eq:infintt}.
The above identity shows that $\mu$ is a moment map for the action of $\TT$.
We summarize our construction as follows:
	\begin{theo}
		\label{theo:ham}
		Let $M=V/ \Gamma$ be a quotient torus, where $V$ is an
		affine space with $\vecV$ as an underlying vector space. We
		assume that $\vecV$ is a complex Hermitian vector space,
		endowed with an
		an alternate almost complex structure $J_V$ as in
		Proposition~\ref{prop:strk}.
		 Then the gauge group action of $\TT$ on the moduli space $\scrF$
		 is formally Hamiltonian with respect to $\Omega$, with
		 moment map $\mu$ given by Formula~\eqref{eq:mmdef}.
 In other
	words,
	for every $ \zeta\in\LieTT$, we have
	$$
		D\ipp{\mu ,\zeta}  = -\iota_{X_\zeta}\Omega,
$$
where $X_\zeta$ is the infinitesimal action of $\zeta$ on $\scrF$.
\end{theo}
\subsection{Moment map and symplectic density}
\label{sec:density}
The moment
map defined by Formula~\eqref{eq:mmdef} is some type of \emph{symplectic
density}, as we are going to see.
For $F\in\scrF$, we have by definition
$$ 2\mu(F) =  g ( i F \circ J_M ,F).
$$
We choose a local oriented orthonormal frame
$(e_1,f_1,\cdots,e_n,f_n)$ of $T_xM$  adapted to the Kähler structure
$(g_M,J_M,\hat\omega_M)$. This means that $J_Me_j=f_j$, so that the
symplectic form is given by the formula
$$
\hat \omega_M=\sum_{k=1}^n e_j^*\wedge f_j^*
$$
on $T_xM$.
Then
by definition of the metric $g$
\begin{align*}
	2\mu(F) &=  \sum_{k=1}^n  g_V(i F\circ J_M (e_k),F(e_k))
+ g_V(i F\circ J_M (f_k),F(f_k))\\
	&= \sum_{k=1}^n ig_V (iF(f_k),F(e_k) )
 -g_V(i F(e_k),F(f_k) ) \\
	&=  -2\sum g_V(i F(e_k),F(f_k) )\\
	&= -2 \sum \omega_V (F(e_k),F(f_k))\\
\end{align*}
The pullback of the Euclidean symplectic
form $\omega_V$ by $F$ is defined by
$$
F^*\omega_V(\eta_1,\eta_2) = \omega_V(F(\eta_1),F(\eta_2)),
$$
and the above computation shows that
$$
\mu(F) = - \sum_{k=1}^n F^*\omega_V(e_k,f_k).
$$
Assuming $n=2$ for simplicity, we have
$$
\hat \omega_M = e_1^*\wedge f_1^*+e_2^*\wedge f_2^* 
$$
so that
\begin{equation*}
F^*\omega_V\wedge \hat\omega_M = -\mu(F)\vol_M
\end{equation*}
and  we deduce the following result~:
\begin{prop}
	\label{prop:density}
In the case where $M=V/\Gamma$ is a torus of real dimension~$4$, we have 
\begin{equation}\label{eq:density}
	\mu(F) = -\frac {(F^*\omega_V)\wedge \hat\omega_M}{\vol_M},
\end{equation}
for every $F\in\scrF$.
In particular, if $F=\scrD f$ for some $f\in\scrM$, then
$$
	\mu(\scrD f) =  -\frac {(f^*\omega_M)\wedge \hat\omega_M}{\vol_M}.
$$
\end{prop}

\section{HyperKähler moment map}
\label{sec:hk}

\subsection{HyperKähler torus}
\label{sec:hktorus}
We restrict our attention to the case of a hyperKähler torus of real dimension $4$.
More precisely, we assume that the vector space  $\vecV$ is equiped with a
linear
isomorphism $\vecV\simeq\HH$ with 
the space of quaternions $\HH$. 
Recall that the space of quaternions $\HH$ can be constructed as a $4$-dimensional real
vector space: every
quaternion $q\in\HH$ can be written as a linear combination
$$
q=a+bi+cj+dk,
$$
where $a,b,c ,d $ are real numbers and the quaternions
$i$,
$j$ and $k$  satisfy the \emph{quaternionic relations} 
$$
i ^2 = j ^2 = k^2= ijk = -1.
$$
 In particular, we have a canonical inclusion 
 $\CC\subset \HH$, since every complex number of the form
$\lambda=a+bi$ is by definiton a quaternion. 
Among the  various possible structures of complex vector space on $\HH$, we
define the action
of $\lambda\in\CC$, on $q\in\HH$ 
 by the 
\emph{quaternionic multiplication on the right}:
$$
\lambda\cdot q = q\lambda.
$$
This multiplicative action of $\CC$ endows $\HH$ with a structure of complex vector space.
In particular, this complex
structure is such that the map
\begin{equation}
	\label{eq:isom}
	\begin{tikzcd}[  baseline=(current  bounding  box.center), cramped,
  row sep = 0ex,
  column sep = 1.5em,
  /tikz/column 1/.append style={anchor=base east},
  /tikz/column 2/.append style={anchor=base west}
   ]
		  \CC^2 \ar[r] & \HH \\
		 (z_1,z_2)\ar[r,mapsto] & z_1 + jz_2.
\end{tikzcd}
\end{equation}
is a complex linear isomorphism.

We define $3$ almost complex structures $I,J, K$ on $\HH$, given by
the quaternionic multiplication by $i,j$ and $k$ \emph{on the left}:
$$
I\cdot q = i q, \quad J\cdot q = jq\quad \mbox{ and }\quad K\cdot q  = kq,
$$
and we also denote by $I$, $J$ and $K$ the almost complex structures
deduced on $\CC^2$ using the isomorphism~\eqref{eq:isom}.

From now on $V$ is an affine space  with the underlying  vector
space $\vecV$ identified to $\HH$. Then $\vecV$ is isomorphic to the complex vector
space $\CC^2$ via \eqref{eq:isom}, equiped with its canonical Hermitian 
structure. Furthermore $\vecV$ is endowed with three additional almost complex structure $\bullet =
I, J$  and
$K$, deduced from the one on $\HH$.
The  three almost complex structures are compatible with the metric $g_V$
and their
corresponding symplectic forms
$\hat\omega_\bullet = \hat\omega_I, \hat\omega_J$ and $\hat\omega_K$ are  given by
$$
\hat\omega_I = g_V(I\cdot,\cdot),\quad
\hat\omega_J = g_V(J\cdot,\cdot)\quad \mbox{ and } \quad
\hat\omega_K = g_V(K\cdot,\cdot).
$$

\subsection{Selfduality}
The metric $g_V$ and the orientation induced by $\omega_V$
provide a \emph{Hodge operator}, which is an involution
$\star:\Lambda^2\vecV^*\to\Lambda^2\vecV^*$ defined by 
$$
\gamma_1\wedge \star \gamma_2 = g_V(\gamma_1,\gamma_2)\vol_V,
$$
where $\vol_V = \frac {\omega_V^2}2$.
 We have the well known splitting of $2$-forms 
$$
\Lambda^2\vecV^*= 
\Lambda^{2,+}\vecV^*\oplus
\Lambda^{2,-}\vecV^*
$$
into \emph{selfdual} and \emph{anti-selfdual} $2$-forms, corresponding to the
$\pm1$ eigenvalues of the Hodge $\star$ operator.
The isomorphism between $\vecV$ and $\CC^2$ together with the canoncial
coordinates $z_1 = x_1+iy_1$ and $z_2=x_2+iy_2$, give the fomulas
\begin{align*}
	\hat \omega_I &= dx_1\wedge dy_1 - dx_2\wedge dy_2,\\
	\hat \omega_J &= dx_1 \wedge dx_2+ 
dy_1\wedge dy_2, \\
	\hat \omega_K &= -dx_1\wedge dy_2 - dx_2\wedge dy_1.
\end{align*}
Notice that these forms are quite different from $\omega_V = dx_1\wedge dy_2 + dx_2\wedge
dy_2$. Using the fact that $\star (dx_1\wedge dy_1) =
dx_2\wedge dy_2$, $\star (dx_1\wedge dx_2) = -dy_1\wedge dy_2$ and $\star (dx_1\wedge
dy_2) = dy_1\wedge dx_2$, we observe that the forms $\hat\omega_\bullet$  span the space of anti-selfdual $2$-forms
$\Lambda^{2,-}\vecV^*$, whereas
$\omega_V$ is a selfdual $2$-form. In particular, we have the decomposition
$$
\Lambda^{2,-}\vecV^* =  \RR \hat\omega_I \oplus \RR \hat\omega_J \oplus \RR
\hat\omega_K
$$
and the formulas
$$
\hat \omega_I^2 = \hat \omega_J^2 =\hat \omega_K^2 = -\omega_V^2=
-2\vol_V.
$$
Given a lattice $\Gamma$ of $\vecV$, we obtain a torus $M=V/\Gamma$ of real
dimension~$4$.
The metric $g_V$, the almost complex structures $i, I, J, K$ and the Kähler
forms $\omega_V$, $\hat\omega_\bullet$  descend to the quotient $M$.
We obtain a flat \emph{canonical Kähler structure} $(M,g_M,i,\omega_M)$ and
a \emph{conjugate hyperKähler
structure} $(M,g_M,I,J,K)$ with three Kähler forms denoted
$\hat\omega_\bullet$ as well.
 By construction  
the canonical Kähler form $\omega_M$ is selfdual, whereas the $\hat\omega_\bullet$
are anti-selfdual.

\subsection{Triple moment map}
\label{sec:hkmap}
According Proposition~\ref{prop:strk}, the almost complex structures $I,J$
and $K$ on $M$ induce three almost complex structures $\cI$, $\cJ$ and $\cK$ on
the moduli space $\scrF$. Together with the metric $\cG$, they provide a
hyperKähler structure, with three corresponding Kähler forms on $\scrF$,
 denoted $\Omega_I, \Omega_J$ and $\Omega_K$. We also denote by $\scrR_I$,
 $\scrR_J$ and $\scrR_K$ the three involutions of $\scrF$ induced by the
 almost complex structures.

By Proposition~\ref{prop:presk}, the action of $\TT$ on $\scrF$ preserves
the hyperKähler structure. By Theorem~\ref{theo:ham} the action of $\TT$ is
Hamiltonian with respect to the three symplectic forms and the moment maps are given by
\begin{equation}\label{eq:mmhk}
	\mu_\bullet(F)= -\frac12 g(\scrR_\bullet F,F)
\end{equation}
or, equivalently by Proposition~\eqref{prop:density}
\begin{equation}\label{eq:densities}
	\mu_\bullet(F)= -\frac{(F^*\omega_V) \wedge \hat
	\omega_\bullet}{\vol_M}
\end{equation}
for $\bullet = I, J, K$.
These identities can  be gathered into a single one
$$
(F^*\omega_V)^- = \frac 12\sum_\bullet \mu_\bullet (F)\hat\omega_\bullet
$$
where $\beta^-$ denotes the anti-selfdual component of a differential
$2$-form.
Indeed, the forms $\hat\omega_\bullet$ provide an orthogonal frame for the space of
anti-selfdual forms and they satisfy $|\hat\omega_\bullet|^2=2$.
We check that 
$$
(F^*\omega_V)^-\wedge \hat\omega_\bullet =
F^*\omega_V\wedge\hat\omega_\bullet=-\mu_\bullet(F) \vol_M
$$
and introduce the \emph{hyperKähler moment map}
\begin{align}
	\nonumber \mub:\scrF &\to \LieTT^3\simeq C^\infty(M,\RR)^3\\
	\label{eq:mmb}
	F& \mapsto  (\mu_I(F),\mu_J(F),\mu_K(F)).
\end{align}

\begin{rmk}
	\label{rmk:mutilde}
The forms $\hat\omega_\bullet$ are anti-selfdual and provide an isomorphism
	\begin{align}
		\label{eq:xi}
		\xi :\LieTT^3 & \to \Omega^{2,-}(M)\\
	\nonumber	(\zeta_I,\zeta_J,\zeta_K) &\mapsto \frac 1{\sqrt 2}\sum_\bullet
	\zeta_\bullet\hat\omega_\bullet,
	\end{align}
	which respects the metrics.
By definition
	\begin{equation}
		\label{eq:asd}
		\sqrt 2	(F^*\omega_V)^- =  \xi(\mu_I(F),\mu_J(F),\mu_K(F)) 
	\end{equation}
and it makes sense to interpret the hyperKähler moment map as a map
\begin{align}
	\nonumber \tilde\mub:\scrF &\to \Omega^{2,-}(M)\\
	\label{eq:mmbtilde}
	F& \mapsto  \sqrt 2 (F^*\omega_V)^-
\end{align}
related to $\mub$ by the identity
$$
\tilde\mub =\xi\circ\mub.
$$

\end{rmk}

\begin{prop}\label{lemma:sd}
For every $F\in\scrF$, the following properties are equivalent:
	\begin{enumerate}
		\item $F^*\omega_V$ is selfdual;
		\item $ \mub(F)=0$.
	\end{enumerate}
\end{prop}
\begin{proof}
	The equivalence is an immediate consequence of
	Formula \eqref{eq:asd}. 
\end{proof}
In the case where $F =\scrD f$, we have an interesting interpretation of the
moment map which is the leitmotif to this paper:
\begin{prop}
	\label{prop:symplecto}
	Let $f\in \scrM$ be a smooth map. 
	 Then the
	following properties are equivalent:
	\begin{enumerate}
		\item $\mub(\scrD f)=0$;
		\item $f^*\omega_M$ is a selfdual harmonic form.
	\end{enumerate}
	In particular, if $f^*[\omega_M]=[\omega_M]$, then $f$ is a
	symplectomorphism if, and only if, $\mub(\scrD f)=0$, which is to
	say:
$$
	\Symp(M,\omega_M)= \{f\in\scrM, f^*[\omega_M]=[\omega_M] \mbox{ and
	} \mub\circ \scrD (f) =0\}.
$$
\end{prop}
\begin{proof}
	Assume that $f\in\scrM$ and $\mub(\scrD f)=0$. 
By Proposition~\ref{lemma:sd}, $\mub(\scrD f)=0$ is equivalent to the fact that
	$F^*\omega_V$ is a selfdual form, where $F=\scrD f$. Since
	$F^*\omega_V=f^*\omega_M$, we deduce that $f^*\omega_M$ is selfdual and
	closed. In particular, it is harmonic and the first part of the proposition
	follows.

	If $f^*[\omega_M]=[\omega_M]$, then
	$f^*\omega_M - \omega_M$ is an exact form. 
	The Kähler
	form $\omega_M$ is selfdual.
If $f^*\omega_M$ is selfdual as well, then $f^*\omega_M-\omega_M$ is an
	exact selfdual
	 form, which implies that it vanishes identically.
	  We conclude that $f$ is a
	symplectic map.
\end{proof}
\section{HyperKähler flow}
\label{sec:hkflow}
\subsection{Prescribed cohomology classes and Hodge projector}
The moduli space of $\vecV$-valued differential $1$-forms $\scrF$ contains
the subspace of closed forms:
$$
\scrF_c=\{F\in\scrF, dF=0\}.
$$
Let $\alpha\in H^1(M,\vecV)\setminus 0$ be a fixed  cohomology class. We introduce
the subspace of  closed forms $\scrF_\alpha$  with cohomology class contained
in the line~$\RR\alpha$:
$$
\scrF_\alpha =\{F\in \scrF, dF=0 \mbox{ and }  [F]\in
\RR\alpha\}\subset\scrF_c\subset\scrF^.
$$
It is a standard fact that Hodge theory extends to $\vecV$-valued forms. 
The  formal adjoint of $d$ with respect to the $L^2$-metric
$\cG$ is denoted $d^\star$ and  the Laplacian operator is defined by $\Delta = dd^\star+d^\star d$.
Hodge theory provides an orthogonal projection with respect to the metric~$\cG$
$$
\Pi_c: \scrF \to \scrF_c.
$$
Indeed, any $1$-form $F\in \scrF$ admits an orthogonal decomposition
$$
F = F_{h} + d\beta + d^\star b,
$$
where $F_{h}$ is harmonic, $\beta \in \Omega^0(M,\vecV)$ and $b \in
\Omega^2(M,\vecV)$. 
Then, the projection $\Pi_c:\scrF\to \scrF_c$ onto the closed component
is given by
$$
\Pi_c (F) =  F_{h} + d\beta.
$$
Cohomology classes are represented by a unique harmonic
form. This leads  to an orthogonal decomposition into harmonic components 
$$F_{h}= F_\alpha +
F_\alpha^\perp,$$
where $[F_\alpha]\in \RR\alpha$ and $F^\perp_\alpha$ is orthogonal to the
harmonic representative of $\alpha$. The 
orthogonal projection  
$$
\Pi_\alpha:\scrF\to\scrF_\alpha
$$
is then
defined by
$$
\Pi_\alpha(F)= F_\alpha + d\beta.
$$

\subsection{Energy of the moment map}
The \emph{energy functional} of the hyperKähler moment map
$$
\phi:\scrF \to \RR
$$
is  defined by
\begin{align}
	\label{eq:energy}
	\phi(F)&= \frac 12 \|\mub(F)\|_{L^2}^2\\
	\nonumber &= \frac 12\| \mu_I(F) \|_{L^2}^2 + \frac
	12 \| \mu_J(F) \|_ {L^2}^2+
	\frac 12\|\mu_K(F)\|^2_{L^2} .
\end{align}
The functional $\phi$ is non negative and its vanishing locus
agrees with the zero set the hyperKähler moment map
$$
\mub^{-1}(0)=\phi^{-1}(0).
$$
Furthermore, the vanishing set of $\phi$ agrees with the set of critical
values  according to the following
proposition:
\begin{prop}
	\label{prop:crit}
The critical points of $\phi:\scrF\to \RR$ are the zeroes of the
	hyperKähler moment map $\mub$. Similarly, the critical points of
	the restricted functional $\phi:\scrF_\alpha\to \RR$ are the zeroes
	of $\mub$ contained in $\scrF_\alpha$.
\end{prop}
The proof of Proposition~\ref{prop:crit}
 follows from the elementary computations
 carried out in the rest of this section.
The differential of $\phi$ is readily computed:
\begin{align}
	\label{eq:diff0}
	D\phi|_F\cdot \dot F = &  
\ipp{D\mu_I|_F\cdot \dot F, \mu_I(F) }\\ 
\nonumber	+&\ipp{D\mu_J|_F\cdot \dot F, \mu_J(F) }\\
\nonumber	+& \ipp{D\mu_K|_F\cdot \dot F, \mu_K(F)}
\end{align}
where $\ipp{\cdot,\cdot}$ denotes the $L^2$ inner product on $\LieTT$.
By Formula~\eqref{eq:diffmu},
$$D\mu_I|_F\cdot \dot F = - g(\scrR_I F,\dot F),$$
hence
$$
\mu_I(F) D\mu_I|_F\cdot \dot F = -g(\mu_I(F)\scrR_IF,\dot F)
$$
and more generally
$$
\mu_\bullet(F) D\mu_\bullet|_F\cdot \dot F = -g(\mu_\bullet(F)\scrR_\bullet
F,\dot F).
$$
Integrating the above formula, we have
\begin{align}
	\ipp{D\mu_\bullet|_F\cdot \dot F,\mu_\bullet (F)} & = -\cG(\mu_\bullet
	(F)\scrR_\bullet F,\dot F) \label{eq:diff1}\\
	& =
-\Omega_\bullet (\mu_\bullet(F)  i F,\dot F) \\
	& =-
\iota_{X_{\mu_\bullet(F)}(F)}\Omega_\bullet(\dot F).
\end{align}
where $X_{\mu_\bullet(F)}$ is the vector field corresponding to the infinitesimal
action of~$\mu_\bullet(F)\in \LieTT$ on~$\scrF$.
\begin{lemma}
	\label{lemma:formphi}
	For every $F, \dot F\in\scrF$, we have
	\begin{equation}\label{eq:diffphi}
		D\phi|_F\cdot \dot F = \ipp{\sum_\bullet W_\bullet(F), \dot
		F}
	\end{equation}
where 
\begin{equation}\label{eq:w}
W_\bullet(F) = -\mu_\bullet(F)\scrR_\bullet F .
\end{equation}
In particular

	\begin{equation}
		\label{eq:key}
			D\phi|_F\cdot  F = 4\phi(F).
	\end{equation}
\end{lemma}
\begin{proof}
	Formulas~\eqref{eq:diffphi} and \eqref{eq:w} are deduced from~\eqref{eq:diff0} and~\eqref{eq:diff1}. The
	identity~\eqref{eq:key} is a simple consequence: 
	$$
	D\phi|_F\cdot F = -\sum_\bullet\cG(\mu_\bullet(F)\scrR_\bullet F,F)
	= -\sum_\bullet \int g(\mu_\bullet(F)\scrR_\bullet F,F)\vol_M
	$$
	Since $g(\mu_\bullet(F)\scrR_\bullet F,F)=
	\mu_\bullet(F)g(\scrR_\bullet F,F)=-2\mu_\bullet(F)^2$, we deduce
	that $D\phi|_F\cdot F= 2\sum _\bullet \|\mu_\bullet(F)\|^2 =
	4\phi(F)$.
	\end{proof}
\begin{proof}[Proof of Proposition~\ref{prop:crit}]
	If $F\in\scrF$ satisfies $\mub(F)=0$, then $\mu_\bullet(F)=0$ hence $W_\bullet(F)=0$ and
	$D\phi|_F=0$. 
	Conversely, if $F$ is a critical point of $\phi:\scrF\to \RR$, then
	$4\phi(F)= D\phi|_F\cdot F= 0$, by Formula~\eqref{eq:key}, which
	implies $\mub(F)=0$. 

	If $F$ is a critical point of the restriction
	$\phi:\scrF_\alpha\to\RR$, then $D\phi|_F\cdot F$ vanishes and the
	same proof shows that $\mub(F)=0$.
\end{proof}

As a direct consequence of
	Formula~\eqref{eq:diffphi} we obtain the following result:
\begin{cor}
	\label{cor:gradient}
The gradient of $\phi:\scrF\to \RR$ at $F$ is given by the formula
$$
	\nabla \phi(F)= \sum W_\bullet(F).
$$
The gradient $\nabla^\alpha\phi$ of the restriction $\phi:\scrF_\alpha\to\RR$ at $F\in \scrF_\alpha$ is given by the
	projection 
	$$\nabla^\alpha\phi(F)=\Pi_\alpha(\nabla\phi(F)).
	$$
	For every $F\in\scrF$, we have 
	$$\ipp{\nabla\phi(F),F}=4\phi(F)$$
	and, similarly, for every  $F\in\scrF_\alpha$
	$$\ipp{\nabla^\alpha\phi(F),F}=\ipp{\Pi_\alpha \nabla\phi(F),F}=4\phi(F).$$
\end{cor}

\subsection{Modified moment map flow}
The downward gradient flow of the moment map energy functional
$\phi:\scrF\to\RR$  is  the classical \emph{hyperKähler moment map flow} used in many
gauge theoretic settings. However the subspace
$\scrF_\alpha \subset \scrF$ is not invariant under the gauge group and the
flow may not preserve this subspace. This is problematic, since we are mostly
interested in closed differential, with Proposition~\ref{prop:symplecto} in mind.
We get aroung this issue  by considering the
flow of the restricted functional $\phi:\scrF_\alpha\to \RR$, defined
 by
\begin{equation}\label{eq:flow1}
	\boxed{\frac{\del F}{\del t} = - \Pi_\alpha\nabla\phi(F).}
\end{equation}
The above flow is called the \emph{modified hyperKähler moment map flow}.

\begin{prop}
	\label{prop:zeroes}
	The following conditions are equivalent for $F\in\scrF_\alpha$:
	\begin{enumerate}
		\item $F$ is a zero of $\mub$
		\item $F$ is a zero of $\phi$
		\item $F$ is a critical point of the functional
			$\phi:\scrF_\alpha\to\RR$
		\item $F$ is a fixed point of the modified moment map
			flow.
	\end{enumerate}
\end{prop}
\begin{proof}
	The proof is obvious since fixed points are by definition the
	critical point of the restricted functional $\phi:\scrF_\alpha\to
	\RR$, which were identified at Proposition~\ref{prop:crit}.
\end{proof}
An essential feature of the modified moment map flow is the decay property
of the $L^2$ norm:
\begin{prop}
	\label{prop:decay}
	Let $F_t\in\scrF_\alpha$ be a solution of the  modified  moment
	map  flow, for $t$ in some interval $I$. Then 
	$$
	\frac\del{\del t}\|F_t\|^2_{L^2} = -8\phi(F_t),
	$$
	for every $t\in I$.
	In particular
	$t\mapsto \|F_t\|_{L^2}$ is a non increasing function on $I$.
\end{prop}
\begin{proof}
We compute
$$
	\frac{\del\|F_t\|^2}{\del t} = 2\ipp{ \frac {\del F_t}{\del
	t},F_t} =
	-2\ipp{\Pi_\alpha \nabla\phi(F_t),F_t}=-2D\phi|_{F_t}\cdot F_t.
$$
	an the result follows by Equation~\eqref{eq:key}.
\end{proof}
\begin{rmk}
	In the case of an ordinary differential equation,
	Proposition~\ref{prop:decay} would guaranty the long time existence
	of the flow. An analogue of this proposition holds indeed in the
	polyhedral setting (cf. Proposition~\ref{prop:formpol}) and
	insures the long time existence of the flow
	 (cf. Corollary~\ref{cor:limit}).
\end{rmk}
\subsection{Short time existence}
The modified hyperKähler moment map flow has been considered so far
from a purely formal perspective. We are going to show that the
Cauchy-Lipschitz theorem  applies, once suitable Hölder spaces have been
introduced.

We consider the $C^{k,\nu}$-Hölder norms, denoted $\|\cdot\|_{k,\nu}$, defined 
via the metrics $g_M$ and $g$ on the spaces of functions and tensors on the $4$-torus~$M$, where $\nu\in (0,1)$ is the Hölder
regularity exponent and $k$ is the number of derivatives controlled by
the norm.  The spaces of smooth functions and differential forms can be
completed into Banach spaces with respect to these Hölder norms. 
For example, we
denote by $\scrF^{k,\nu}$ the completion of the  space
of smooth differential forms $\scrF$  with respect to the $C^{k,\nu}$-Hölder
norm. 

If $k\geq 2$, the exterior derivative $d$, its adjoint $d^\star$ and  the
Laplacian operator $\Delta=dd^\star+d^\star d$ are 
defined  on $\scrF^{k,\nu}$. The  Hodge decomposition theory holds and, in
particular, we may 
consider the subspace of closed forms
$\scrF_c ^{k,\nu}$. We can also consider the subspace of closed forms with cohomology class in
$\RR\alpha$, denoted $\scrF_\alpha^{k,\nu}$ and the orthogonal projection
$\Pi_\alpha:
\scrF^{k,\nu}\to \scrF_\alpha^{k,\nu}$ (cf. Lemma~\ref{lemma:hodge}), which extends the projection
defined at \S\ref{sec:cohom} in the smooth settings. 

\begin{theo}
	\label{theo:cauchy}
	Let $k\geq 2$ and $\nu\in(0,1)$ be some Hölder exponents. For every
	$\alpha\in H^1(M,\vecV)\setminus 0$ and
	$F\in\scrF_\alpha^{k,\nu}$, 
	there exists  $\epsilon >0$ and a unique differentiable map
	\begin{align*}
		[-\epsilon,\epsilon]& \to \scrF_\alpha^{k,\nu} \\
		t&\mapsto F_t
	\end{align*}
	such that 
	\begin{itemize}
	\item  We have the initial condition $F_0=F$ and
	\item the map is a 
		solution of the modified hyperKähler  moment map flow. 
	\end{itemize}
\end{theo}
Theorem~\ref{theo:cauchy} follows immediately from the
Cauchy-Lipschitz theorem.
 A key step to prove the Lipschitz condition is the following classical result of
Hodge theory:
	\begin{lemma}\label{lemma:hodge}
		For $k\geq 2$, the projection
		$\Pi_\alpha:\scrF^{k,\nu}\to\scrF_\alpha^{k,\nu}$ defines a continuous map 
		with respect to the $C^{k,\nu}$-Hölder norm. In
	other words, there exists a constant $C>0$, such that for every
		$F\in\scrF^{k,\nu}$
$$
	\|\Pi_\alpha F\|_{k,\nu} \leq C\|F\|_{k,\nu}.
$$
	\end{lemma}
	\begin{rmk}
		\label{rmk:hodge}
		More generally,  Hodge projectors are
		all continuous provided suitable Hölder spaces are
		introduced. For example, the projection of differential
		$p$-forms onto their closed, co-closed or harmonic components
		are continuous as well. However we include a proof of the
		above theorem for clarity.
	\end{rmk}
	\begin{proof}[Proof of Lemma~\ref{lemma:hodge}]
		For $k\geq 2$, the  classical Hodge theory shows that
		every $1$-form
		$F\in\scrF ^{k,\nu}$ admits a $\cG$-orthogonal decomposition
$$
		F= F_{h} + \Delta G,
$$
		where $G\in \scrF^{k+2,\nu}$ is a $1$-form orthogonal to harmonic forms and
		$F_{h}$ is the
	harmonic component of $F$.
	A $C^{k,\nu}$-estimate on $F$
provides a	 $C^{k-2,\nu}$-estimate for $\Delta F$.
		Taking the Laplacian of both sides of the above identity
		gives $\Delta F = \Delta^2 G$, since $F_{h}$ is harmonic.
	The $C^{k,\nu}$-norm of
	$F$ controls the $C^{k-2,\nu}$ norm of $\Delta^2 G$. The operator
	$\Delta$ is selfajoint, hence $\Delta G$
	is orthogonal to the kernel of $\Delta$. The Schauder
		estimates for the elliptic operator $\Delta$ provide a control on the
		$C^{k,\nu}$-norm
		of $\Delta G$. Since $G$ was chosen
	orthogonal to harmonic forms, we deduce a $C^{k+2,\nu}$ control on
		$G$, by the Schauder estimates.
	In conclusion, there exists a constant $c_1>0$, independent of $F$, such that
$$
	\|G\|_{{k+2,\nu}}\leq c_1\|F\|_{{k,\nu}}.
$$
In particular 
		\begin{equation}
			\label{eq:ell0}
	\|dd^\star G\|_{{k,\nu}}\leq c_2\|F\|_{{k,\nu}},
		\end{equation}
for some constant $c_2>0$, independent of $F$.

		It turns out that there
		exists a universal constant $c_3>0$, such that
		\begin{equation}
			\label{eq:ell1}
			\|F_{h}\|_{L^2}\leq c_3\|F\|_{{k,\nu}}.
		\end{equation}
		The proof goes by contradiction.
If this is not true, we can find a sequence $F_j$ of differential
		forms such that 
		$$
		\lim_{j\to\infty} \|F_j\|_{{k,\nu}} =0 \quad \mbox{ and
		}\quad
		\|F_{j,h}\|_{L^2}=1.
$$
Since the space of harmonic form is finite dimensional, we may assume, up
		to extracting a subsequence, that $F_{j,h}$ converges to
		a harmonic form $F_{h}$ with 
		$\|F_{h}\|_{L^2}=1$. In particular, $F_h$ is non
		vanishing.
Computing the $L^2$ inner product, we have
$$
		\ipp{F_j, F_{h}}= \ipp{F_{j,h},F_{h}}.
$$
In particular 
$$\lim \ipp{F_j,F_{h}}= \|F_{h}\|^2_{L^2}=1.$$
By the Cauchy-Schwartz inequality, $ \ipp{F_j,F_{h}}\leq
\|F_j\|_{L^2}\|F_h\|_{L^2}$,
but $\|F_j\|_{L^2}\leq \|F_j\|_{k,\nu}\vol(M)$ hence $\lim \|F_j\|_{L^2}=0$.
We conclude that  
$$\lim \ipp{F_j,F_{h}}= 0$$
 which is a
contradiction and estimate \eqref{eq:ell1} is proved.

The orthogonal decomposition into harmonic components $F_{h}= F_\alpha + F_\alpha^\perp$
gives the estimate
$$
\|F_ \alpha\|_{L^2}\leq \|F_h\|_{L_2}.
$$
Together with estimate \eqref{eq:ell1}, we deduce that 
\begin{equation}
	\label{eq:ell2}
			\|F_{\alpha}\|_{L^2}\leq c_3\|F\|_{{k,\nu}}.
\end{equation}

All norms are equivalent on a finite dimensional space, in particular on
the space of harmonic forms. Therefore, there
		exists a universal constant $c_4>0$ such that for every
		harmonic form $\beta_h$
$$
		c_4 \|\beta_h\|_{{k,\nu}}\leq \|\beta_h\|_{L^2}.
$$
Together with estimate \eqref{eq:ell2}, we deduce that 
		\begin{equation}
			\label{eq:ell3}
			 \|F_{\alpha}\|_{k,\nu}\leq c_5\|F\|_{{k,\nu}}
		\end{equation}
		where $c_5= \frac{c_3}{c_4}$.

Finally $F= F_\alpha +F_\alpha^\perp +dd^\star G+ d^\star dG$, and
$\Pi_\alpha F = F_\alpha + dd^\star G$ so that
$$
\|\Pi_\alpha F\|_{k,\nu}\leq \|F_\alpha\|_{k,\nu}+  \|dd^\star
G\|_{k,\nu}$$
and the lemma follows form estimates \eqref{eq:ell3} and \eqref{eq:ell0}.
	\end{proof}

\begin{proof}[Proof of Theorem~\ref{theo:cauchy}]
	The map of $F\mapsto \nabla\phi(F)$ is  polynomial of order~$3$ 
	with respect to the coefficients of
	$F$ by Corollary~\ref{cor:gradient}. 
	Therefore, the map $F\mapsto \nabla\phi(F)$ is 
	locally Lipschitzian  with respect to the
	$C^{k,\nu}$-Hölder
	norm. The projector $\Pi_\alpha$ is continuous by
	Lemma~\ref{lemma:hodge}, hence the composition $F\mapsto \Pi_\alpha\nabla\phi(F)$ is locally
	Lipschitzian as well. The standard proof of existence of solution
	of an evolution equation rests on the existence of fixed points of a
	functional
	$$
	\Upsilon : \cB\to\cB
	$$
	where 
	 $\cB$  is the Banach space of continuous maps 
	 $$C^0([-\epsilon,\epsilon],
	\scrF_\alpha^{k,\nu}),$$ 
	endowed with the norm
$$
\|F_t\|_\infty = \sup_{t\in [-\epsilon,\epsilon]}\|F_t\|_{k,\nu}.
$$
More precisely, $\Upsilon$ is defined for each  $F\in\cB$,
	by
$$
	\Upsilon(F)_t = F_0 +\int_0^t\Pi_\alpha\nabla\phi(F_s)ds
$$
	For $\epsilon>0$  sufficiently small, the functional $\Upsilon$ is
	a locally $\frac12$-contractant map of Banach space for $(\cB,\|\cdot
	\|_\infty)$
and  the Banach fixed point theorem applies.
\end{proof}

\section{Renormalized flow and real blowup}
\label{sec:renflow}
\subsection{Homotopy of the symplectomorphism group}
\label{sec:motivret}
We give a sketch of a potential application of the modified moment map
flow,
for investigating the homotopy type of $\Symp_\alpha(M,\omega_M)$.  Let $B^{n+1}$ be a
closed Euclidean ball of dimension $n+1$ and $S^n$ be its boundary sphere.
Let 
$$h:S^n\to\Symp_\alpha(M,\omega_M)$$
be a continuous map, with respect to
the $C^1$-topology on $\scrM_\alpha$. We denote  $h(\tau)$ by $h_\tau$
as usual, for $\tau\in S^n$.
We choose a marked point $x_0\in M$ and assume, for simplicity, that
$h_\tau(x_0)=x_0$ for every $\tau\in S^n$. 

The map $H:S^n \to \scrF_\alpha$, defined by $H_\tau=\scrD h_\tau$, is 
continuous  with respect to the $C^0$-topology. The cohomology  class $\alpha$
is can be undertood as an affine subspace of $\scrF_\alpha$. In particular this space is contractible,
 hence $H$ extends as a continuous homotopy $H:B^{n+1}\to
\alpha$.
Since $\alpha$ is an integral cohomology class, there exists a unique  family of maps 
$h:B^{n+1}\to \scrM_\alpha$ which extend the map
$h:S^n\to\Symp_\alpha(M,\omega_M)$, with the property that
$h_\tau(x_0)=x_0$ and $\scrD h_\tau =H_\tau$ for every $\tau\in B^{n+1}$.
This shows that the initial homotopy $h:S^n\to \Symp_\alpha(M,\omega_M)$
is trivial in $\scrM_\alpha$. We would like to know whether it is also trivial
in $\Symp_\alpha(M,\omega_M)$. 

The idea is to try to push back the
homotopy in the symplectic group via the modified moment map flow.
For every $\tau\in S^n$, the point $H_\tau$ is a fixed point of the
modified moment map flow, by assumption. We now make some extremely strong
hypothesis: 
\begin{enumerate}
	\item[(H1)] We assume that, for every $\tau\in B^ {n+1}$, the
modified moment map flow $F^\tau_t$, with initial condition
		$F^\tau_0=H_\tau$, exists
for every $t\geq 0$ and converges toward a limit $\tilde H_\tau$ as $t\to
+\infty$.
		\item[(H2)] We suppose that the limit $\tilde H_\tau$ belongs to the vanishing set of $\phi$ 
and depends continuously on the initial
condition. 
\item[(H3)] We assume that the limit $\tilde H_\tau$ is different from $F=0$, for every $\tau\in
	B^{n+1}$.
\end{enumerate}
Properties $(H1)$ and $(H2)$ are conjectural 
in the smooth case. However they  hold in the polyhedral context, as proved by
Theorem~\ref{theo:duistermaat}.
Under these hypothesis, we have a deformed $H$ into continuous map 
$ \tilde H: B^{n+1}\to
\scrF_\tau$, wich agrees with $H$ along $S^n$ and takes values in the
vanishing set of $\phi$.
If the hypothesis $(3)$ holds as well, then $\tilde H_\tau$ belongs to
$\scrF_\alpha\setminus 0$. By Lemma~\ref{lemma:nonvanish} we can deduce
that the cohomology
class $[\tilde H_\tau]$ is non zero and we can write 
$$
[\tilde H_\tau]= \kappa(\tau)\alpha
$$
for some non vanishing continuous function $\kappa:B^{n+1}\to \RR$.
Furthermore, $\kappa =1$ along the boundary $S^n$. We consider the rescaled
family $\hat H_\tau = \kappa^{-1}(\tau)\tilde H_\tau$. By definition $[\hat
H_\tau]=\alpha$
 and the exists a unique family of
maps $f_\tau\in\scrM_\alpha$ such that $\scrD f_\tau =\hat H_\tau$ and
$f_\tau(x_0)=x_0$. By constructions the maps $f_\tau$ belong to
$\Symp_\alpha(M,\omega_M)$ by Theorem~\ref{theo:eqmu} and $h_\tau=f_\tau$ for every $\tau\in S^n$. 
We conclude that the map $h:S^n\to\Symp_\alpha(M,\omega_M)$ is
homotopically trivial in $\Symp_\alpha(M,\omega_M)$.

If hypothesis $(H1-H2-H3)$ hold for every map $h$ as above, it would readily follow that
the evaluation map $ev:\Symp_\alpha(M,\omega_M)\to M$ at $x_0$ is a homotopy
equivalence.
However this is too good to be true. In the polyhedral setting, 
we prove at Theorem~\ref{theo:polflow} 
that there always exists non trivial solutions of the polyhedral modified
moment map flow converging toward $F=0$. Hence, hypothesis $(H3)$ is
generally false in the polyhedral context and we shall see that this is directly
related to the existence of proper solitons. We are not able to prove
a similar result in the smooth setting, but it is likely that the same
phenomenon arise.
The conclusion of this discussion is that 
 proper solitons should play a role in the topological description of
the space of  symplectic maps of the torus and contribute, somehow, to its
complexity.

\subsection{Blownup moduli space}
For any cohomology class $\alpha\in H^1(M,\vecV)\setminus 0$, 
the sphere of radius $r>0$ in $\scrF_\alpha$ is defined by
$$
\SS_\alpha(r)=\{F\in \scrF_\alpha, \|F\|_{L^2}=r\}
$$
and the sphere is denoted $\SS_\alpha$ in the case $r=1$. The
\emph{real blowup} of the moduli space
$$p:\widehat{\scrF_\alpha}\to\scrF_\alpha,
$$
is defined  by 
$\widehat{\scrF_\alpha} = [0,+\infty ) \times \SS_\alpha$,
where the \emph{blowdown map}
	is given by $p(r,F)=rF$.
The real blowup contains an \emph{exceptional divisor}
$$\scrE=p^{-1}(0)=\{0\}\times\SS_\alpha\subset \widehat\scrF_\alpha,$$ 
such that the restriction of blowdown map  $p:\widehat
{\scrF}_\alpha\setminus \scrE\to \scrF_\alpha\setminus \{0\}$ is  bijective. 
One can think of the real blowup as a system of spherical coordinates
 for~$\scrF_\alpha$,
with an extra boundary component, corresponding to $r=0$.

\subsection{Solitons}
\label{sec:solitons}
A  critical
 points $F\in\SS_\alpha(r)$ of the functional $\phi:\scrF_\alpha\to \RR$ is also a critical point of the restricted
functional $\phi: {\SS_\alpha}(r)\to \RR$, but the converse is not
necessarily true. 
A differential form  $F\in\SS_\alpha(r)$  is a critical points of the restricted functional
$\phi:\SS_\alpha(r)\to\RR$ if, and only if, the gradient  of
$\phi:\scrF_\alpha\to \RR$ at $F$ is
orthogonal to the  sphere $\SS_\alpha(r)$. This property is equivalent to
have a radial gradient
$\Pi_\alpha\nabla\phi(F)=\kappa F$,
for some constant $\kappa\in\RR$. However by
Corollary~\ref{cor:gradient}, the identity $\ipp{\Pi_\alpha\nabla\phi(F),F}=
4\phi(F)$ implies $ \kappa r^2=
4\phi(F)$.  
We deduce that $F$ satisfies the following
equation, called the \emph{soliton equation}:
\begin{equation}
	\label{eq:soliton}
	\boxed{	\|F\|^2_{L^2} \Pi_\alpha\nabla\phi(F)=4\phi( F) F.}
\end{equation}

	\begin{dfn}
		\label{dfn:soliton}
		A  solution $F\in\scrF_\alpha$ of
		Equation~\eqref{eq:soliton} is called a soliton. Then
		either
		$\phi(F)\neq 0$, and we say that $F$ is a proper soliton,
		of $\phi(F)=0$ and  $F$ is called a non proper soliton..
\end{dfn}
 \begin{rmk}
The soliton equation~\eqref{eq:soliton} is \emph{real
homogeneous}.
	 Therefore, the space of solitons is a conical subspace  of $\scrF_\alpha$.
 \end{rmk}
By definiton, we have the following property:
\begin{lemma}
	A differential form  $F\in\scrF_\alpha$
	 is a soliton if, and only if, either $F=0$ or $F$ is a critical point of the restricted
	functional $\phi:\SS_\alpha(r)\to\RR$, where $r=\|F\|_{L^2}$.
\end{lemma}

Following the ideas of Atiyah-Bott~\cite{AB}, Donaldson~\cite{D} and
Kirwan~\cite{Kir,KFM} in the finite dimensional case, the
functional $\phi:\SS_\alpha\to\RR$ should be regarded as a Morse-Bott function. By definition, the
critical set of the restriced functional is the space of solitons in
$\SS_\alpha$. Furthermore, Morse-Bott homology theory provides important
topological information about the critical set.
Finally, the subspace of non proper solitons in $\SS_\alpha$ 
is directly related to 
 the topology of the the symplectic group as we are going to see.

\begin{lemma}
	\label{lemma:nonvanish}
If $F\in\scrF_\alpha$ is a non proper soliton with vanishing cohomology
	 class, then $F=0$.
\end{lemma}
\begin{proof}
Let $F$ be a non proper soliton with vanishing cohomology class. By
	definition $\phi(F)=0$, which is equivalent to
	$\tilde\mub(F)=(F^*\omega_V)^-=0$, which implies that $F^*\omega_V$ is selfdual. 
	Since $[F]=0$, there exists a map $f:M\to
	M$, homotopic to a constant map, such that $F=\scrD f$. Using the
	identity $F^*\omega_V= f^*\omega_M$, we conclude that $f^*\omega_M$ is
	selfdual. On the other hand $[f^*\omega_M]=0$,
	 since $f$ is homotopic to a constant map. A selfdual exact form must vanish, hence
	$f^*\omega_M=0$, which implies $f_*=0$ by the non degeneracy of
	$\omega_M$. Finally, $f$ is  a constant map, so that $\scrD f= F= 0$.
\end{proof}

\begin{cor}\label{cor:nonproper}
	The cone spanned by the action of $\RR^*$ and the image of
	$\Symp_\alpha(M,\omega_M)$ by $\scrD$
	agrees with the space of non proper solitons  in
	$\scrF_\alpha\setminus 0$.
	Equivalently, the map
	\begin{align*}
		\varpi:\Symp_\alpha(M,\omega_M)&\to \SS_\alpha\\
		f &\mapsto \frac{\scrD f}{\|f_*\|_{L^2}}
	\end{align*}
	induces a bijection between $\Symp_\alpha(M,\omega_M)/\vecV$ and
	the space of non proper solitons in $\SS_\alpha$, modulo the
	antipodal map of $\SS_\alpha$.
\end{cor}
\begin{rmk}
	The Corollary is stated informally, but it is possible to show that 
	$\varpi$ induces a diffeomorphism between moduli spaces, 
	 once suitable Hölder spaces have been introduced.
\end{rmk}
\begin{proof}
	Given $f\in\Symp_\alpha(M,\omega_M)$, we have $\mub(\scrD f)=0$
	hence $\phi(\scrD f)=0$, hence $\scrD f$ is a non proper soliton. 
	The fact that $f$ is a diffeomorphism implies that $\scrD f\neq 0$.

Conversely, let $G$ be a non
	proper soliton in $\scrF_\alpha\setminus 0$. By Lemma~\ref{lemma:nonvanish}, the
	cohomology class of $G$ is non zero. Hence there exist
	$\kappa\in\RR\setminus 0$, such that $\kappa [G]=\alpha$. In
	particular $F=\kappa G$ has integral cohomology class $\alpha$. It
	follows that there exists $f\in\scrM_\alpha$ such that $F=\scrD f$.
	By construction $G$ belongs to the ray spanned by $\scrD f$. The fact that $\phi(G)=0$
	implies $\phi(F)=0$, which shows that
	$f\in\Symp_\alpha(M,\omega_M)$. 
The second statement is an immediate consequence of the first.
\end{proof}
\subsection{Solitons and flow lines}
A soliton provides a natural solution of the modified moment map flow,
according to the following lemma.
\begin{lemma}
	\label{lemma:dicho}
	Let $G\in\scrF_\alpha \setminus 0$ be a soliton and $r:I\to
	[0,+\infty)$  be a 
	solution, of the ordinary differential equation
	$$
	\|G\|^2_{L^2}\frac{d r}{dt}= -4r^3\phi(G)
	$$
	defined on some interval $I$.
	 Explicitly, either
	 \begin{enumerate}
		 \item $G$ is not proper, then
	 $r(t)=r_0$ with $I=\RR$ for some constant $r_0\geq 0$.
 \item or $G$ is proper and
	$$r(t)=\frac{1}{\sqrt{8(t-t_0)\phi(G)}}
	$$ 
	for some constant $t_0\in\RR$ with $I=(t_0,+\infty)$.
	 \end{enumerate}
			 Then 
	$$F_t=p(r(t),G)=r(t)G\in\scrF_\alpha$$
	is a solution of the modified moment map flow with the following
	dichotomy: 
	\begin{enumerate}
		\item If $G$ is not proper, then $F_t$ is a
	stationary solution of the flow.
\item If $G$ is proper, then $F_t$ is not stationary and $F_t$
	converges toward~$0\in\scrF_\alpha$.
	\end{enumerate}
\end{lemma}
\begin{proof}
	Put $F_t =r(t)G$, where $G$ and $r$ satisfy the assumptions of
	the lemma.
	By definition 
	$$\frac{\del}{\del t} F_t = \frac{dr}{dt}
	G=-4\|G\|_{L^2}^{-2}r^3\phi(G)G.$$
	Since $G\in \SS_\alpha$ is a soliton, then
	$\|G\|_{L^2}^2\Pi_\alpha\nabla\phi(G)=4\phi(G)G$, hence
	$\frac{\del}{\del t}F_t= -r^3\Pi\nabla\phi(G) =
	-\Pi_\alpha\nabla\phi(F_t)$, which shows that $F_t$ is a flow line of the
	modified moment map.
\end{proof}

\subsection{Blownup flow}
In this section, we show that the modified moment map flow
admits a natural lift to the real blowup $\widehat{\scrF_\alpha}$.

From now on,  $F_t\in\scrF_\alpha$ is assumed to be a smooth solution of the modified moment map
flow, defined for $t \in [t_0,t_1)$,
with $t_1\in (t_0,+\infty]$.
The function 
$$r : [t_0,t_1)\to \RR$$ 
defined  by
$$
r(t)=\|F_t\|_{L^2}
$$
satisfies the ODE
$$
\frac{d r^2}{d t} = -8\phi(F_t)
$$
according to Proposition~\ref{prop:decay}.
In particular $r(t)$ is a non increasing function of $t\in [t_0,t_1)$.
If fact $r$ must be   decreasing unless it is constant. Indeed, if the
derivative of $r$ vanishes at some $t_2\in[t_0,t_1)$, then
$\phi(F_{t_2})=0$, which implies that  $F_{t_2}$ is a fixed point of
the modified moment map flow. By
the local uniqueness property of the Cauchy-Lipschitz
Theorem~\ref{theo:cauchy}, the flow
is static on the interval $[t_0,t_1)$. 
Similarly, $r>0$ unless the flow is static. Indeed, it $r(t_2)=0$ for some
$t_2\in[t_0,t_1)$, then $\frac{dr^2}{dt}$ vanishes at $t_2$ and the above
argument shows that the flow is static on the interval. Thus we have proved the following
lemma:
\begin{lemma}
	\label{lemma:nv}
	The function $r:[t_0,t_1)\to\RR$ is positive and decreasing,
	unless $F_t$ is a static solution of the modified  moment map flow 
	along the interval. 
\end{lemma}

Given a smooth solution $F_t$ of the modified moment map
flow defined for
 $t\in[t_0,t_1)$, we know that $F_t$ does not vanish, unless $F_t=0$ along
 the interval.
 Then, there exists a unique point $(r(t), G_t)\in
\widehat\scrF_\alpha$, such that $p(r(t),G_t)=F_t$, defined by  
$$r(t)=\|F_t \|_{L^2},\quad G_t=\frac {F_t}{r(t)}.$$

We are going to see that $(r, G)$ is  solution of an evolution equation
defined on the real blowup. First 
 $2r\frac{d r}{d t} = -8\phi(F) = -8\phi(rG)= -8r^4\phi(G)$, so that
\begin{equation}\label{eq:Ren1}
	\frac{d r }{d t} = -4r^3\phi(G).
\end{equation}
By the chain rule, we have
\begin{equation}\label{eq:Rc1}
\frac {\del F}{\del t}= \frac{d r}{d t} G + r\frac{\del G}{\del
t}= -4r^3\phi(G)G + r \frac{\del G}{\del t}.
\end{equation}
Since $F$ is a flow line, we have
\begin{equation}\label{eq:Rc2}
\frac{\del F}{\del t} = -\Pi_\alpha \nabla\phi(F) = -r^3\Pi_\alpha\nabla\phi(G).
\end{equation}
By Equations~\eqref{eq:Rc1} and~\eqref{eq:Rc2}, we deduce that
\begin{equation}\label{eq:Ren2}
r^{-2} \frac{\del G}{\del t} = 4\phi(G)G -\Pi_\alpha\nabla\phi(G).
\end{equation}
In conclusion, we have the following proposition
\begin{prop}
	\label{prop:lifteq}
	Let $(r(t),G_t)\in\widehat\scrF_\alpha$ be a  path defined
	for $t\in[t_0,t_1)$, such that $F_t=r(t)G_t$ is solution of the
	modified moment map flow.
	Then $(r,G)$ is solution of the system of differential equations 
	\begin{align}
		\frac{\del r }{\del t} &= -4r^3\phi(G)\\
		 \frac{\del G}{\del t} &=r^2\left ( 4\phi(G)G
		 -\Pi_\alpha\nabla\phi(G)\right)
	\end{align}
	for $t\in[t_0,t_1)$.
\end{prop}

\subsection{Renormalized flow}
Let $(r(t),G_t)\in\widehat \scrF_\alpha$ be
a solution of the lifted flow as above. We introduce a
function $s:[t_0,t_1)\to \RR$, solution of the ODE
\begin{equation}
	\boxed{	\label{eq:rep}
	\frac{d s}{d t} = r^{2}.}
\end{equation}
If $F_t$ does not vanish, then $r>0$ and $s$ is strictly increasing.
Thus, $s$ can be used as a reparametrization of the evolution
equation.
In particular $r^{-2}= \frac {d t}{d s}$ so that
$$
r^{-2}\frac{\del G}{\del t } = \frac{d t}{d s} \frac{\del G}{\del t}
= \frac{\del G}{\del s}.$$
We deduce the differential system of equations, 
\begin{equation}
	\boxed{
		\frac{\del G}{\del s}= 4\phi(G)G-\Pi_\alpha \nabla\phi(G)\label{eq:RenA}}
\end{equation}
and
\begin{equation}
	\boxed{
		\frac{d r }{d s}=-4 r\phi(G)\label{eq:RenB}}
\end{equation}
Equation~\eqref{eq:RenA}, which is the downward gradient flow of the
restricted functional $\phi:\SS_\alpha\to\RR$, is called the \emph{renormalized
flow}.
The ODE given by Equation~\eqref{eq:RenB}
and Equation~\eqref{eq:rep} are used to pass from  the renormalized
 flow to the modified moment map flow  and vice versa:
let $G_s$ be a solution of the renormalized flow for $s$ in some interval
$I$. Then we can solve the ODE \eqref{eq:RenB} on the interval $I$. The
general solution has the form $r(s)=r_0e^{-4 \tau(s)}$, where $\tau$ is an
integral of $\phi(G_s)$ on $I$ and $r_0\geq 0$, by convention. If $r_0=0$,
then $r=0$ and we obtain $F=rG=0$.
If $r_0>0$, we have $r>0$ and we can solve the ODE~$\frac{d t}{d s} = \frac
1{r^{2}}= \frac 1{r_0^2}e^{8\tau(s)}$
	deduced from Equation~\eqref{eq:rep}.
	Hence
	$t(s)=\frac 1{r_0^2}\int e^{8\tau(s)}ds$
	is defined on $I$ and can be used as a reparametrization to obtain a solution $F_t=rG$ of
	the modified moment map flow.
From our discussion, we deduce the following proposition:
\begin{prop}
	\label{prop:reneq}
 The space of fixed points of the evolution
	equations given by the
	renormalized flow~\eqref{eq:RenA} and Equation~\eqref{eq:RenB}
consists of pairs  $(r,G)\in\widehat\scrF_\alpha$  such that $G$ is a
	soliton of $\SS_\alpha$ and $r\geq 0$. If $r>0$ then $G$ is a non
	proper soliton. In
	particular, the only fixed points $(r,G)$ such that $G$ is a proper
	soliton belong to the exceptional divisor $\scrE$ of the real blowup.
\end{prop}

In the spirit of Morse-Bott theory, we show that the downward gradient flow
of $\phi:\SS_\alpha\to \RR$, which is the renormalized flow, has short time
existence:
\begin{theo}
	\label{theo:str}
	Let $\alpha$ be a  cohomology class in $H^1(M,\vecV)\setminus 0$
	and $(k,\nu)$ some Hölder parameters with $k\geq 2$. For every
	$G\in\scrF_\alpha^{k,\nu}$, such that $\|G\|_{L^2}=1$, 
	there exists $\epsilon>0$ and
	a continuous map $G_s\in \scrF_\alpha^{k,\nu}$ defined for
	$s\in[-\epsilon,\epsilon]$, which is a solution of the renormalized
	flow and such that $G_0=G$. Furthermore, this solution is unique on
	the interval.
\end{theo}
\begin{proof}
	We notice that the term $4\phi(G)G$ is polynomial in the
	coefficients of $G$.
	The proof relies on the Cauchy-Lipschitz theorem applied to
	Equation~\eqref{eq:RenA} and follows the same argument as in
	Theorem~\ref{theo:cauchy}. 
\end{proof}

\subsection{Proofs of the main theorems for the smooth case}
In this section, we collect our results to prove the main statement of the
introduction.
\begin{proof}[Proof of Theorem~\ref{theo:A}]
	The canonical Kähler structure and the
	conjugate hyperKähler structure on a quotient torus is described 
	at \S\ref{sec:hktorus}.
The resulting hyperKähler structure on the moduli space $\scrF$ follows
	from Proposition~\ref{prop:strk} and the hyperKähler moment map is
	introduced at \S\ref{sec:hkmap}. The theorem then follows from
	Lemma~\ref{lemma:sd} and Formula~\eqref{eq:densities}.
\end{proof}
\begin{proof}[Proof of Theorem~\ref{theo:eqmu}]
	The theorem  is a restatement of Proposition~\ref{prop:symplecto}.
\end{proof}
\begin{proof}[Proof of Theorem~\ref{theo:C}]
 	The fixed points of the modified moment map flow are the zeroes of
	$\phi$ by Proposition~\ref{prop:zeroes}.
	The $L^2$ decay properties is a consequence of
	Proposition~\ref{prop:decay} and Theorem~\ref{theo:cauchy} shows
	the short time existence of the flow.
\end{proof}
\begin{proof}[Proof of Theorem~\ref{theo:D}]
	This is a restatement of Theorem~\ref{theo:str}.
\end{proof}

\section{Rigidity and solitons}
\label{sec:rigidsolitons}
The goal of this section is to give a proof of
Theorem~\ref{theo:connected} and, more specifically, the stronger result
given at Theorem~\ref{theo:connectedsol}.
\subsection{Main result}
Given  a symplectic cohomology class $\alpha$, we denote by 
$$\scrS^{k,\nu}_\alpha\subset \scrF_\alpha^{k,\nu}\setminus 0$$
the space of non
zero solitons with Hölder regularity $(k,\nu)$, for some $k\geq 2$. 
 We denote by define 
 $\scrS^{k,\nu}_{\alpha,p}$  the subspace of proper solitons and by
 $\scrS^{k,\nu}_{\alpha,np}$ the subspace of non proper solitons of
 $\scrS_\alpha^{k,\nu}$. Then we  have the obvious  partition
 $$
\scrS^{k,\nu}_\alpha=\scrS_{\alpha,p}^{k,\nu}\sqcup\scrS_{\alpha,np}^{k,\nu}.
 $$
It turns out that the  components of the above decomposition
are also topologically separated for $k\geq 3$, in the sense of the following theorem:
\begin{theo}
	\label{theo:connectedsol}
	Let $\alpha\in H^1(M,\vecV)$ be a symplectic cohomology class and  $k\geq 3$. Then, 
	the spaces of solitons $\scrS_{\alpha,p}^{k,\nu}$ and $\scrS_{\alpha,np}^{k,\nu}$ are open and closed
	subspaces of $\scrS^{k,\nu}_\alpha$,
	endowed the the Hölder $C^{k,\nu}$-topology. 

In particular, if $\psi:X\to\scrS_\alpha^{k,\nu}$ is a continuous map from a
	connected topological space $X$ such that
	$\psi(X)\cap\scrS^{k,\nu}_{\alpha,np}\neq
	\emptyset$, then $\psi(X)\subset \scrS^{k,\nu}_{\alpha,np}$.
	\end{theo}

The proof of Theorem~\ref{theo:connectedsol} is postponed at the end of
this section. We notice that the result implies one of our main theorems:
\begin{proof}[Proof of Theorem~\ref{theo:connected}]
The Whitney topology is stronger that the Hölder topology and the result is
	an immediate consequence of  Theorem \ref{theo:connectedsol}.
\end{proof}

\subsection{Linearized soliton equation}
By Definition~\ref{dfn:soliton}, a soliton  $F\in \scrF_\alpha$ is a solution of the equation $E(F) =0$,
where $E:\scrF_\alpha\to\scrF_\alpha$ is the functional
$$
E(F) = 4\phi(F)F - \|F\|_{L_2}^2\Pi_\alpha\nabla\phi(F) .
$$
Alternatively, we  may consider the dual functional 
$$e:\scrF_\alpha
\to L(\scrF_\alpha,\RR)$$
with values in the dual of $\scrF_\alpha$,
given by
$$
e(F)\cdot \dot F = \ipp{E(F),\dot F}
$$
for every $\dot F\in\scrF_\alpha$.
By definition
$$
e(F)\cdot\dot F = 4\phi(F)\ipp{F,\dot F} - \|F\|_{L^2}^2 D\phi|_F\cdot \dot F.
$$
In the particular case where $f:M\to M$ is a symplectomorphism of
$(M,\omega_M)$ with $[\scrD
f] =\alpha$, we have $\phi(\scrD f)=0$ and $F=\scrD f$ is a critical point of
$\phi:\scrF_\alpha\to \RR$. Since $\|f_*\|_{L^2}=\|\scrD f\|_{L^2}$,   the differential of $e$ is given by
$$
De|_{\scrD f}\cdot (\dot F_1, \dot F_2) = -\| f_*\|_{L^2}^2 D^2\phi|_{\scrD f}\cdot (\dot F_1,\dot
F_2) 
$$
and $De|_{\scrD f}$ agrees with  the Hessian of $\phi$ up to a constant
factor. 
\begin{lemma}
	Let $F\in\scrF_\alpha$ be a non proper soliton. Then, the Hessian of $\phi$ at $F$ is
	given by the formula~:
\begin{equation}\label{eq:hess}
	D^2\phi|_F\cdot (\dot F_1,\dot F_2)=\int_M \left (
	\sum_\bullet  g(\scrR_\bullet F,\dot F_1) g(\scrR_\bullet F,\dot
	F_2) 
	\right )\vol_M
\end{equation}
for every $\dot F_1, \dot F_2\in\scrF_\alpha$.
\end{lemma}
\begin{proof}
	By Lemma~\ref{lemma:formphi} and Equation~\eqref{eq:diffphi}
	$$
	D\phi|_F\cdot F_1 = \sum_\bullet \ipp{W_\bullet(F),\dot F_1}.
	$$
	where $W_\bullet(F)=-\mu_\bullet(F)\scrR_\bullet F$ by
	Equation~\eqref{eq:w}. Thus
	$$
	DW_\bullet|_F\cdot \dot F_2= -(D\mu_\bullet|_F\cdot \dot F_2) \scrR_\bullet F,
	$$
	since $\mu_\bullet(F)=0$.
	By Equation~\eqref{eq:diffmu}, we deduce that
	$$
	DW_\bullet|_F\cdot \dot F_2= g(\scrR_\bullet F,\dot F_2)\scrR_\bullet F,
	$$
	which implies
	\begin{align*}
		D^2\phi|F\cdot(\dot F_1,\dot F_2) &= \sum_\bullet
	\ipp{g(\scrR_\bullet F,\cdot F_2)\scrR_\bullet F, \dot F_1}\\
		&= \sum_\bullet
		\ipp{g(\scrR_\bullet F,\cdot F_2), g(\scrR_\bullet F, \dot
		F_1)}\\
	\end{align*}
	and the lemma follows.
\end{proof}
The following corollaries are immediate:
\begin{cor}
	Let $F$ be a non proper soliton in $\scrF_\alpha$.
The Hessian of $\phi:\scrF_\alpha\to\RR$ at  $F$ is
	a non negative symmetric bilinear form. Moreover its kernel
consists of vectors $\dot F\in\scrF_\alpha$ such that $D\mub|_F\cdot \dot F=0$. 
\end{cor}
	\begin{cor}
		Let $f$ be a symplectomorphism of $(M,\omega_M)$ and
		$F=\scrD f\in \scrF_\alpha$, where $\alpha=[F]$.
		Then, the differential of the soliton equation  $DE|_F$ is
		a  non positive selfadjoint operator, whose kernel
		(and cokernel) is  identified to $\ker D\mub|_F\subset
		\scrF_\alpha $.
\end{cor}

We continue our computation  at $F=\scrD f$, where
$f\in\Symp_\alpha(M,\omega_M)$: we
can write Formula~\eqref{eq:hess} as
\begin{equation}\label{eq:hess1}
	D^2\phi|_F(\dot F_1,\dot F_2)=\int_M 
	\sum_\bullet  (D\mu_\bullet|_F\cdot \dot F_1) (D\mu_\bullet|_F
	\cdot \dot F_2) 
	\vol_M
\end{equation}
and, using the natural inner product on $\RR^3$, the above formula can be
expressed as
\begin{equation}\label{eq:hess2}
	D^2\phi|_F(\dot F_1,\dot F_2)=\ipp{ 
 D\mub|_F\cdot \dot F_1, D\mub|_F
	\cdot \dot F_2 
	}.
\end{equation}

Let $v$ be a tangent vector field along $f:M\to M$. Then, there exists a
smooth family of maps
$f_t:M\to M$ defined for $t$ in a neighborhood of $0$, such that $f_0=f$
and $\frac{\del f_t}{\del t}|_{t=0}= v$. 
By definition $\scrD f_t = \rho \circ (f_t)_*$, hence
$$
\left .\frac \del{\del t}\scrD f_t\right |_{t=0}=\left .\frac{\del (\rho
\circ (f_t)_*)}{\del
t}\right |_{t=0}= d \rho \circ v_*=d(\rho\circ
v)=d\eta
$$
where $\eta=\rho\circ v$.

We will  use the notations summarized in the commutative
diagram below: if $f$ is a diffeomorphism,
we may write $v=f_*u$, for some vector field $u$ on $M$. We put $\eta=\rho\circ v$
and  
obtain the commutative diagram:
\begin{equation}
	\label{diag:conv}
	\begin{tikzcd}
		TM\ar[rr, bend left = 40,"\scrD f"]\ar[d, "\pi"]\ar[r,"f_*"] & TM\ar[d,"\pi"]\ar[r,"\rho"]& \vecV \\
		M \ar[r,"f"]\ar[u, bend left = 60,"u"] \ar[ru,"v"]\ar[rru,
		bend right = 60,"\eta", swap] & M
\end{tikzcd}
\end{equation}
Hence, if $F=\scrD f$  and $\dot F_i= d\eta_i$, where $\eta_i=\rho\circ v_i$
for some vector fields $v_i$ along $f$, then
\begin{equation}
\label{eq:hess3}
D^2\phi |_{\scrD f}\cdot(d\eta_1,d\eta_2)= \ipp{D\mub|_{\scrD f} \cdot
d\eta_1,
D\mub|_{\scrD f}
\cdot  d\eta_2   }.
\end{equation}
According to Remark~\ref{rmk:mutilde}, the hyperKähler moment map $\mub$
can be considered as a map $\tilde\mub$ with values in anti-selfdual forms.
This point of view makes  computations evident, as in the
next lemma~:
\begin{lemma}
	\label{lemma:dmutilde}
	Let $f$ be a diffeomorphism of $M$ and $v$ be a vector field along~$f$. 
	Then for every vector field $v$ along $f$, we have
$$
	D\tilde \mub|_{\scrD f}\cdot d\eta= \sqrt 2 d^-(\iota_uf^*\omega_M) 
$$
where $u$ is the vector field on $M$ defined by $v=f_* u$ and $\eta=\rho\circ
	v$.
\end{lemma}
\begin{proof}
	Let $f_t$ be a smooth family of maps defined for $t$ in some
	neighborhood of $0\in\RR$, such that $f_0=f$ and $\frac{\del
	f_t}{\del t}|_{t=0} =v$.  By definition
$$
	\tilde \mub(\scrD f_t) = \sqrt 2 (f_t^*\omega_M)^-
$$
and its variation at $t=0$ is given by
$$
	D\tilde \mub|_{\scrD f} \cdot d\eta =\sqrt 2 \left .
	\frac{\del}{\del
	t}\right |_{t=0}(f^*_t\omega_M)^-
$$
If we write $v=f_*u$, where $u$ is a vector field on $M$, then 
the $RHS$ can be expressed as
$$
	(\scrL_uf^*\omega_M)^- = d^-(\iota_u f^*\omega_M)+(\iota_ud
	f^*\omega_M)^- ,
$$
where $\scrL_u$ is the Lie derivative,
and we obtain the lemma since $f^*\omega_M$ is closed.
\end{proof}

\begin{lemma}
Let $v_1, v_2$ be two tangent vector fields
	along a symplectomorphism $f:M\to M$. Using the notation
	conventions of Diagram~\eqref{diag:conv}, we have
	$$
	D^2\phi |_{\scrD f} (d\eta_1,d\eta_2)=2 \ipp{d^-(u_1^\flat), d^-u_2^\flat}
$$
	where $\flat: TM\to T^*M$ is symplectic musical isomorphism, 
	defined by $u^\flat = \iota_u\omega_M$.
\end{lemma}
\begin{proof}
	We know that $\tilde \mub$ and $\mub$ have the same $L^2$-norms, by
	definition (cf. Remark~\ref{rmk:mutilde}), and the result follows from
	Lemma~\ref{lemma:dmutilde} and Equation~\eqref{eq:hess3}. 
\end{proof}
If $f$ is a diffeomorphism, there is a correspondence between vector fields
$u$ on $M$ and vector field $v$ along $f$ via the equation $f_*u =v$. We
denote the inverse of this map 
$$P_f = (f^{-1})_*=(f_*)^{-1},
$$
so that $u=P_fv$.

\begin{cor}
	\label{cor:diffe}
	Let $v_1$ and $v_2$ be two vector fields along a symplectomorphism
	$f\in\scrM_\alpha$. Then, we have
$$
  De|_{\scrD f} \cdot (d\eta_1,d\eta_2)= -2\| f_*\|_{L^2}^2 \ipp{d^-
	(P_fv_1)^\flat,d^-(P_fv_2)^\flat}
$$
where $\eta_i=\rho\circ v_i$.
\end{cor}
\begin{cor}
	\label{cor:diffpsi}
	Let $v$ be a vector field along a symplectomorphism
	$f\in\Symp_\alpha(M,\omega_M)$.
	Then, we have
$$
	d^\star DE|_{\scrD f} \cdot d\eta = -2\| f_*\|_{L^2}^2
	 \; \rho \cdot P_f^\star\cdot (d^\star d^-
	(P_fv)^\flat)^\sharp
$$
	where $P_f=(f_*)^{-1}$ 
and 
$\flat:TM\to T^*M $ and $\sharp:T^*M\to TM$ denote the symplectic duality
	musical operators.
\end{cor}
\begin{proof}
	By definition of $d^\star$, we have
	$$\ipp{d^\star DE|_{\scrD f}\cdot d\eta_1,\eta_2}= 
	\ipp{DE|_{\scrD f}\cdot d\eta_1,d\eta_2}= De|_{\scrD
	f}(d\eta_1,d\eta_2).$$
	By Corollary~\ref{cor:diffe}, we deduce 
	$$
	\ipp{d^\star DE|_{\scrD f}\cdot d\eta_1,\eta_2}= -2\|
	f_*\|^2_{L^2}\ipp{d^\star d^- (P_f v_1)^\flat, (P_f v_2)^\flat}
	$$
	Now $\ip{\beta, u^\flat}= \ip{\beta^\sharp,u}$ and
	$\ip{v_1,v_2}=\ip{\rho v_1,\rho v_2}$
therefore, 
	\begin{align*}
		\ipp{d^*DE|_{\scrD f}\cdot d\eta_1,\eta_2}&= 
		-2\|
		f_*\|^2_{L^2}\ipp{ (d^\star d^- (P_fv_1^\flat) )^\sharp
		, P_f v_2}\\
		&=  -2\|f_*\|^2_{L^2}\ipp{ P_f^\star\cdot (d^\star d^-( P_f
		v_1)^\flat)^\sharp , v_2}\\
		&=  -2\|f_*\|^2_{L^2}\ipp{\rho\cdot P_f^\star\cdot (d^\star d^-( P_f
		v_1)^\flat)^\sharp , \eta_2}.
	\end{align*}	
and the corollary follows.
\end{proof}

\subsection{Rigidity of non proper solitons in the case of maps}
\label{sec:rigid}
A symplectomorphisms $f\in\Symp_\alpha(M,\omega_M)$ provides a non proper
soliton
$F=\scrD f\in\scrF_\alpha\setminus 0$.
By definition, this is a solution of the
equation $E\circ\scrD (f)=0$. The symplectic deformations of the map $f$ provide an
infinite dimensional family of solutions of the equation $E\circ \scrD
f=0$. We are going to prove that  they are locally the only ones:
\begin{theo}
	\label{theo:ift}
	Given  some Hölder
	space parameters $(k,\nu)$ with $k\geq 4$ and a symplectomorphism
	$f\in\scrM^{k,\nu}_\alpha$, there exists
	$\epsilon>0$ with the following property~: for every $h\in
	\scrM_\alpha^{k,\nu}$ such that $\scrD h$ is a soliton and
	 $\|f-h\|_{k,\nu}\leq
	\epsilon $,  then  $h$ is a symplectomorphism.
\end{theo}
\begin{proof}
	The result is an immediate consequence of Lemma~\ref{lemma:weaker}
	and Theorem~\ref{theo:ift2}.
\end{proof}

For technical reasons, it is convenient to replace the soliton equation for
maps $E\circ\scrD (f)=0$ with the weaker equation 
$$\Psi(f)=0,
$$
where $\Psi$ is the functional
\begin{align*}
	\Psi:\scrM_\alpha  &\to \Omega^0(M,\vecV)\\
	f&\mapsto d^\star E(\scrD f).
\end{align*}
Obviously, $E(\scrD f)=0$ implies $\Psi(f)=0$. Conversely, $\Psi(f)=0$
implies that $E(\scrD f)$ is orthogonal to closed  forms. Since $E(\scrD f)$ is
 closed, by definition of $E$, it follows that $E(\scrD f)$ is
harmonic and we have the following result:
\begin{lemma}
	\label{lemma:weaker}
	The solutions $f\in\scrM_\alpha$ of the equation $\Psi(f)=0$ are
	the maps  such that $E(\scrD f)$ is a harmonic form. 
\end{lemma}

In the rest of this section, we prove a rigidity result for $\Psi$.
Loosely stated, near a symplectic map $f$, every zero of $\Psi$ is also a
symplectic map:
\begin{theo}
	\label{theo:ift2}
	Given  some Hölder
	space parameters $(k,\nu)$ with $k\geq 4$ and a symplectomorphism
	$f\in\scrM^{k,\nu}_\alpha$, there exists
	$\epsilon>0$ with the following property~: for every $h\in
	\scrM_\alpha^{k,\nu}$ such that $\Psi(h)=0$  and
	 $\|f-h\|_{k,\nu}\leq
	\epsilon $,  then  $h$ is a symplectomorphism.
\end{theo}
\begin{proof}
	The theorem is a restatement of Corollary~\ref{cor:theoift}.
\end{proof}

\begin{lemma}
The map $\Psi$ is given by
$$
	\Psi(h)=4\|(h^*\omega_M)^-\|^2_{L^2}  
	d^\star \scrD h -\|h_*\|^2_{L^2}
	d^\star\nabla\phi(\scrD h)
$$
\end{lemma}
\begin{proof}
	The formula is obtained from the definition of $E(F)$, replacing $F$ with
	$\scrD h$ and using
	the fact that $\|\scrD h\|_{L^2}=\|h_*\|_{L^2}$. We also notice
	that $d^\star\Pi_\alpha F= d^\star F$ for every $F\in\scrF$. Indeed,
	the orthogonal projection $\Pi_\alpha$ operates by adding some
	harmonic and coclosed terms, which belong to the kernel of
	$d^\star$. Finally, we have $\phi(\scrD h) 
	=\frac 12\|\tilde\mub(\scrD h)\|_{L^2}^2=
	\|(h^*\omega_M)^-\|_{L^2}^2$ which proves the formula.
\end{proof}
\begin{lemma}
	\label{lemma:diffpsi}
The variation of $\Psi$ at a symplectomorphism $f$ is given by 
$$
	D\Psi|_f\cdot v = -2\|\scrD
	f\|^2_{L^2}\rho\cdot P_f^\star\cdot  (d^\star d^-(P_f\cdot v)^\flat)^\sharp
$$
	where $P^\star_f$ is the formal ajoint of $P_f=(f_*)^{-1}$ and  $v$
	is a vector field along $f$. The differential is formally
	selfadjoint in the sence that 
	$$
	\ipp{D\Psi|_f\cdot v_1,\rho\cdot v_2} =
	\ipp{\rho \cdot v_1, D\Psi|_f\cdot v_2} .	$$
	Furthermore, $D\Psi|_f$ is non positive, in the sense that 
	$$
	\ipp{D\Psi|_f\cdot v,\rho\cdot v} = -2 \|\scrD f\|_{L^2}^2 \|d^-(P_f
	v)^\flat\|_{L^2}^2,
	$$
	and the kernel of $D\Psi|_f$ is identified to symplectic vector
	fields along $f$.
\end{lemma}
\begin{proof}
	By definition $D\Psi|_f\cdot v$ is the variation of $d^\star E|_{\scrD
	f}\cdot d\eta$, where $\eta=\rho\circ v$ and the formula follows from
	Corollary~\ref{cor:diffpsi}.
	Now,
	\begin{align*}
\ipp{\rho\cdot P_f^\star \cdot (d^\star d^- (P_f
		v_1)^\flat)^\sharp,\rho\cdot v_2}
		&=\ipp{P_f^\star\cdot (d^\star d^- (P_f v_1)^\flat)^\sharp , v_2} \\
		&=\ipp{(d^\star d^- (P_f v_1)^\flat)^\sharp , P_fv_2} \\
		&= \ipp{d^\star
	d^- (P_f v_1)^\flat , (P_fv_2)^\flat}\\
			&= \ipp{
	d^- (P_f v_1)^\flat , d(P_fv_2)^\flat}\\
		&= \ipp{
	d^- (P_f v_1)^\flat , d^-(P_fv_2)^\flat}
	\end{align*}
so that 
$$
	\ipp{D\Psi|_f\cdot v_1,\rho\cdot v_2}= \ipp{\rho\cdot v_1, D\Psi|_f\cdot v_2},
$$
and we conclude that $D\Psi|_f$ is formally selfadjoint. In the case where
$v=v_1=v_2$, the above computation proves the last identity of the
	lemma.
\end{proof}
\subsection{Weinstein chart}
The  \emph{Weinstein
charts} provide nice local coordinates on the moduli
space $\scrM$. They are particularly well suited  for applying the implicit
function theorem in the proof of
Theorem~\ref{theo:ift2}.
We recall the basic steps
for the construction of a Weinstein chart  (for more details see~\cite{Ban,MDS}):
\begin{enumerate}
	\item
The product $N=M\times M$ is
endowed with the product symplectic form 
$$\omega_N= p_1^*\omega_M
-p_2^*\omega_M,$$
where $p_i:N\to M$ are the canonical projections on the
$i$-th factor of $N=M\times M$.
It is well known that $f$ is a symplectomorphism of $(M,\omega_M)$ if, and only if,
its graph 
$$G_f=\{(x,f(x)), x\in M\} \subset N
$$
is a Lagrangian submanifold of $(N,\omega_N)$.

\item
The total space  $T^*M$ of the contangent bundle is endowed with the canonical symplectic form
$$\omega_{T^*} = d\Lambda,$$
where $\Lambda$ is the Liouville form on $T^*M$, defined by
$$
\Lambda(w)= \tilde\pi(w) (\pi_*(w))
$$
where $\pi:T^*M\to M$ is the canonical projection,  $\pi_*:TT^*M\to TM$
is its tangent map and $\tilde \pi : TT^*M\to T^*M$ is the canonical projection
of the tangent bundle of $T^*M$.

\item
If $f$ is a symplectomorphism, the Lagrangian tubular
neighborhood theorem shows that there exists an open neighborhood
$\cU\subset N$ of
$G_f$  and an open neighborhood $\cV\subset T^*M$ of the $0$-section
$\beta_0:M\to T^*M$, together with a diffeomorphism
$$\theta:\cV\to \cU,$$
such that 
$$\theta^* \omega_N = \omega_{T^*} \quad \mbox{ and }\quad \theta \circ \beta_0(x) =
(x,f(x)) \forall x\in M.$$

\item 
In particular, any differential $1$-form $\beta$ on $M$, such that
$\beta(M)\subset \cV$,  defines
an embedding $\theta\circ\beta:M\to N$, whose image, denoted $G_{f,\beta}$, is a deformation of $G_f$.
Another classical result is that $G_{f,\beta}$  is Lagrangian if, and only if, $\beta$ is a
closed differential $1$-form. 

\item
By
construction $p_1\circ \theta\circ\beta_0 = \id_M$. It follows that
$\varphi_\beta = p_1\circ \theta\circ\beta$ is  a diffeomorphism of
$M$ as well, provided the differential $1$-form $\beta$ has sufficiently small  $C^1$-norm.
Then we can define
$$
f_\beta = p_2\circ\theta\circ\beta\circ\varphi_\beta^{-1},
$$
so that $G_{f,\beta}$ is the graph of $f_\beta:M\to M$. Notice that, by
definition $f_{\beta_0}=f$.
		
		\item The Weinstein chart at $f$ is given by
$$
W(\beta) = f_\beta,
$$
and is defined on a sufficiently small $C^1$-neighborhood of the $0$-section
		$\beta_0$. By construction, $f_\beta$ is homotopic to $f$,
		hence $W(\beta)\in\scrM_\alpha$, where $\alpha$ is the
		symplectic class $[\scrD f]$.
In the sequel is more convenient to use the $(k,\nu)$-Hölder norms, for some
$k\geq 2$ and we can  summarize the  construction in
Proposition~\ref{prop:weinstein}.
\end{enumerate}
\begin{prop}
	\label{prop:weinstein}
	Let $\alpha$ be a symplectic class and $(k,\nu)$  some Hölder parameters
	 with $k\geq 2$.
	For every symplectic map
	$f\in\scrM_\alpha^{k,\nu}$, there exists
	an open neighborhood $\scrU$ of $\beta_0=0$ in the space
	$\Omega^1(M)_{k,\nu}$ of differential $1$-forms
	with Hölder regularity  $C^{k,\nu}$ and an open neighborhood $\scrV$ of $f$ in the
	space of Hölder maps~$\scrM_\alpha^{k,\nu}$,
	such that the Weinstein map at $f$
	$$
	W:\scrU\to\scrV
	$$
	is a local
	diffeomorphims. Furthermore, the Weinstein chart restricts as a diffeomorphism
	$$W:\scrU_c\to\scrV_\omega$$
	where $\scrU_c\subset \scrU$ is  the subspace of closed forms and
	$\scrV_\omega\subset\scrV$ is the subspace of symplectic maps with
	respect to $\omega_M$.
\end{prop}
We can summarize Proposition~\ref{prop:weinstein} by a diagram,
\begin{equation}
	\begin{tikzcd}
		\scrU \ar[r,"W"] & \scrV  \\
		\scrU_c\ar[u,hook] \ar[r,"W"] &
		\ar[u,hook]  \scrV_\omega
\end{tikzcd}
\end{equation}
where vertical arrows are the canonical
inclusion maps and horizontal maps are diffeomorphisms.

\subsection{Differential of the Weinstein map}
In this section we consider the Weinstein chart at a symplectic map $f$ 
and we compute its differential  at
$\beta_0=0$.
Let $\dot \beta$ be a smooth differential $1$-form on $M$ and $\beta_t=t\dot\beta$, with
$\varphi_t=\varphi_{\beta_t}= p_1\circ\theta\circ\beta_t$. Then $\varphi_t$ is
a $1$-parameter family of diffeomorphisms, for $t$ sufficiently small,
such that $\varphi_0=\id_M$. We put
$f_t=f_{\beta_t}=W(\beta_t)$ and we compute
$$
\left . \frac{\del\varphi_{t}}{\del t}\right |_{t=0}= (p_1)_* \cdot
\theta_*\cdot\dot \beta 
$$
where $\dot \beta$ is understood as a vertical tangent vector field along
the zero section of $T^*M\to
M$.
Hence
$$
\left . \frac{\del\varphi^{-1}_{t}}{\del t}\right |_{t=0}= - (p_1)_* \cdot
\theta_*\cdot\dot \beta 
$$

The vector field $\theta_*\cdot \dot\beta$ along $G_f$ is neither vertical nor
horizontal in $N=M\times M$. However 
$$
\left .
\frac{\del f_t}{\del t}\right |_{t=0}=
\left .\frac{\del}{\del t} \theta\circ \beta_t\circ\varphi_t^{-1}\right |_{t=0} =
\theta_*\cdot\dot\beta - \theta_*\cdot (\beta_0)_*\cdot  (p_1)_*\cdot
\theta_*\cdot\dot\beta 
$$
is parallel to the fibers of $p_2:N\to M$. Indeed
$(p_1)_*\theta_*(\beta_0)_*=\id$, so
that the first projection of the above vector vanishes.
In conclusion
$$
\left . \frac{\del W(\beta_t)}{\del t}\right |_{t=0} = (p_2)_* \theta_*\cdot\dot \beta
$$
which is to say 
\begin{equation}
	\label{eq:dw}
DW|_{\beta_0}  = (p_2)_*\circ\theta_*.
\end{equation}
We deduce the following lemma
\begin{lemma}
	\label{lemma:diffw}
	The $1$-form $\dot \beta$ is closed if, and only if,
	$v=(p_2)_*\theta_*\cdot \dot\beta$ is a
	symplectic vector field along $f$, in other words
	$$
	d\dot\beta =0\Leftrightarrow d\iota _{u}\omega_M=0,$$
	 where
	$u$ is the vector field defined by $f_*u=v$.
\end{lemma}
\begin{proof}
The Weinstein chart is a local diffeomorphism that maps closed differential
	$1$-forms to symplectic maps and vice-versa, according to
	Proposition~\ref{prop:weinstein}. The lemma
	follows from the computation of $DW|_{\beta_0}$ given by Formula~\eqref{eq:dw}.
\end{proof}
\subsection{Implicit function theorem}
\label{sec:ift}
The map $\Psi$, defined in the smooth setting at~\S\ref{sec:rigid},
admits a canonical  extension to Hölder spaces
$$
\Psi:\scrV\to \Omega^0(M,\vecV)_{k-2,\nu},
$$
where $\scrV$ is an open neighborhood of $f$ in $\scrM_\alpha^{k,\nu}$,
chosen according to Proposition~\ref{prop:weinstein}, and 
$\Omega^0(M,\vecV)_{k-2,\nu}$ is the space of $\vecV$-valued functions on $M$ with Hölder regularity
$(k-2,\nu)$, for some $k\geq 2$.

 By definition, $\Psi$ vanishes along the space of
symplectomorphism $\scrV_\omega$. Our goal is to prove that there are no
other solutions of the equation $\Psi=0$ in $\scrV$,
provided it is a sufficiently small neighbordhood of $f$ in
$C^{k,\nu}$-norm.

It is convenient to work on an open set $\scrU\subset\Omega^1(M)_{k,\nu}$,
using the Weinstein chart
$W:\scrU\to\scrV$, where $\scrU$ is a neighbordhood of $\beta_0=0$, as in
Proposition~\ref{prop:weinstein}. The we  introduce the reparametrized map
\begin{equation}
	\label{eq:phihold}
	\Phi = \Psi\circ W:\scrU \to \Omega^0(M,\vecV)_{k-2,\nu}
\end{equation}
and our problem is equivalent to ask whether the points of $\scrU_c$ are the only
zeroes of $\Phi$ in $\scrU$.

By Lemma~\ref{lemma:diffpsi} and Lemma~\ref{lemma:diffw}, the 
kernel of $D\Phi|_{\beta_0}$ is the space of closed differential $1$-forms
$\scrW^{k,\nu}_c$ with
regularity $C^{k,\nu}$. In particular, Hodge theory provides an orthogonal splitting
$$
\Omega^1(M)_{k,\nu}=\scrW^{k,\nu}_c\oplus \scrW^{k,\nu}_{d^\star}
$$
where $\scrW_{d ^\star}^{k,\nu}$ is the space of coexact forms with regularity
$C^{k,\nu}$.

We have a similar splitting
$$
\Omega^1(M)_{k-2,\nu}=\scrW^{k-2,\nu}_c\oplus \scrW^{k-2,\nu}_{d^\star}
$$
for Hölder regularity $C^{k-2,\nu}$, under the assumption that  $k\geq 4$.
The symplectic duality $\sharp :\Omega^1(M)\to \Gamma(TM)$ composed with
$\scrD f:\Gamma(TM)\to \Omega^0(M,\vecV)$ provides an isomorphism
$$
\scrD f\circ \sharp:\Omega^1(M)_{k-2,\nu}\to \Omega^0(M,\vecV)_{k-2,\nu}.
$$
By Lemma~\ref{lemma:diffpsi},  the image of the factor
$$
\hat \scrW_{c}^{k-2,\nu}=\scrD f\circ\sharp(\scrW_c^{k-2,\nu})
$$
is the cokernel of $D\Phi|_{\beta_0}$.
and we have a splitting
$$
\Omega^0(M,\vecV)_{k-2,\nu}= \hat\scrW^{k-2,\nu}_c\oplus
\hat\scrW^{k-2,\nu}_{d^\star}
$$
where
$$
\hat\scrW^{k-2,\nu}_{d^\star}= \scrD
f\circ\sharp(\scrW^{k-2,\nu}_{d^\star})
$$
Finally Hodge orthogonal projection 
$$
\Pi_{d^\star}:\Omega^1(M)_{k-2,\nu}\to \scrW^{k-2,\nu}_{d^\star}
$$
induces a projector 
$$
\Pi_{d^\star}:\Omega^0(M,\vecV)_{k-2,\nu}\to \hat\scrW^{k-2,\nu}_{d^\star}
$$
via the isomorphism $\scrD f\circ\sharp$.
and we can consider the projected map
$$
\hat \Phi:\scrU\to\hat\scrW_{d^\star}^{k-2,\nu}
$$
defined by
$$
\hat\Phi = \Pi_{d^\star}\circ\Phi.
$$
By Remark~\ref{rmk:hodge}, the projector $\Pi_{d^\star}$ is continuous so
that $\hat\Phi$ is a differentiable map.
By construction $D\hat\Phi|_{\beta_0}$ is surjective and its kernel is
$\scrW_c^{k,\nu}$. Therefore the operator
$$
D\hat\Phi|_{\beta_0}:\scrW^{k,\nu}_{d^\star}\to\hat
\scrW^{k-2,\nu}_{d^\star}
$$
is an isomorphism.

By the implicit function theorem, we deduce the following results
\begin{theo}
	\label{theo:ifttech}
	Let $(k,\nu)$ be some Hölder parameters with $k\geq 4$.
If $\scrU$ is a sufficiently small neighbohood of $\beta_0=0$ in
	$\Omega^1(M)_{k,\nu}$, then the restriction 
	 of $\hat\Phi$ to every affine subspace parallel to
	$\scrW_{d^\star}^{k,\nu}$ is an embedding. More precisely
	for every 
	closed form $\beta\in\scrU_c$, 
$$
	\hat\Phi:\scrU\cap
	(\beta+\scrW^{k,\nu}_{d^\star})\to\hat \scrW^{k-2,\nu}_{d^\star}
	$$
	is an embedding.
	In particular $\beta$ is the only zero of $\hat \Phi$ in
	this affine subspace.
\end{theo}
\begin{cor}
	With the assumptions of Theorem~\ref{theo:ifttech}, 
	the map $\Phi$ defined at~\eqref{eq:phihold} satisfies
	$$\Phi^{-1}(0) = \scrU_c.
	$$
\end{cor}
\begin{proof}
If $\beta$ is a zero of $\Phi$, it is also a zero of the projected map
	$\hat\Phi$ and the result follows from Theorem~\ref{theo:ifttech}.
\end{proof}
\begin{cor}
	\label{cor:theoift}
	With the assumption of  Theorem~\ref{theo:ifttech},
	the map $\Psi$ defined on $\scrV$ satisties
	$$
	\Psi^{-1}(0)=\scrV_\omega.
	$$
	Furthemore, if $G\in \scrF^{k-1,\nu}_\alpha$ is a soliton of the form $G=\scrD h$,
	with $h\in\scrV$, then $h$ is a symplectomorphism.
\end{cor}
\begin{proof}[Proof of Theorem~\ref{theo:connectedsol}]
	The space $\scrS^{k,\nu}_{\alpha,np}$ is  the vanishing locus of
	the restricted energy functional $\phi:\scrS^{k,\nu}_\alpha\to\RR$. This
	functional is continuous, hence
	$\scrS^{k,\nu}_{\alpha,np}$ is a closed subspace of
	$\scrS^{k,\nu}_\alpha$.

	We now show that the space $\scrS^{k,\nu}_{np}$ is open in the
	space of non zero solitons.
	Let $F$ be a non proper soliton of $\scrF^{k,\nu}_\alpha\setminus 0$. 
 According to
	Lemma~\ref{lemma:nonvanish}, the cohomology class $[F]$ is non
	vanishing  and
	we may assume that
	$[F]=\alpha$, up to rescaling by a constant factor. Thus,
	there exists a symplectic map $f\in\scrM^{k+1,\nu}_\alpha$ such that
	$F=\scrD f$. 
	More precisely, we can define $f=\chi(F)$,  where $\chi$ is
	the integral defined by~\eqref{eq:primit}.
	  The map 
	$\chi$ was defined in the smooth case, however the construction extends in an obvious way as a
	continuous map:
$$
	\chi: \{G\in\scrF^{k,\nu}_\alpha, [G]=\alpha\} \to
	\scrM_\alpha^{k+1,\nu}.
$$
 For $k\geq 3$, we choose an open
	neighborhood $\scrV$ of $f$ in $\scrM_\alpha^{k+1,\nu}$, according to
	Corollary~\ref{cor:theoift}. 
	By definition, if $h\in\scrV$ is  such that $G=\scrD h$ is a
	soliton, then $h$ is symplectic.
	By continuity, $\chi^{-1}(\scrV)$ is an open neigborhood of $F$ in the
	source space of $\chi$ and the rescaling map
	\begin{align*}
		q:\scrF^{k,\nu}_\alpha &\to \{G\in \scrF_\alpha^{k,\nu},
		[G]=\alpha\}\\
		G&\mapsto \frac \alpha{[G]}G
	\end{align*}
is well defined and continuous in a neighborhood of $F$.
	We consider the neighborhood of $F$ given by 
	$\scrV' = (\chi\circ
	q)^{-1}(\scrV)\subset \scrF_\alpha^{k,\nu}$. If $F'\in \scrV'$ is
	a soliton, 
	then $G= q(F')$ is also a soliton, and  $h=\chi(G)\in \scrV$
	is such that $G=\scrD h$. By Corollary~\ref{cor:theoift}, this
	implies that $h$ is a symplectic map, which shows  that $G$ is not
	a proper soliton. It follows that $F'$ is not proper as well. 
In conclusion, every soliton of $\scrV'$ is not proper. This proves that
	$\scrS^{k,\nu}_{\alpha,np}$ is open in $\scrS_\alpha^{k,\nu}$.

	Finally, $\scrS^{k,\nu}_{\alpha,p}$ is the complement of
	$\scrS^{k,\nu}_{\alpha,np}$ in the space $\scrS^{k,\nu}_\alpha$. It follows that
	$\scrS^{k,\nu}_{\alpha,p}$ is open and closed in $\scrS_\alpha^{k,\nu}$.
	The  last statement of the theorem is
	classical, by the connexity of $X$.
\end{proof}

\section{Polyhedral  symplectic geometry}
\label{sec:polmap}
In this section, we introduce polyhedral analogues of all the
geometrical objects introduced  in the smooth setting, in the first part of
this work. 
The constructions stem from a quotient torus $M=V/\Gamma$, as in
\S\ref{sec:hk}, where  the underlying vector space $\vecV$ 
is equiped with a linear isomorphism with the space of quaternions $\HH$.
In \S\ref{sec:hk}, we saw that the quotient construction various structures
on $M$, including a flat Riemannian metric $g_M$ and several 
integrable compatible almost complex structures, $i,I,J$ and $K$ on $M$, with Kähler
forms $\omega_M, \hat\omega_I, \hat\omega_J$ and $\hat\omega_K$.
We will continue to use the same background geometry on the quotient torus
$M$, as in the smooth setting.

\subsection{Triangulations and polyhedral  maps}
In the next paragraphs, we recall some standard definitions and introduce
some basic concepts 
related to polyhedral geometry.

A \emph{$k$-simplex}, contained in some ambient affine space,
is  the \emph{convex hull} of $k+1$
\emph{affinely independent} points of the affine space. 
A  \emph{simplicial complex} $\scrK$ is a collection of simplices contained in some fixed
 affine ambient space. The collection must be stable by intersection and by passing to
subfaces.

Given a simplex $\sigma$, we say that $f:\sigma\to M$ is an \emph{affine
map} if it
admits an affine lift $\tilde f$ to the universal cover
$$
\begin{tikzcd}
	& V\ar[d,"p"]\\
	\sigma \ar[r,"f"]\ar[ur,"\tilde f"]& M
\end{tikzcd}
$$
Similarly a map $v:\sigma\to TM$ is affine if it admits an affine lift
to the tangent space of the universal cover $v:\sigma\to TV$.

A \emph{triangulation} $\scrT=(M,\scrK,\ell)$ of $M$ is given by a triple
where:
\begin{itemize}
	\item $\scrK$ is a geometric simplicial
complex contained in some ambient affine space,
		\item  $\ell:|\scrK|\to M$ is a homeomorphism, where $|\scrK|$ is
the topological space associated to $\scrK$.
\end{itemize}
	 A \emph{Euclidean
triangulation} is a
triangulation such that the restriction $\ell:\sigma\to M$ to every simplex
$\sigma\in\scrK$ is an affine map.
We say that a triangulation is \emph{finite}, if the underlying simplicial complex
is finite.

In this paper, we only consider 
	  \emph{finite Euclidean triangulations}  of the quotient torus
	  $M$. For simplicity,  a
	  \emph{triangulation} of $M$ always refers to finite Euclidean
	  triangulation, unless stated otherwise.
	 
	We have an alternate definition by passing to the universal cover:
 a finite Euclidean triangulation of $M$ is given equivalently by 
 a locally finite $\Gamma$-invariant  simplicial
	decomposition of~$V$.

	\begin{dfn}A map
$f:M\to \RR$  is called polyhedral with respect to a Euclidean triangulation
	$\scrT$ of $M$ if the following conditions are satisfied: 
\begin{enumerate}
	\item The map $f$ is continuous.
	\item The restriction of the map $f\circ \ell:\sigma\to \RR$  is affine for
		every $\sigma\in\scrK$.
\end{enumerate}
We define similarly polyhedral maps on $M$ with values in a vector space, an
affine space, in $M$ or in $TM$.
	\end{dfn}
The \emph{moduli space of polyhedral maps} $f:M\to M$, with respect
to some fixed Euclidean triangulation $\scrT$ of $M$, 
is denoted 
$$
\scrM(\scrT) = \{f:M\to M, \mbox{ $f$ is polyhedral with respect to
$\scrT$}\}.
$$
The moduli space is the obvious analogue of the space of smooth maps
$\scrM$. One of the main features of $\scrM(\scrT)$ is that it is a finite
dimensional manifold. 

Once a triangulation $\scrT$ is prescribed, we will often identify a
simplex $\sigma\in\scrK$ with its image $\ell(\sigma)\subset M$. Bearing
this in mind, it makes sense to talk about the restriction of a map defined
on $M$ to any simplex $\sigma$ of the triangulation understood as a domain
in $M$. Furthermore, a vertex $\sigma_0\in\scrK_0$ is assimilated to a point
$\sigma_0\in M$ via $\ell$.

A \emph{vector field} $v$, along  a polyhedral map $f\in\scrM(\scrT)$, is a map defined on the set of
$0$-simplexes of the triangulation,
such that the following diagram is commutative 
$$
\begin{tikzcd}
	&TM\ar[d,"\pi"]\\
	\scrK_0\ar[r,"f"]\ar[ur,"v"]	&M
\end{tikzcd}
$$
where $\pi:TM\to M$ is the tangent bundle and the set of vertices $\scrK_0$
is understood as a finite subset of $M$.
With such a definition, a vector field can be extended uniquely as a map $v:M\to
TM$, polyhedral with
respect to $\scrT$, such that the following diagram commutes
$$
\begin{tikzcd}
	&TM\ar[d,"\pi"]\\
	M\ar[r,"f"]\ar[ur,"v"]	&M
\end{tikzcd}
$$

The tangent space of $\scrM(\scrT)$ at $f$, denoted $T_f\scrM(\scrT)$, is
the vector
space of polyhedral vector fields along $f$. The tangent space provides an
exponential local chart at $f$ for the moduli space.
The exponential map if defined by
\begin{align*}
	T_f\scrM(\scrT)&\to\scrM(\scrT)\\
	v&\mapsto f+v
\end{align*}
where $f+v$ is given by
$$
(f+v)(x)=f(x)+v(x),
$$
for every $x\in M$ and $v(x)\in T_{f(x)}M$ acts by translation on the torus $M$.
By definition $\frac \del{\del t}(f+tv)|_{t=0}=v$, which shows that the space of
polyhedral vector fields along $f$ is indeed the tangent space.

\subsection{Differential for polyhedral maps}
\label{sec:diffpol}
A polyhedral map $f\in\scrM$ is generally not differentiable. However its
restriction $f:\sigma\to M$ is affine, hence differentiable, for every
$\sigma\in\scrK$. Thus, we have a collection of differentials
$$ \scrD f|_\sigma = \rho\circ (f|_\sigma)_*
$$
for every $\sigma\in\scrK$.
Since $f|_\sigma$ is affine, it follows that $\scrD f_\sigma$ is constant on
$\sigma$ and in particular $d\scrD f|_\sigma=0$. If $\sigma_2$ is a face of $\sigma_1\in\scrK$, then
$j^*\scrD f_{\sigma_1}=\scrD f_{\sigma_2}$, where $j:\sigma_2\to\sigma_1$ is the
canonical inclusion. In particular, all the informations of the
family $(\scrD
f_\sigma)_{\sigma\in\scrK}$ can be deduced only from the facets $(\scrD
f_\sigma)_{\sigma\in\scrK_4}$. 
Depending on the context, $(\scrD f_\sigma)$ can be understood as a family
where $\sigma\in\scrK$ or $\sigma\in\scrK_4$ without any loss of generality.

We define 
the space of families of  constant differential forms on facets
\begin{align*}
	\scrF(\scrT)=\{(F_\sigma)_{\sigma\in\scrK_4},\quad
	&	F_\sigma \in\Omega^1(\sigma,\vecV)\\
	& \mbox { and } F_\sigma \mbox{ is constant on $\sigma$}\}
\end{align*}
By definition, the differential $\scrD$ induces a linear map
$$
\scrD:\scrM(\scrT)\to\scrF(\scrT).
$$

We denote by $\Omega^k(M,\scrT)$ the space of families of differential forms
$\beta=(\beta_\sigma)_{\sigma\in\scrK}$ where $\sigma\in\scrK$, where $\beta_\sigma$ is a smooth
differential $k$-form on $\sigma$, understood as a domain in $M$. An
element $\beta=(\beta_\sigma)\in\Omega^k(M,\scrT)$
satisfies the \emph{Whitney condition}, if for every $\sigma_1,\sigma_2\in\scrK$,
with $\sigma_2$ a face of $\sigma_1$, we have
$j^*\gamma_{\sigma_1}=\gamma_{\sigma_2}$. 
The subspace of elements in $\Omega^k(M,\scrT)$ that satisfy the Whitney
condition is called the space of Whitney forms and  is denoted $\Omega^k_w(M,\scrT)$.
We define a differential operator $d$ on the space of Whitney forms: for
$\beta=(\beta_\sigma)$, we put $d\beta = (d\beta_\sigma)$. It is easy to
check that the operator $d$ preserves the Whitney condition 
and $d^2=0$. Thus we have recalled the definition of the Whitney complex
$$
d_k:\Omega^k_w(M,\scrT) \to \Omega^{k+1}_w(M,\scrT).
$$
We can extends  the construction to the case of $\vecV$-valued differential forms
and we obtain the Whitney cohomology spaces
$$
H^k_w(M,\vecV,\scrT).
$$
In particular, we have the following lemma:
\begin{lemma}
	The family of $\vecV$-valued  differential $1$-forms,  $ F=(\scrD
	f|_\sigma)$, where $\sigma\in\scrK$, is a closed Whitney form.
\end{lemma}

Smooth forms always satisfy  the Whitney condition. Hence we have a canonical map
$H^k(M,\vecV)\to H^k_w(M,\vecV,\scrT)$ from the DeRham cohomology. 
It turns out that this map is an isomorphism by the Whitney theorem. 
In particular, for every polyhedral map $f\in\scrM(\scrT)$, the family 
$F=(\scrD f|_\sigma)$ defines a cohomology class in $H^1_w(M,\vecV,\scrT)$,
hence in $H^1(M,\vecV)$.
For simplicity of notations, the family of differential forms $(\scrD f|_\sigma)$ is
often referred to as $\scrD f$.
Eventually, we have defined a  map
$$
\alpha:\scrM(\scrT)\to  H^1(M,\vecV)
$$
given by $\alpha(f)=[\scrD f]$

For an integral class $\alpha\in H^1(M,\vecV)\setminus 0$, we define the subspace   $\scrM_\alpha(\scrT)$ of polyhedral maps
$f$ such that $\alpha(f)=\alpha$. 
We also introduce the subspace of Whitney forms of $\scrF(\scrT)$ with
cohomology class in $\RR\alpha$, denoted $\scrF_\alpha(\scrT)$.
By construction, $\scrD$ defines a linear map
$$
\scrD:\scrM_\alpha(\scrT)\to\scrF_\alpha(\scrT),
$$
and its image is identified to $\scrM_\alpha(\scrT)/\vecV$.

Given $f\in\scrM(\scrT)$, the map tangent map $f_*$ is only defined along
simplexes of $\scrK$. The pullback of a smooth 
form $\beta$ by $f$,
defined by $f^*\beta=(f^*\beta|_\sigma\sigma)$,  satisfies the Whitney condition
as well. If $\beta$ is closed, so is $f^*\beta$, in the sense of Whitney.
As an application, 
  $f^*\omega_M=(f^*\omega_M|_\sigma)$ is well a defined closed
Whitney $2$-form. 
Furthermore, the fact that $f$ is affine along every simplex implies that
$f_*$ is constant along every simplex and that the pullback
$f^*\omega_M$ is constant along every simplex of the triangulation.
\begin{dfn}
	A polyhedral map $f\in\scrM(\scrT)$ such that the pullback
	$f^*\omega_M$ vanishes along every simplex of the triangulation is called
	a polyhedral symplectic map. 
The space of polyhedral symplectic maps with respect to $\scrT$ is denoted
$\Symp(M,\omega_M,\scrT)$. 
\end{dfn}
\begin{rmk}The space $\Symp(M,\omega_M,\scrT)$ does not have a group
structure, since polyhedral maps cannot be composed without passing to some
direct limit over all possible triangulations of $M$.
Proposition~\ref{prop:classical} does no apply out of the box, and it is not even clear
whether a symplectic polyhedral map is a homeomorphism of $M$.
\end{rmk}
Similarly to the smooth setting, we introduce the notation 
$$\Symp_\alpha(M,\omega_M,\scrT)$$
for
the space of polyhedral symplectic maps such that $f\in\scrM_\alpha(\scrT)$.

\subsection{HyperKähler structure in the polyhedral case}
The space of families $F=(F_\sigma)_{\sigma\in\scrK_4}$
of constant $\vecV$-valued  constant differential $1$-forms $\scrF(\scrT)$, is endowed with a HyperKähler
structure similar to the smooth setting.
The metric $g$ defined on the bundle $T^*M\otimes\vecV\to M$ induces a
$L^2$-metric $\cG$ on $\scrF(\scrT)$, similar to the smooth case (cf.
Formula~\eqref{eq:metr}) and given by
$$
\cG(F,G)=\sum_{\sigma\in\scrK_4}\int_\sigma
g(F_\sigma,G_\sigma)\vol_M
$$
for every $F=(F_\sigma)$ and $G=(G_\sigma)\in\scrF(\scrT)$. Following the
construction in the smooth case (cf. \S\ref{sec:hkmap}),
the almost complex structures $I, J$ and
$K$ on $M$  induce three almost complex structures $\cI, \cJ$ and $\cK$ on
$\scrF(\scrT)$,
defined by
\begin{align*}
	(\cI F)_\sigma&= -F_\sigma\circ I,\\
	(\cJ F)_\sigma&= -F_\sigma\circ J,\\
	(\cK F)_\sigma&= -F_\sigma\circ K,
\end{align*}
compatible with the metric $\cG$, with corresponding Kähler
forms 
$$\Omega_I=\cG(\cI\cdot,\cdot),\quad \Omega_J=\cG(\cJ\cdot,\cdot) 
\mbox{ and }
\Omega_K=\cG(\cK\cdot,\cdot).
$$
In conclusion, the moduli space $\scrF(\scrT)$ is endowed with a
natural hyperKähler structure $(\scrF(\scrT),\cG,\cI,\cJ,\cK)$.

\subsection{Torus action in the polyhedral case}
 The group $\TT(\scrT)$ 
is defined as the space of
functions $\lambda:\scrK_4\to S^1$  on the set of facets of the
triangulation. Equivalently, $\lambda$ can be understood as a collection
$(\lambda_\sigma)_\sigma\in\scrK_4$ of
constant functions $\lambda_\sigma: \sigma\to S^1$ on each facet of the
triangulation.
In fact $\TT(\scrT)$ is isomorphic as a group to the real torus $(S^1)^m$, where the
dimension $m$ is the number of facets of
the triangulation.

The Lie algebra $\LieTT(\scrT)$ of the group $\TT(\scrT)$, is  the
vector space of real valued functions $\zeta:\scrK_4\to \RR$ defined on the set of
of facets of the triangulation. As in the case of the group, such a
function $\zeta$ can be thought of as collection
$(\zeta_\sigma)_{\sigma\in\scrK_4}$ of constant real valued functions
$\zeta_\sigma:\sigma\to\RR$ on each facet of the triangulation.
The Lie algebra is endowed with a $L^2$ inner product, defined by
$$
\ipp{\zeta,\zeta'}= \sum_{\sigma\in\scrK_4}= \int_\sigma
\zeta_\sigma\zeta'_\sigma \vol_M
$$
and the exponential map $\exp:\LieTT(\scrT)\to\TT(\scrT)$ is defined by
$\exp(\zeta)= \lambda$, with
$\lambda(\sigma)=e^{i\zeta(\sigma)}$ for each $\sigma\in\scrK_4$.

The fact that $\vecV$ is a complex vector space induces a multiplicative
action by complex valued functions on the space of $\vecV$-valued
differential forms.
In particular, the group $\TT(\scrT)$ acts by complex multiplication on the
moduli space $\scrF(\scrT)$.
 More precisely, given $\lambda =(\lambda_\sigma)\in\TT(\scrT)$ and
$F=(F_\sigma)\in \TT(\scrT)$, we define $\lambda\cdot F\in\scrF(\scrT)$ by
$$
(\lambda \cdot F)_\sigma= \lambda_\sigma F_\sigma
$$
for every $\sigma\in\scrK_4$,
where the product in the RHS is the multiplication by the complex valued
function $\lambda_\sigma$.

\begin{prop}
	\label{prop:A1}
	The hyperKähler structure $(\scrF(\scrT),\cG,\cI,\cJ,\cK)$ is invariant
	under the action of the torus $\TT(\scrT)$.
	\end{prop}
\begin{proof}
	The proof is formally the same as in the smooth setting (cf.
	Proposition~\ref{prop:presk}).
\end{proof}

As in the smooth case, we have three involutions $\scrR_\bullet$ of
$\scrF(\scrT)$,
 given by
$$
\scrR_I F = i\cI F,\quad 
\scrR_J F = i\cJ F,\quad 
\scrR_K F = i\cK F
$$
with three maps
$\mu_\bullet:\scrF(\scrT)\to \LieTT(\scrT)$
defined by
$$
\mu_\bullet(F)(\sigma) = - \frac 12\ip{\scrR_\bullet F_\sigma, F_\sigma}
$$
Following the analogy (cf. \S\ref{sec:ham}), the action of $\TT(\scrT)$  on $\scrF(\scrT)$ is  Hamiltonian
with respect to each Kähler form $\Omega_\bullet$:
\begin{theo}
	\label{theo:A2}
	The group  action of $\TT(\scrT)$ on $\scrF(\scrT)$ is Hamiltonian with respect to
	each Kähler forms $\Omega_\bullet$. The moment maps are
	given by $\mu_\bullet$ for $\bullet=I,J,K$. More precisely,
	$\mu_\bullet$ is
	$\TT(\scrT)$-invariant and for every $\zeta \in\LieTT(\scrT)$
	$$
	D\ipp{\mu_\bullet,\zeta} = -\iota_{X_\zeta}\Omega_\bullet
	$$
	where $X_\zeta$ is the vector field on $\scrF(\scrT)$ associated to
	the infinitesimal action of $\zeta$. In this sense, $\mu_\bullet$
	is a moment map for the group action with respect to the Kähler form
	$\Omega_\bullet$.
\end{theo}
\begin{proof}
	The proof of Theorem~\ref{theo:ham}, in the smooth setting, can be
	repeated
	formally in the polyhedral context.
\end{proof}

\subsection{Symplectic density in the polyhedral case}
Formula~\ref{eq:density} was used in the smooth setting to express the
moment map in terms of symplectic density. However the
formula is a consequence of algebraic pointwise computation. It follows
that a similar formula holds  in
the polyhedral setting:
for $F=(F_\sigma)\in\scrF(\scrT)$, we deduce that $\mu_\bullet
(F)\in\LieTT(\scrT)$ is given by the family
$$
\mu_\bullet(F)(\sigma)= - \frac{F_\sigma^*\omega_V\wedge
\hat\omega_\bullet}{\vol_M}
$$
The three moment maps can be gathered in a hyperKähler moment map
$$
\mub=(\mu_I,\mu_J,\mu_K)
$$
and using the isomorphism $\xi$ (cf. Formula~\ref{eq:xi}), we have
$$
\xi\circ\mub(F)(\sigma)= \sqrt
2(F_\sigma^*\omega_V)^-.
$$
Hence the moment can be interpreted as a map 
$\tilde\mub$ on $\scrF$,
with values in the space of families of constant anti-selfdual forms on
each facet of the triangulation and given by
$\tilde\mub=\xi\circ\mub$.

\subsection{Polyhedral Hodge theory}
Stokes theorem extends to Whitney form:  in the case of our
torus $M$ endowed with a Euclidean triangulation $\scrT$, let $\beta
=(\beta_\sigma)\in\Omega^3_w(M,\scrT)$ be a   Whitney form. For every
$4$-simplex $\sigma\in\scrK_4$, we have
$$
\int_{\del\sigma} j^*\beta_\sigma =\int_\sigma d\beta_\sigma
$$
where $j:\del\sigma\to\sigma$ is the canonical inclusion,
by the usual Stokes theorem on a simplex. By the Whitney condition, the boundary
contributions cancel out and we have
$$
\int_M d\beta:=\sum_{\sigma\in\scrK_4}\int_\sigma d\beta_\sigma =0.
$$
We can define the exterior product for families $\beta=(\beta_\sigma)$ and
$\gamma=(\gamma_\sigma)$ by
$\beta\wedge\gamma=(\beta_\sigma\wedge\gamma_\sigma)$. By the usual
Leibnitz formula
$$
d(\beta \wedge\gamma)=d\beta\wedge\gamma +(-1)^p\beta\wedge d\gamma
$$
if $\beta$ is a $p$-form.
If $\beta$ is a Whitney $p$-form  and $\gamma$ is a Whitney $q$-form with
$p+q=3$, we deduce from the Stokes theorem
$$
(-1)^p\int_M d\beta\wedge\gamma +\int_M \beta\wedge d\gamma=0.
$$

We define a Hodge $\star$-operator on families on $\Omega^p(M,\scrT)$
by $\star\beta = (\star \beta_\sigma)$, where $\star$ is the usual
Hodge operator.
If $\beta\in \Omega^2(M,\scrT)$ satisfies $\star\beta=\beta$, we say that $\beta$ is selfdual $2$-form,
as in the smooth setting.
Then
\begin{align*}
\int_M \beta\wedge
	\beta &=
	\sum_{\sigma\in\scrK_4}\int_\sigma\beta_\sigma\wedge\beta_\sigma\\
	&=
\sum_{\sigma\in\scrK_4}\int \ip{\beta_\sigma,\star\beta_\sigma}\vol_M\\
	&=
\sum_{\sigma\in\scrK_4}\int \ip{\beta_\sigma,\beta_\sigma}\vol_M 
	=
 \|\beta\|_{L^2} ^2
\end{align*}
so that $\int_M \beta\wedge \beta >0$ or $\beta =0$.
We deduce the following result, analogue to the smooth case:
\begin{lemma}
	\label{lemma:sdpol}
The only selfdual exact Whitney $2$-form is the zero form.
\end{lemma}
\begin{proof}
Is $\beta$ is an exact Whitney $2$-form, there exists a Whitney $1$-form
	$\gamma$ such that $\beta = d\gamma$. By Stokes theorem, we deduce
	that $\int_M\beta\wedge\beta = \int _M d(\gamma\wedge \beta)=0$,
	hence $\beta=0$.
\end{proof}

\begin{cor}
	\label{cor:B}
	Let $f\in\scrM(\scrT)$ be a polyhedral map of the torus such that 
	$f^*[\omega_M]=[\omega_M]$. Then the following conditions are
	equivalent
	\begin{enumerate}
		\item $\mub\circ\scrD f = 0$.
		\item $f$ is symplectic, in the polyhedral sense.
	\end{enumerate}
\end{cor}
\begin{proof}
	The proof is similar to the one in the smooth setting (cf.
	Proposition~\ref{prop:symplecto}). Assuming that
	$\mub\circ\scrD f=0$, we have the equivalent equation
	$\tilde\mub(\scrD f)=0$ which means that the familly 
	$f^*\omega_M|_\sigma$ is selfdual on every $4$-simplex
	$\sigma\in\scrK_4$.
	Hence $f^*\omega_M$ is a selfdual Whitney $2$-form. By assumption
	$f^*\omega_M-\omega_M$ is an exact Whitney $2$-form and by Lemma~\ref{lemma:sdpol}, we
	conclude that $f^*\omega_M= \omega_M$.

	Conversely, if $f$ is symplectic in the polyhedral sense, then $f^*\omega_M=\omega_M$
	along every $4$-simplex, which shows that $f^*\omega_M$ is
	selfdual, hence $\tilde\mub(\scrD f)=0$.
\end{proof}

We conclude this section with a proof of one of our main theorems:
\begin{proof}[Proof of Theorem~\ref{theo:AP}]
	The analogue of Theorem~\ref{theo:A} is given by
	Theorem~\ref{theo:A2} for the polyhedral case.
	The analogue of Theorem~\ref{theo:eqmu} is given by
	Corollary~\ref{cor:B}.
\end{proof}

\begin{rmk}
	\label{rmk:dl}
Let  $\scrT_i=(M,\scrK_i,\phi_i)$ be two Euclidean triangulations of the
	torus for
	$i=1,2$, such that $\scrT_2$ is a refinment of $\scrT_1$.
	By definition of polyhedral  maps, there is a canonical injection
	$\scrM(\scrT_1)\hookrightarrow\scrM(\scrT_2)$. Hence the family of
	moduli spaces $\scrM(\scrT)$ parametrized by triangulations and ordered
	by refinement is an \emph{inductive family}. The direct limit 
	$$\scrM^{PL}=\lim_{\stackrel{\to}{\scrT}} \scrM(\scrT)$$
	is a topological space endowed with the final topology.
	As a set
	$\scrM^{PL}$ is the space of piecewise linear maps of the torus, by
	definition. 
	We can consider the subspace  
	of piecewise linear symplectic maps, obtained as the direct limit
	of the spaces of polyhedral symplectic maps
	$$\Symp^{PL}(M,\omega_M)=\lim_{\stackrel{\to}{\scrT}}
	\Symp(M,\omega_M,\scrT).$$
The moduli
	spaces $\scrF(\scrT)$ and the groups $\TT(\scrT)$ also have direct
	limits $\scrF^{PL}$ and $\TT^{PL}$.
	Formally, we have an interesting framework where 
	our moment map constructions extend to the
	piecewise linear  setting without any reference to the choice of a
	particular triangulation.
\end{rmk}

\subsection{Polyhedral modified moment map flow}
Given $\alpha\in H^1(M,\vecV)\setminus 0$, we defined at \S\ref{sec:diffpol} the subspace
$\scrF_\alpha(\scrT)$ as the subspace of Whitney forms of $\scrF(\scrT)$
with cohomology class in $\RR\alpha$.
The Euclidean metric $\cG$ induces an orthogonal projector, which is the analogue of
the Hodge projector in the smooth setting, denoted
$$
\Pi_\alpha :\scrF(\scrT)\to\scrF_\alpha(\scrT).
$$
We consider the energy of the moment map also have a polyhedral analogue
defined by  the functional 
$$
\phi:\scrF(\scrT)\to\RR,
$$
given by
$$
\phi(F)=  \frac 12
\|\mub(F)\|^2_{L^2}.
$$
The gradient of $\phi$ on $\scrF(\scrT)$ is denoted $\nabla\phi$, whereas the
gradient of the restrition $\phi:\scrF_\alpha(\scrT)\to\RR$, denoted
$\nabla^\alpha\phi$, is related to $\nabla\phi$ via the orthogonal
projection
$$
\nabla^\alpha\phi(F)=\Pi_\alpha\nabla\phi(F)
$$
for every $F\in\scrF_\alpha(\scrT)$.
We condider the downward gradient flow of $\phi$ along
$\scrF_\alpha(\scrT)$:
\begin{equation}
	\label{eq:polmmf}
\frac{\del F}{\del t}=-\Pi_\alpha\nabla\phi(F).
\end{equation}
The evolution equation~\eqref{eq:polmmf} is called the \emph{polyhedral
modified moment map flow}. This ordinary
differential equation is an analogue of the modified moment map flow
defined in the smooth setting. However, the Cauchy-Lipschitz theorem for
ODE insures
the local existence and uniqueness of solutions of
equation~\eqref{eq:polmmf}.

All the identities computed in the smooth setting have formal analogues in
the polyhedral case. The most useful properties are summarized in
Proposition~\ref{prop:formpol}. Before stating the resutls, we introduce 
the vector fields $W_\bullet$, defined on $\scrF(\scrT)$ by the formula
$$
	W_\bullet(F)=-\mu_\bullet(F)\scrR_\bullet F
$$
for $\bullet=I,J,K$. The proof 
of Lemma~\ref{lemma:formphi} is immediately adapted to the polyhedral
setting and gives the following result:
\begin{prop}
	\label{prop:formpol}
	The gradient of $\phi$ on $\scrF(\scrT)$ is given by the
	formula
	$$
	\nabla\phi(F)=\sum_\bullet W_\bullet(F).
	$$
	In addition
	$$
	\ipp{\nabla\phi(F),F}=4\phi(F).
	$$

	For every solution $F_t$ of the modified moment map flow, for $t$
	in some interval, we have
	$$
	\frac{d}{dt}\|F_t\|^2_{L^2}=-8\phi(F_t).
	$$
	In particular, the $L^2$-norm is non increasing along the flow.
\end{prop}
\begin{cor}
	\label{cor:limit}
	The critical points of $\phi:\scrF_\alpha(\scrT)\to \RR$ are the
	fixed points of the polyhedral modified moment map flow and
	agree with the
	vanishing locus of $\phi$ in $\scrF_\alpha(\scrT)$.

	Furthermore, a flow line $F_t$ defined for
	$t$ in some interval $[t_0,t_1)$ admits a unique extension 
	for $t\in [t_0,+\infty)$, as a
	solution of the modified moment map flow.
	Finally, the limiting orbits of the flow are non empty compact
	sets.
\end{cor}
\begin{proof}
	If $F\in\scrF_\alpha$, then
	$$\ipp{\nabla\phi(F),F}
	=\ipp{\Pi_\alpha\nabla\phi(F),F}=\ipp{\nabla^\alpha\phi(F),F}.$$
	By Proposition~\ref{prop:formpol}, we deduce that 
	$\ipp{\nabla^\alpha\phi(F),F}=4\phi(F)$. This shows that a critical
	point of the restriction $\phi:\scrF_\alpha\to \RR$ is a zero of
	$\phi$. Conversely  $\phi(F)=0$ implies $\mu_\bullet(F)=0$ hence $W_\bullet(F)=0$, by
	definition.
	Proposition~\ref{prop:formpol} then shows that $F$ is a critical
	point of $\phi$.

	The monotonicity of the $L^2$ norm along flow lines stated in
	Proposition~\ref{prop:formpol} insures that the flow cannot blowup:
 if $F_t$ is
	a solution of the flow defined for $t\in [t_0,t_1)$, then
	$\|F_t\|_{L^2}\leq \|F_{t_0}\|_{L^2}$. 
	The existence of the flow for $t\in[t_0,+\infty)$  follows from a
	classical result of ODE theory and the flow line is bounded by
	$\|F_{t_0}\|_{L^2}$. The limiting orbit of $F_t$ is closed, by definition,
	and the compactness property follows from the boundedness. 
	The non emptyness of the limiting orbit of $F_t$ is clear, as we
	can always extract a converging subsequence of $F_t$, by the
	Bolzano-Weirstrass theorem.
\end{proof}

\subsection{A Duistermaat theorem}
In this section, we will often omit the reference to the triangulation
$\scrT$ and denote $\scrF(\scrT)$ by $\scrF$, for simplicity. The moment
map and the modified moment map flow are always in reference to the
polyhedral framework, for some fixed triangulation $\scrT$.
By continuity of $\phi$, the  sets
\begin{equation}
	\label{eq:defve}
V_{\epsilon} = \{F\in\scrF_\alpha,\; \phi(F)<\epsilon\}
\end{equation}
are open neighborhood of the vanishing locus of $\phi$ in $\scrF_\alpha$, for every $\epsilon>0$.

The function $\phi$ is non
increasing along the modified moment map flow, by definition of a downward
gradient flow. 
This implies that the set $V_\epsilon$ is invariant under the flow:
\begin{lemma}
	\label{lemma:trap0}
Let $F_t$ be a solution of the modified moment map flow, defined for $t$ in
	some interval $I$. If $F_{t}\in V_\epsilon$ for $t\in I$, then
	$F_{t'}\in V_\epsilon$ for every $t'\in I$ with $t'\geq t$.  
\end{lemma}

We consider the closed ball or radius $R>0$
centered at the origin in $\scrF_\alpha$
$$
\bar B_R = \{ F\in \scrF_\alpha,\; \|F\|\leq R\}
$$
and  $C_\epsilon$ be the compact subset of $\bar B_R$ defined by
$$
C_\epsilon = \bar B_R\setminus V_\epsilon.
$$
We introduce the
$$
\delta = \inf_{F\in C_\epsilon} N(F)>0
$$
where $N:\scrF_\alpha\to \RR$ is the continuous function defined by
\begin{equation}
	\label{eq:defN}
N(F)=\|\nabla^\alpha\phi(F)\|^2.
\end{equation}
By continuity of $N$ and compactness of $C_\epsilon$, 
the infimum $\delta$ is a minimum, with $\delta=N(F_{min})$ for some
$F_{min}\in
C_\epsilon$.
By Corollary~\ref{cor:limit}, $\phi$  and $N$ have the same vanishing locus
on $\scrF_\alpha$. Hence $N$ does not vanish on $C_\epsilon$ and it follows
that $\delta=N(F_{min})>0$.
Notice that $\delta=\delta(\epsilon, R)$ depends on the choice of
$\epsilon$ and $R$.

\begin{lemma}
	\label{lemma:trap}
	Let $F_t$ be a solution of the modified moment map flow, defined
	for $t\in[0,+\infty)$. If $F_0\in \bar B_R$, then $F_t\in \bar B_R\cap V_\epsilon$ 
	 for every $t\geq
	\frac{\phi(F_0)}{\delta}$.
\end{lemma}
\begin{proof}
We consider a solution $F_t$ of the modified moment map flow, with the
	assumptions of the lemma.
	 By Proposition~\ref{prop:formpol}
	the $L^2$-norm is non
	increasing along the flow, hence $F_t\in \bar B_R$ for every $t\geq
	0$.
	By Lemma~\ref{lemma:trap0}, $V_{\epsilon}$ is stable under the
	flow. In conclusion 
	$$\tilde
	V_{\epsilon}=\bar B_R\cap V_\epsilon$$
	is also stable under the flow. Notice that  we have a
	partition of the ball $\bar B_R=\tilde V_\epsilon\cup C_\epsilon$.
	As a solution of a downward gradient flow, the flow line satisfies
	the ODE
$$
	\frac {d\phi(F_t)}{dt}=-
	\|\nabla^\alpha\phi(F_t)\|_{L^2}^2=-N(F_t).
$$
	In particular, if $F_t \in C_\epsilon$, we have
	\begin{equation}
		\label{eq:diffineq}
	\frac {d\phi(F_t)}{dt} \leq -\delta.
	\end{equation}
	If $F_t\in C_\epsilon$ for every $t\in[0,T]$, we
	obtain by integration of the differential
	inequality~\eqref{eq:diffineq}
	\begin{equation}
		\label{eq:diffineqc}
	\phi(F_T) \leq \phi(F_0) -T\delta.
	\end{equation}
	Using the fact that $\phi(F_T)\geq 0$, we deduce that $T\leq
	T_0=\frac{\phi(F_0)}{\delta}$.
	Now there are two possibilities: 
	\begin{itemize}
		\item $F_t\in \tilde V_{\epsilon}$ for
	some time $t_0\leq T_0$.
 We argued above that 
			$\tilde V_{\epsilon}$ is stable under the flow.
			 Hence $F_t \in \tilde V_{\epsilon}$ for every $t\geq t_0$, in
	particular, for every $t\geq T_0$.
\item	
	If, on the contrary, $F_t \in C_\epsilon$ for every $t\in
			[0,T_0]$, we deduce that $\phi(F_{T_0})\leq 0$ from
			Inequality~\eqref{eq:diffineqc}. Since $\phi$ is
			non negative, we
	have $\phi(F_{T_0})=0$. This is a contradiction since $F_{T_0}\in
			C_\epsilon$, hence $\phi(F_{T_0})>0$, by assumption.
	\end{itemize}
	We conclude that $F_t$ hits
	$\tilde V_{\epsilon}$ for some time $t_0\leq T_0$. Since $\tilde
	V_\epsilon$ is stable under the flow, we have
	$F_t\in \tilde V_{\epsilon}$ for every $t\geq t_0$,
	in particular, for every $t\geq T_0$, which proves the lemma.
	\end{proof}

\begin{cor}
	\label{cor:limphi}
For every solution of the modified moment map flow $F_t$, defined for
	$t\in [0,+\infty)$, we have 
	$$\lim_{t\to+\infty} \phi(F_t)=0.
	$$
	In particular every limiting orbit of the modified moment map flow
	is contained in the zero locus of $\phi$.
\end{cor}
\begin{proof}
We choose $R=\|F_0\|$ in the above constructions. Then for every
	$\epsilon>0$, there exists $\delta$ such that for every $t\geq
	\frac {\phi(F_0)}\delta$, we have $\phi(F_t)<\epsilon$, by
	Lemma~\ref{lemma:trap}. This shows that the limit of $\phi(F_t)$ is
	zero. The statement about limiting orbits is immediate, by
	continuity of $\phi$.
\end{proof}

Eventually, we prove our main result for the modified moment map flow:
\begin{theo}
	\label{theo:retract}
	Every solution $F_t\in\scrF_\alpha(\scrT)$ of the modified moment map flow, defined for
	$t$ in some interval $[0,+\infty)$, admits a limit
	$$
	F_\infty=\lim_{t\to+\infty} F_t\quad \in \scrF_\alpha(\scrT)
	$$
	such that~$\phi(F_\infty)=0$.
In particular, we have an  extended flow 
	$$
	\Theta:[0,+\infty]\times \scrF_\alpha(\scrT) \to
	\scrF_\alpha(\scrT)
	$$
	 defined as follows: for $F\in\scrF_\alpha(\scrT)$ and
	$t\in [0,+\infty)$,
	we put
	\begin{itemize}
		\item
	$\Theta(t,F)=F_t$, 
	where $F_t$ is the solution of the modified moment map flow with initial
	condition $F_0=F$.
\item If $t=+\infty$, 
	we put	$\displaystyle \Theta(+\infty,F)=F_\infty=\lim_{t\to+\infty} F_t$.
	\end{itemize}
	The extended flow $\Theta$ is a continuous map and defines a continuous retraction of
	$\scrF_\alpha(\scrT)$
	onto the zero locus of $\phi$ in $\scrF_\alpha(\scrT)$.
\end{theo}
The proof of Theorem~\ref{theo:retract} rests on the \L{}ojasiewicz estimate (cf. \cite[Lemma 2.2]{L05} or
Proposition 1 p. 67 of \cite{Lo}, or Proposition 6.8 in \cite{BM}):
\begin{lemma}[\L{}ojasiewicz gradient inequality]
	\label{lemma:loj}
	Let $\phi$ be a real analytic function on an open set $W\subset \RR^k$ and
	$x\in W$,
a critival point of $\phi$, such that $\phi(x)=0$. Then, there exists a neighborhood $U$ of
$x$ and 
some constants $c>0$ and $\nu\in (0,1)$, such that
$$
\|\nabla \phi(y)\|\geq c|\phi(y)|^\nu
$$
for every $y\in W$,
where $\|\cdot \|$ is an Euclidean norm on $\RR^k$ and $\nabla$
is the gradient with respect to the Euclidean inner product.
\end{lemma}
\begin{proof}[Proof of Theorem~\ref{theo:retract}]
	The function $\phi: \scrF_\alpha(\scrT) \to\RR$ is polynomial
	function of degree~$4$. In
	particular, $\phi$ is analytic and~Lemma~\ref{lemma:loj} applies.
	As before, we consider the closed ball $\bar B_R\subset
	\scrF_\alpha$ or radius $R>0$ centered at the origin. The vanishing locus
	$K=\phi^{-1}(0)\cap \bar B_0$
	is
	closed in $\bar B_0$, hence compact. For each point $F\in K$, we
	choose an open neighborhood $U_F$ of $F$
	in $\scrF_\alpha$ and constants $c_F>0$ and $\nu_F\in(0,1)$ such
	that the \L{}ojasiewicz
	estimate holds.

	The familly of  sets $U_F$ for $F\in K$ is an open cover $K$.
	 By compactness, we can extract a finite cover $U_1,\cdots,
	U_k$ and we define
	$$
	U = \bigcup U_i,\quad \nu = \max \nu_i \quad \mbox{ and }
	\quad c=\min c_i.
	$$
In particular, for every $F\in U$, we have
	\begin{equation}
		\label{eq:loja}
	\|\nabla^\alpha \phi(F)\|\geq c\phi(F)^\nu
	\end{equation}
	provided $\phi(F)\leq 1$. In other words,
	Inequality~\ref{eq:loja} holds on $U\cap V_1$, where $V_1$ is
	defined by~\eqref{eq:defve}.

	If $\bar B_R\subset  V_1\cap U$, we have proved that
	\eqref{eq:loja} holds on the ball.
Suppose this is not the case, in other words that $\bar
	B_R \setminus (U\cap V_1)$ is not empty. Then, the function $\phi$ admits a minimum $\epsilon_0>0$ 
	on the compact set  $\bar B_R\setminus (U\cap V_1)$. In particular
	$V_{\epsilon_0}$ does not intersect this compact set, hence
$$
	V_{\epsilon_0}\cap \bar B_R \subset U\cap V_1\cap  \bar B_R.
$$
	In conclusion,  an open
	neighborhood 
	$$\tilde V_{\epsilon_0}=V_{\epsilon_0}\cap \bar B_0
	$$
	of $K$, with respect to the induced topology on the ball, 
	such that the \L{}ojasiewicz inequality \eqref{eq:loja} holds for
	every $F\in\tilde V_{\epsilon_0}$.

	Notice that $\tilde V_{\epsilon_0}$ cannot cover the ball $\bar
	B_R$. Indeed, $\tilde V_{\epsilon_0}$ is a subset of $U\cap V_1$
	which does not cover the ball.
	Hence $C_{\epsilon_0}=\bar B_R\setminus
	V_{\epsilon_0}$ is a non
	empty compact subset of $\bar B_R$.
	We consider the infimum $\delta_0$ of the function $N$ (cf.
	Formula~\eqref{eq:defN}) on $C_{\epsilon_0}$ as before and we have
	$\delta_0>0$.

For $F\in\scrF_\alpha$ and we define $R=\|F\|+1$. In particular $\bar B_R$
is a neighbordhood of $F$ in $\scrF_\alpha$.
For every  $G\in \bar B_R$, we denote by $G_t$ the
corresponding flow line, defined for $t\in[0,+\infty)$, 
with $G_0=G$. By Lemma~\ref{lemma:trap} 
$$G_t\in \tilde V_{\epsilon_0} \quad \forall t\geq \frac
M{\delta_0},
$$
where $M$ is the maximum of $\phi$ on the ball $\bar B_R$.
From the flow  equation, we deduce
$$
	\frac{d}{d t}\phi(G_t)= -\|\nabla^\alpha\phi(G_t)\|^2,
	$$
	hence
$$
	-\frac d{d t} \phi(G_t)^{1-\nu} = (1-\nu)
	\phi(G_t)^{-\nu} \| \nabla^\alpha \phi(G_t) \|^2.
$$
For every $t\geq\frac M{\delta_0}$, we have  by 
	\eqref{eq:loja}
$$
	-\frac d{d t} \phi(G_t)^{1-\nu} \geq
	c(1-\nu)\|\nabla^\alpha \phi(G_t)\|.
$$
Hence, for every $t_1>t_0\geq \frac M{\delta_0}$, we obtain by integrating the above
inequality
\begin{align}
	\nonumber
	\phi(G_{t_0})^{1-\nu}
	- \phi(G_{t_1})^{1-\nu} &= -\int _{t_0}^{t_1} \frac
	d{d t} \phi(G_t)^{1-\nu} dt \\
	\label{eq:lojint1}
	\phi(G_{t_0})^{1-\nu}
	- \phi(G_{t_1})^{1-\nu} &\geq c(1-\nu)  \int_{t_0}
	^{t_1} \|\nabla\phi(G_t)\| dt.
\end{align}
	By the Jensen inequality, we have
	\begin{align}
			\nonumber  d(G_{t_0}, G_{t_1}) &\leq \int_{t_0}
			^{t_1}\left \| \frac d{d
	t} G_t \right \| dt \\
		\label{eq:lojint2}
		&\leq
\int_{t_0} ^{t_1}\| \nabla^\alpha\phi( G_t)\| dt 
   \end{align}
   where $d(F,G)=\|F-G\|$ is the distance induced by the Euclidean norm 
   on~$\scrF_\alpha$. By Inequalities~\eqref{eq:lojint1}
   and~\eqref{eq:lojint2}, we have
\begin{equation}
		\label{eq:loj2} d(G_{t_0}, G_{t_1}) 
		\leq c_2 \left (\phi(G_{t_0})^{1-\nu} -
		\phi(G_{t_1})^{1-\nu} \right),
	\end{equation}
	where $c_2= \frac 1{c(1-\nu)}$.

By Corollary~\ref{cor:limphi}, $\displaystyle 
\lim_{t\to+\infty}\phi(G_t)=0$, hence $\displaystyle \lim_{t\to +\infty}
\phi(G_t)^{1-\nu}=0$. This implies that the function $t\mapsto
		\phi(G_t)^{1-\nu}$ satisfies the Cauchy
		criterion. By
		Inequality~\eqref{eq:loj2}, we deduce that for every
		$\epsilon >0$, there exists $t_0\geq \frac M{\delta_0}$,
		such that for every $t_1>t_0$, we have
		$d(G_{t_0}, G_{t_1})\leq \epsilon$. We see that $t\mapsto
		G_t$ satifies the Cauchy criterions as well. 
		By completness of $\scrF_\alpha$, 
		the flow line  $G_t$ converges as $t$ goes to infinity.
		In particular $F_\infty =\lim_{t\to+\infty}\Theta(t,F)$
		exists for every $F\in\scrF_\alpha$ and we have proved that
		the extended
		flow is well defined on 
		 the entire space $[0,+\infty]\times
		\scrF_\alpha$.

By the classical ODE theory, the flow $\Theta$ is continuous on
$[0,+\infty)\times \scrF_\alpha$. We are going to show that 
 $\Theta$ is also continuous at every point of
 $\{+\infty\}\times \scrF_\alpha$,
  hence everywhere. 
First observe that passing to the limit in
		Inequality~\eqref{eq:loj2} provides a control between
		$G_t$ and its limit $G_\infty$. More precisely, let $G_t$
		be any solution of the modified moment map flow defined for
		$t\in[0,+\infty)$, with $G_0\in \bar B_R$ as before.
		Using the factt that $\displaystyle
		\lim_{t_1\to\infty}\phi(G_{t_1})=0$, we deduce that
	\begin{equation}
		\label{eq:loj3} d(G_t, G_{\infty})  \leq c_2
		\phi(G_{t})^{1-\nu} ,
\end{equation}
		  for every $t\geq\frac M{\delta_0}$. This implies 
\begin{align}
	\nonumber
	d(G_\infty,F_\infty)&\leq d(G_\infty,G_t)+d(G_t,F_t)+d(F_t,F_\infty)\\
	\label{eq:contrinterm}
	&\leq  d(G_t,F_t)+ c_2 \phi(F_t)^{1-\nu}+
	c_2\phi(G_t)^{1-\nu}.
\end{align}
for every $t\geq \frac M{\delta_0}$.

Given $\epsilon >0$, we denote by $\delta>0$ the constant provided by
Lemma \ref{lemma:trap} and we put $T=  M \max(\frac 1\delta,\frac 1{\delta_0})$.
Then for every $t\geq T$, we have $\phi(G_t)\leq \epsilon$ for every
flow line $G_t$ with $G_0=G\in\bar B_R$. By
Inequality~\ref{eq:contrinterm}, we have
$$
	d(G_\infty,F_\infty)\leq d(G_T,F_T)+2c_2\epsilon^{1-\nu}.
$$
Using the continuity of $\Theta$ at $(T,F)\in [0,+\infty)\times
\scrF_\alpha$, we deduce that there exists 
a neighbordhood $W$ of
$F$ in $\bar B_R$, such that for every $G\in W$, we have
$d(\Theta(T,F),\Theta(T,G))\leq \epsilon$.
In conclusion  
$$
d(\Theta(+\infty,G),\Theta(+\infty,F))\leq	\epsilon +2c_2
\epsilon^{1-\nu}
$$
for every $F\in W$, which shows that $G\mapsto\Theta(+\infty,G)$ is
continuous at $F$.
Furthermore, for every $t\geq T$ and $G\in W$, 
\begin{align*}
	d(G_t,F_\infty)&\leq  d(G_t,G_\infty)+d(G_\infty,F_\infty)\\
	&\leq
c_2(\phi(G_t)^{1-\nu}) + d(G_\infty,F_\infty)\\
	&\leq
\epsilon+3c_2\epsilon^{1-\nu},
\end{align*}
which shows that $\Theta$ is continuous at $(+\infty,F)$.
\end{proof}

	\subsection{Flow and regular locus}
\label{sec:regular}
The vanishing locus of restriction of the hyperKähler moment map
$\mub:\scrF_\alpha(\scrT)\to\LieTT^3(\scrT)$
agrees with the vanishing locus of $\phi$ in $\scrF_\alpha(\scrT)$.
This set is real algebraic, but little is understood about the regularity
of the vanishing locus.
More precisely, we introduce the following definition:
\begin{dfn}
	\label{dfn:regular}
	A point $F$ of the vanishing locus of $\phi:
	\scrF_\alpha(\scrT)\to\RR$ is called 
	regular, if the differential of the hyperKähler moment map
	$\mub:\scrF_\alpha(\scrT)\to\LieTT(\scrT)^3$ has constant 
	rank in neighborhood of $F$.
\end{dfn}

\begin{rmks}
	\begin{enumerate}
		\item 
			The vanishing set of $\phi:\scrF_\alpha \to
			\LieTT(\scrT)^3$ 
has the structure of a smooth manifold in a neighborhood 
			of a regular point.  
		\item
	We do not have a good understanding of the regularity consition.
	 In fact
	it is not clear whether there exists any regular point. It seems
			sensible to expect that
	regular points should be \emph{generic}, perharps up to passing to
			a suitable
	refinement of the triangulation $\scrT$. 
	\end{enumerate}
\end{rmks}
If the regularity condition is met at some point $F$, the polyhedral
version of the modified moment map flow is  well behaved in a
neighborhood of $F$, in the sense that the flow has  exponential convergence. 
This property is a strong incentive to explore numerical versions
of the flow,  that should provide effective examples of polyhedral
symplectic maps of the torus $M$ \cite{JR2}.
\begin{theo}
	\label{theo:stable}
	Let $F\in\scrF_\alpha$ be a regular point of the vanishing
	locus of $\phi$. Then, there exists an open neighborhood $V$ of $F$ in
	$\scrF_\alpha(\scrT)$, such that
	\begin{enumerate}
	\item The vanishing locus of $\phi$ is $V$ is a manifold.
	\item The set $V$ in invariant under the flow $G\mapsto
		\Theta(t,G)$ for $t\geq 0$.
	\item For every $G\in V$, te flow line $t\mapsto \Theta(t,G)$
		converges exponentially fast as $t\to+\infty$
		toward a zero of $\phi$ in $V$.
	\end{enumerate}
\end{theo}
\begin{proof}
	We choose first an open neighborhood $V$ of $F$ in
	$\scrF_\alpha(\scrT)$, such that the rank of
	the differential of $\mub$ is constant on $V$.
	Then the vanishing locus of $\phi$ in $V$ is a manifold. 
	By Proposition~\ref{prop:formpol},  
	$\nabla\phi(F)=\sum W_\bullet(F)$.
	Following the formal computations, as in the smooth case, 
	we deduce the analogue of Formula~\eqref{eq:hess2}
	for the Hessian of $\phi$ at $G\in V$ with $\phi(G)=0$,  given by
$$
 D^2\phi|_G(\dot F,\dot F)=\|D\mub|_G\cdot\dot F\|^2_{L^2}
$$
Therefore, the kernel of the Hessian of $\phi$ at $G$ is the tangent space
 to the vanishing locus of $\phi$ in $\scrF_\alpha(\scrT)$ whereas the
	Hessian  is
	positive in transverse directions. 
	Therefore, the restriction $\phi:V\to \RR$ is a Morse-Bott
	function and its critical set agrees with its vanishing set in $V$.
	Furthermore, the vanishing set is stable. The theorem then follows
	from classical result of ODE in the context of Morse-Bott theory.
\end{proof}

\begin{proof}[Proof of Theorem~\ref{theo:duistermaat}]
	The first part of the theorem is given by
	Theorem~\ref{theo:retract} and the statement about exponential
	convergence is given by Theorem~\ref{theo:stable}.
\end{proof}
\subsection{Polyhedral renormalized flow}
In this section $\alpha$ is a symplectic cohomology class in
$H^1(M,\vecV)$. The image of $\Symp_\alpha(M,\scrT,\omega_M)$ by $\scrD$ is
denoted $S_\alpha(\scrT) \subset \scrF_\alpha(\scrT)$.
Recall that $\scrD$ induces a homeomorphism between $\scrM_\alpha/\vec V$
and its image in $\scrF_\alpha(\scrT)$. 
It follows that the topology of $\Symp_\alpha(M,\omega_M,\scrT)$ and
$S_\alpha(\scrT)$ are closely related. Furthermore, we have the following
lemma, relating the topology of $S_\alpha(\scrT)$ with the vanishing set of
$\phi$:
\begin{lemma}
	\label{lemma:cone}
	The real cone 
	$$
	\RR^*\cdot
	S_\alpha(\scrT)\subset\scrF_\alpha(\scrT)\setminus 0,
	$$ spanned by
	$S_\alpha(\scrT)$, 
	agrees with the vanishing
locus of $\phi$ in $\scrF_\alpha(\scrT)\setminus 0$.
\end{lemma}
\begin{proof}
	If $f \in \Symp_\alpha(M,\omega_M,\scrT)$, then  $\mub(\scrD f)=0$
	hence $\phi(\scrD f)
	=0$. Notice that $[\scrD f]=\alpha$, hence $\scrD f\neq 0$.

	Conversely, if $F\in\scrF_\alpha\setminus 0$ and $\phi( F)=0$, then
	$[F]\neq 0$. This fact is the polyhedral analogue of
	Lemma~\ref{lemma:nonvanish}.
	Indeed, $\phi(F)=0$ is equivalent to the fact that
	$F^*\omega_V|_\sigma$ is selfdual for every facet $\sigma$ of the
	triangulation. 
	If $[F]=0$, there exists a map
	$f\in\scrM_\alpha(\scrT)$ homotopic to the identity such that $F=\scrD f$.
	In particular, $F^*\omega_V=f^*\omega_M$ so that $F^*\omega_V$ is
	is an exact Whitney form. By Lemma~\ref{lemma:sdpol},
	we deduce that $F$ vanishes, which is a contradiction.
	
	In conclusion $[F]\neq 0$ and we can write
	$[F]=\kappa\alpha$, for some $\kappa\in \RR^*$.
	By rescaling, we obtain $G=\kappa^{-1}F$ with $[G]=\alpha$. Since
	$\alpha$ is integral, 
	there exists $h\in\scrM_\alpha(\scrT)$ such that $G=\scrD h$. The
	condition $\phi(F)=0$ implies $\phi(G)=0$ and $h$ is a symplectic
	polyhedral map. Finally $F=\kappa \scrD h$ and $F$ belongs to the
	cone spanned by $S_\alpha$.
\end{proof}

\begin{cor}
	\label{cor:coneretract}
	The vanishing locus of $\phi$ in $\scrF_\alpha(\scrT)\setminus 0$
	admits a continuous retraction onto $S_\alpha\cup -S_\alpha$.
The latter space is homeomorphic to the vanishing locus of
	$\phi:\SS_\alpha(\scrT)\to \RR$, where $\SS_\alpha(\scrT)$ is the
	unit sphere in $\scrF_\alpha(\scrT)$.
\end{cor}
\begin{proof}
	By Lemma~\ref{lemma:cone}, the vanishing locus of $\phi$ in
	$\scrF_\alpha\setminus 0$ is the cone $\RR^*\cdot S_\alpha$. 
	There exists a continuous retraction of $\RR^*$ onto $\{-1,1\}$
	which proves the first part of the corollary. The homeomorphism of
	the last statement is given by the restriction of the radial
	projection $\scrF_\alpha(\scrT)\setminus 0\to\SS_\alpha(\scrT)$.
\end{proof}

\begin{rmk}
	\label{rmk:unfort}
 Theorem~\ref{theo:retract}, provides a continuous retraction of
	$\scrF_\alpha(\scrT)$ onto the
	vanishing locus of $\phi$. Together with
	Corollary~\ref{cor:coneretract}, these facts contribute to the
	belief that topological insights about 
 the space of symplectic
polyhedral maps $\Symp_\alpha(M,\omega_M,\scrT)$	could be derived from 
	the modified moment map flow.
Unfortunately Lemma~\ref{lemma:cone} and Corollary~\ref{cor:coneretract}
deal with the space $\scrF_\alpha\setminus 0$ whereas the full space
$\scrF_\alpha$ is involved in Theorem~\ref{theo:retract}.
\end{rmk}

By Corollary~\ref{cor:coneretract}, the vanishing locus of
$\phi:\SS_\alpha(\scrT)\to \RR$ contains 
significant topological information
about $\Symp_\alpha(\scrT)$.
We would like to treat this function as a Morse-Bott function to obtain
topological invariants for $\Symp_\alpha(\scrT)$.

\subsection{Polyhedral solitons}
As introduced at~\S\ref{sec:solitons} in the smooth setting, we can introduce the space of
\emph{polyhedral solitons}
$F\in\scrF_\alpha(\scrT)$, which are  the solitons
of the \emph{polyhedral soliton equation}
\begin{equation}
	\label{eq:solitonpol}
		\|F\|_{L^2}^2 \nabla^\alpha \phi(F)=4\phi( F) F.
\end{equation}
The non zero solitons are denoted
$\scrS_\alpha(\scrT)$ and we have a partition 
$$
\scrS_\alpha(\scrT)= \scrS_{\alpha,np}(\scrT) \sqcup \scrS_{\alpha,p}(\scrT)
$$
into proper solitons and non proper solitons. The latter coïncide with the
vanishing locus of $\phi$ in $\scrF_\alpha(\scrT)\setminus 0$.
Finally, the \emph{polyhedral renormalized flow} is  the
downward gradient flow of the restriction $\phi:\SS_\alpha\to \RR$ given
 by the formula
\begin{equation}
		\frac{\del G}{\del s}= 4\phi(G)G
		-\nabla^\alpha\phi(G)\label{eq:RenApol}.
\end{equation}
By definition, the fixed points of the renormalized flow are the solitons
contained in $\SS_\alpha (\scrT)$. 
\begin{theo}
	\label{theo:rencompl}
	The polyhedral renormalized  flow on $\SS_\alpha(\scrT)$ is complete.
	Furthermore, the
	limiting orbits of the flow  
	are contained in the space of solitons in $\SS_\alpha(\scrT)$.
\end{theo}
\begin{proof}
This follows from  classical result of ODE theory on compact manifolds, in
	the case of a gradient flow. 
\end{proof}

Conjecture~\ref{conj:intro} is motivated by the fact that the regularity of
the critical locus of $\phi$ is not well understood. 
 This is a central question if we would like to
regard $\phi$ as a Morse-Bott function. 
In contrast with the smooth setting,  we have the following existence result for proper solitons:
\begin{lemma}
	\label{lemma:existsol}
	The set of proper solitons in $\scrF_\alpha(\scrT)$ is not empty.
\end{lemma}
\begin{proof}
There exists an affine symplectic map $f$ in $\scrM_\alpha$. We can deform
	$f$ by moving on of the values $f(\sigma_0)$ where $\sigma_0$ is a
	vertex of the triangulation $\scrT$. For generic choice of
	$f(\sigma_0)$, the
	deformations $h$ is  not symplectic. It follows that $\phi(\scrD
	h)\neq 0$ and this shows that $\phi$ does not vanish indentically
	on $\SS_\alpha$. In particular, the supremum of $\phi$ is positive.
	By compactness, there exists $G\in\SS_\alpha$ such that
	$\phi(G)=\sup_{F\in\SS_\alpha} \phi(F)>0$. Hence, $G$ is a critical
	point of $\phi:\SS_\alpha\to\RR$, in other words, a soliton, and it
	is proper.
\end{proof}
\begin{proof}[Proof of Theorem~\ref{theo:polflow}]
	The result is a restatement of Theorem~\ref{theo:rencompl}
	together with Lemma~\ref{lemma:existsol}, which shows existence of
	proper solitons. The fact that every proper soliton defines a solution of the
	modified moment map flow converging toward $F=0$ is an immediate
	consequence of Lemma~\ref{lemma:dicho}, which applies formally to
	the polyhedral setting.
\end{proof}
\vspace{10pt}
\bibliographystyle{abbrv}
\bibliography{polysymp}

\end{document}